\documentclass[11 pt,reqno]{amsart}

\usepackage[latin1]{inputenc}
\usepackage{amsmath,amsthm,amssymb,amscd,mathrsfs,mathtools,xcolor,accents}
\usepackage[shortlabels]{enumitem}
\usepackage[perpage]{footmisc}
\usepackage[normalem]{ulem}
\definecolor{darkgreen}{rgb}{0.0, 0.6, 0.13}

\usepackage{bbm} 
\usepackage{pifont} 
\usepackage{hyperref}

\newtheorem{thm}{Theorem}[section]
 \newtheorem{cor}[thm]{Corollary}
  
 \newtheorem{lem}[thm]{Lemma}
 \newtheorem{prop}[thm]{Proposition}
\newtheorem{claim}[thm]{Claim}
 \theoremstyle{definition}
 \newtheorem{df}[thm]{Definition}
 \theoremstyle{remark}
 \newtheorem{rem}[thm]{Remark}
 
 \numberwithin{equation}{section}
\def\sqw{\hbox{\rlap{\leavevmode\raise.3ex\hbox{$\sqcap$}}$
\sqcup$}}
\def\findem{\ifmmode\sqw\else{\ifhmode\unskip\fi\nobreak\hfil
\penalty50\hskip1em\null\nobreak\hfil\sqw
\parfillskip=0pt\finalhyphendemerits=0\endgraf}\fi}

\allowdisplaybreaks[4]

\begin{document}
\title[Invariant Gibbs measures and  global strong solutions for 2D NLS]{Invariant Gibbs measures and global strong solutions for nonlinear Schr\"{o}dinger equations in dimension two}
\author[Yu Deng]{Yu Deng$^1$}
\address{$^1$ Department of Mathematics, University of Sourthern California, Los Angeles,  CA 90089, USA }
\email{yudeng@usc.edu}
\thanks{$^1$Y. D. is funded in part by NSF-DMS-2246908 and Sloan Fellowship.}
\author[Andrea R. Nahmod]{Andrea R. Nahmod$^2$}
\address{$^2$ 
Department of Mathematics,  University of Massachusetts,  Amherst MA 01003}
\email{nahmod@math.umass.edu}
\thanks{$^2$ A.N. is funded in part by NSF DMS-2052740, DMS-2101381 and the Simons Foundation Collaboration Grant on Wave Turbulence (Nahmod's Award ID 651469)}
\author[Haitian Yue]{Haitian Yue$^3$}
\address{$^3$Institute of Mathematical Sciences, ShanghaiTech University, Shanghai, 201210, China}
\email{yuehaitian@shanghaitech.edu.cn}
\thanks{$^3$ H.Y. was partially supported by the Shanghai Technology Innovation Action Plan (No.22JC1402400), and by a Chinese overseas high-level young talents program (2022).}

\subjclass[2020]{35Q55, 35R60, 60B20, 60H25, 60H30, 82B10, 37E20}

\keywords{Nonlinear Schr\"odinger equation, invariance of Gibbs measures, probabilistic well-posedness, propagation of randomness.}

\dedicatory{Dedicated to the memory of Professor Jean Bourgain}

\begin{abstract}
We consider the defocusing nonlinear Schr\"odinger equation on $\mathbb T^2$ with Wick ordered power nonlinearity, and prove almost sure global well-posedness with respect to the associated Gibbs measure. The heart of the matter is the uniqueness of the solution as limit of solutions to canonically truncated systems, and the invariance of the Gibbs measure under the global dynamics follows as a consequence. The proof relies on the novel idea of \emph{random averaging operators}.
\end{abstract}

\maketitle

\section{Introduction} In this paper we study the (defocusing) Wick ordered nonlinear Schr\"{o}dinger equation on the torus $\mathbb{T}^2= (\mathbb{R}/2\pi\mathbb{Z})^2$,
\begin{equation}\label{nls}
\left\{
\begin{split}(i\partial_t+\Delta)u&=W^{2r+1}(u),\\
u(0)&=u_{\mathrm{in}},\end{split}
\right.
\end{equation} where $r$ is a given positive integer, $W^{2r+1}$ is the Wick ordered power nonlinearity of degree $2r+1$, which will be defined below. We prove that, almost surely with respect to the associated Gibbs measure, the equation (\ref{nls}) has a global \emph{strong} solution, which is the unique limit of solutions to the canonical finite dimensional truncations. This solution map keeps the Gibbs measure invariant.

For $r=1$ (cubic nonlinearity) this was proved by Bourgain \cite{Bourgain}; the results for $r\geq 2$ are new. We remark that in \cite{OT} Oh and Thomann constructed almost sure global weak solutions to (\ref{nls}) with respect to the Gibbs measure, such that at any time the law of these random solutions is again given by the Gibbs measure. The main point of the current paper is the almost sure \emph{uniqueness} of the solution with respect to the Gibbs measure.

\subsection{Setup and the main theorem} In this section we setup the problem and state our main theorem. For a review of the background and previous results, see Section \ref{previous}.
\subsubsection{Wick ordering and Gibbs measure}\label{wick} We will fix a probability space $(\Omega,\mathcal{B},\mathbb{P})$, and a set of independent complex Gaussian random variables $\{g_k\}_{k\in\mathbb{Z}^2}$ defined on $\Omega$ that are normalized, i.e. $\mathbb{E}g_k=0$ and $\mathbb{E}|g_k|^2=1$, such that the law of $g_k$ is rotationally symmetric.

Let $\mathcal{V}=\mathcal{S}'(\mathbb{T}^2)$ be the space of distributions on $\mathbb{T}^2$. We define the $\mathcal{V}$-valued random variable
\begin{equation}\label{random}f=f(\omega):\omega\mapsto\sum_{k\in\mathbb{Z}^2}\frac{g_k(\omega)}{\langle k\rangle}e^{ik\cdot x},\quad \omega\in\Omega.
\end{equation} Let $\mathrm{d}\rho$ be the Wiener measure on $\mathcal{V}$, defined for Borel sets $E\subset\mathcal{V}$ by \begin{equation}\rho(E)=\mathbb{P}(f^{-1}(E)),\end{equation} so $\mathrm{d}\rho$ is the law of the random variable $f$. This measure $\mathrm{d}\rho$ is a countably additive Gaussian measure supported in $\bigcap_{\varepsilon>0}H^{-\varepsilon}(\mathbb{T}^2)$, which we henceforth denote by $H^{0-}(\mathbb{T}^2)$ (similarly $H^{s-}(\mathbb{T}^2)=\bigcap_{\varepsilon>0}H^{s-\varepsilon}(\mathbb{T}^2)$), but not in $L^2(\mathbb{T}^2)$ (see e.g. Bogachev \cite{Bogachev}). Define the spectral truncation $\Pi_N$ by
\begin{equation}\label{truncN}\mathcal{F}_x\Pi_Nu(k) :=\mathbf{1}_{\langle k\rangle\leq N}\cdot\mathcal{F}_xu(k),
\end{equation} where $\mathcal{F}_x$ is the space Fourier transform, $\langle k\rangle :=\sqrt{|k|^2+1}$ and $\mathbf{1}_P$ denotes the indicator function of a set or property $P$, and define the expectation of the truncated mass,
\begin{equation}\label{expect}
\sigma_N:=\frac{1}{(2\pi)^2}\mathbb{E}\|\Pi_Nf(\omega)\|_{L^2}^2=\sum_{\langle k\rangle\leq N}\frac{1}{\langle k\rangle^2}\sim \log N.
\end{equation}

For each $N$ and each $p\geq 0$, define the Wick ordered powers,
\begin{equation}\label{wickpoly}
\begin{aligned}
W_N^{2p}(u)&=\sum_{j=0}^p(-1)^{p-j}{p\choose j}\frac{\sigma_N^{p-j}p!}{j!}|u|^{2j},\\
W_N^{2p+1}(u)&=\sum_{j=0}^p(-1)^{p-j}{{p+1}\choose {p-j}}\frac{\sigma_N^{p-j}p!}{j!}|u|^{2j}u,
\end{aligned}
\end{equation} and the canonical finite dimensional truncations for (\ref{nls}),
\begin{equation}\label{truncnls}
\left\{
\begin{split}(i\partial_t+\Delta)u_N&=\Pi_NW_N^{2r+1}(u_N),\\
u_N(0)&=\Pi_Nu_{\mathrm{in}}.\end{split}
\right.
\end{equation}  
The following proposition ensures the convergence of the right hand side of (\ref{truncnls}) as $N\to\infty$, and provides the definition of $W^{2r+1}(u)$ in (\ref{nls}).
\begin{prop}\label{nonlin} Let $n$ be a nonnegative integer. Then almost surely in $u$ with respect to the Wiener measure $\mathrm{d}\rho$, the limit
\begin{equation*}\lim_{N\to\infty}W_N^n(\Pi_Nu)=\lim_{N\to\infty}\Pi_NW_N^n(\Pi_Nu)\end{equation*} exists in $H^{0-}(\mathbb{T}^2)$. We will denote this limit by $W^n(u)$.
\end{prop}
For each $N$, we also define the truncated potential energy
\begin{equation}\label{potentialenergy}V_N[u]:=\frac{1}{r+1}\frac{1}{(2\pi)^2}\int_{\mathbb{T}^2}W_N^{2r+2}(\Pi_Nu)\,\mathrm{d}x.
\end{equation} By Proposition \ref{nonlin}, the limit quantity
\begin{equation}\label{hamilton}V[u] =\lim_{N\to\infty}V_N[u]=\frac{1}{r+1}\frac{1}{(2\pi)^2}\int_{\mathbb{T}^2}W^{2r+2}(u)\,\mathrm{d}x\end{equation} is defined $\mathrm{d}\rho$-almost surely in $u$. We can verify that (\ref{truncnls}) is a finite dimensional Hamiltonian system with Hamitonian
\begin{equation}\label{defham}\mathcal{H}_N[u]:=\frac{1}{(2\pi)^2}\int_{\mathbb{T}^2}|\nabla u|^2\,\mathrm{d}x+V_N[u].
\end{equation}This $\mathcal{H}_N[u]$, as well as the mass $\mathcal{M}[u]:=\frac{1}{(2\pi)^2}\int_{\mathbb{T}^2}|u|^2\,\mathrm{d}x,$ is conserved under the flow (\ref{truncnls}).
\begin{prop}\label{gibbsm} Define the measure $\mathrm{d}\mu$ by
\begin{equation}\label{defgibbs}\mathrm{d}\mu= Z^{-1}e^{-V[u]}\,\mathrm{d}\rho,\qquad Z:=\int_{\mathcal{V}}e^{-V[u]}\,\mathrm{d}\rho(u).\end{equation} Then $\mathrm{d}\mu$ is mutually absolutely continuous with $\mathrm{d}\rho$, and the Radon-Nikodym derivative $Z^{-1}e^{-V[u]}$ belongs to $L^q(\mathrm{d}\rho)$ for any $1\leq q<\infty$. We call this $\mathrm{d}\mu$ the \emph{Gibbs measure} for (\ref{nls}).
\end{prop}

Propositions \ref{nonlin} and \ref{gibbsm} stem from seminal works of Nelson \cite{Nel1,Nel2} and of Glimm-Jaffe \cite{GJ1} in the context of quantum field theory (see also Simon \cite{Simon} and Da Prato-Tubaro \cite{DaT}). As stated, a proof of these propositions can be found in \cite{OT}.

\subsubsection{The main theorem} We can now state our main theorem.
\begin{thm}\label{main} There exists a Borel set $\Sigma\subset\mathcal{V}$ with $\mu(\mathcal{V}\backslash\Sigma)=0$, such that $W^{2r+1}(u)\in H^{0-}(\mathbb{T}^2)$ is well-defined for $u\in\Sigma$. Furthermore:
\begin{enumerate}
\item For each $u_{\mathrm{in}}\in\Sigma$ and each $t\in\mathbb{R}$, the solution $u_N(t)$ to (\ref{truncnls}) converges to a unique limit \begin{equation}\label{limittrunc}\lim_{N\to\infty}u_N(t)=u(t)\end{equation} in $H^{0-}(\mathbb{T}^2)$, and $u(t)\in\Sigma$ for each $t\in\mathbb{R}$. This $u(t)$ solves (\ref{nls}) in the distributional sense.\\

\item  The limit $u(t)$ in \eqref{limittrunc} defines, for each $t\in\mathbb{R}$, a map from $\Sigma$ to itself: $u(t)=: \Phi_tu_{\mathrm{in}}$. These maps then satisfy the usual group properties, and keep the Gibbs measure $\mathrm{d}\mu$ invariant, namely
\begin{equation}\mu(E)=\mu(\Phi_t(E))
\end{equation} for any Borel set $E\subset\Sigma$.
\end{enumerate}
\end{thm}
\begin{rem}\label{fixep}  In proving Theorem \ref{main}, we will replace $H^{0-}(\mathbb{T}^2)$ by $H^{-\varepsilon}(\mathbb{T}^2)$ where $0<\varepsilon\ll 1$ and throughout the proof we will fix this $\varepsilon$ (we then take a countable intersection in $\varepsilon$).
\end{rem}
\begin{rem}
The Wick ordering (\ref{wickpoly}) is needed in order to have a meaningful solution theory, due to the infinite mass in the support $H^{0-}(\mathbb{T}^2)$ of the Gibbs measure $\mathrm{d}\mu$; see e.g. \cite{GO15}.
\end{rem}
\begin{rem}\label{powerof2}  (1) Since the Gibbs measure $\mathrm{d}\mu$ is mutually absolutely continuous with the Gaussian measure $\mathrm{d}\rho$, part (1) of Theorem \ref{main} can also be viewed as an almost sure global well-posedness result with random initial data $u_{\mathrm{in}}=f(\omega)$ as in (\ref{random}).

(2) The truncation frequency $N$ in (\ref{truncnls}) can be any positive number. For simplicity, in the proof below we will assume that $N\in 2^{\mathbb{Z}}$. The general case follows from placing $N$ between $N'/2$ and $N'$ where $N'\in 2^{\mathbb{Z}}$, and analyzing the difference $u_{N'}-u_N$ in the same way as $u_{N'}-u_{N'/2}$.

(3) The solution $u(t)$ defined in Theorem \ref{main} is unique in the following sense. Suppose we replace the truncation in (\ref{truncnls}) by any other canonical truncation; say we replace $\Pi_N$ in (\ref{truncN}) by the multiplier with symbol $k\mapsto\varphi(\langle k\rangle/N)$ in both (\ref{expect}) and the second line of (\ref{truncnls}), where $\varphi(z)$ is either $\mathbf{1}_{|z|\leq 1}$ or a smooth function supported in $|z|\leq 1$ and $\varphi(0)=1$. We may also keep the $\Pi_N$ in the first line of (\ref{truncnls}) unchanged, or replaced it by $1$. Then in any case Theorem \ref{main} remains true, moreover the set $\Sigma$ and the limit $u(t)$ obtained \emph{do not depend on $\varphi$}. The proof is basically the same, with only minor adjustments in a few places. We will not pursue these matters here.
\end{rem}

\begin{rem}\label{2rem}
(1)  Theorem \ref{main} is part of the program of constructing invariant Gibbs measures and studying their dynamics for the (renormalized) defocusing nonlinear Schr\"odinger equation (\ref{nls}), see Section \ref{results1} below. With Bourgain \cite{Bourgain94} completing all cases with $d=1$ and  Theorem \ref{main} completing all cases with $d=2$ after Bourgain \cite{Bourgain}, the remaining cases that are expected to be solvable{\footnote{Gibbs measures for \eqref{nls} are available only for $d=1$ (both focusing and defocusing), $d=2$ (defocusing only), and $(d,r)=(3,1)$ (defocusing only). This is related to the existence/nonexistence of $\phi_d^{2r+2}$ theories, see for example Aizenman \cite{Aiz}, Fr\"{o}hlich \cite{Fro}, Brydges-Slade \cite{BrySl} and  Aizenman-Duminil Copin \cite{ADC}. The white noise measure is always formally invariant under (\ref{nls}), but is compatible with the dynamics only when $(d,r)=(1,1)$.
}} are the invariance of Gibbs measure for $(d,r)=(3,1)$, and invariance of white noise for $(d,r)=(1,1)$. We expect both to be strictly harder than Theorem \ref{main} as they are \emph{critical in the probabilistic scaling}, see Section \ref{probsc}.

(2) The corresponding problem of constructing invariant Gibbs measures for nonlinear \emph{wave} equations, is in general much easier than the Schr\"{o}dinger problem, due to the derivative gain in Duhamel's formula. Indeed, this has been completely solved, due to the results of  Friedlander \cite{Fre85}  (see Zhidkov \cite{Zhi94} under more restrictive assumptions on the nonlinearity)  in dimension $d=1$, Bourgain \cite{Bou99} in dimension $d=2$ (see also  Oh-Thomann \cite{OT2}) and the recent result of the authors with Bringmann \cite{BDNY22} for $(d,r)=(3,1)$ proved after the current paper was submitted.
\end{rem}
\subsection{A review of previous results}\label{previous}  We start by reviewing previous results and methods on PDEs in the probabilistic setting. As the literature is now extensive, we will put emphasis on the works most relevant to the current paper.
\subsubsection{Invariant measures}\label{results1} Since the pioneering works of Lebowitz-Rose-Speer \cite{LRS} and Bourgain \cite{Bourgain94,Bourgain}, there have been numerous results regarding invariant measures for nonlinear dispersive equations. In general, for any Hamiltonian dispersive equation we may construct the associated Gibbs measure
\begin{equation}\label{gibbsgen}\mathrm{d}\mu\sim e^{-\beta H}\prod_{x}\mathrm{d}x,
\end{equation} where $\beta>0$ is a fixed parameter (which we may set to be $1$) and $H$ is the Hamiltonian. The definition (\ref{gibbsgen}) is only formal; in some cases it can be justified by using the Gaussian measure as a reference measure and writing $\mathrm{d}\mu$ as a weighted Wiener measure.  For example the Hamiltonian for (\ref{nls}) is 
\[H=\frac{1}{(2\pi)^2}\int_{\mathbb{T}^2}(|\nabla u|^2+\frac{1}{r+1}W^{2r+2}(u))\,\mathrm{d}x\] and the Gibbs measure
\[\mathrm{d}\mu\sim \underbrace{\exp\bigg[\frac{-1}{r+1}\frac{1}{(2\pi)^2}\int_{\mathbb{T}^2}W^{2r+2}(u)\,\mathrm{d}x\bigg]}_{\textrm{weight}}\cdot\underbrace{\exp\bigg[\frac{-1}{(2\pi)^2}\int_{\mathbb{T}^2}|\nabla u|^2\,\mathrm{d}x\bigg]\prod_{x\in\mathbb{T}^2}\,\mathrm{d}x}_{\textrm{Gaussian measure}}\] can be rigorously defined as a weighted Wiener measure\footnote{Strictly speaking the measure defined in Proposition \ref{gibbsm} involves an additional weight which is an exponential of the $L^2$ mass. As the mass is also conserved, this does not affect any invariance properties.}, as in Proposition \ref{gibbsm}. Defining such Gibbs-type measures and studying their properties under various dynamics is a major problem in constructive quantum field theory.

The Gibbs measure $\mathrm{d}\mu$ for a given dispersive equation  is \emph{formally} invariant due to a ``formal Liouville's Theorem" and the conservation of Hamiltonian. It is of great interest to establish this invariance rigorously, as this would be the first step in studying the global dynamics from the statistical ensemble point of view. In \cite{Bourgain94, Bourgain}, Bourgain developed a systematic way of showing the invariance of $\mathrm{d}\mu$ from the invariance of finite dimensional Gibbs measures, provided we have local well-posedness or almost sure local well-posedness with respect to $\mathrm{d}\mu$.

Therefore, justifying the invariance of $\mathrm{d}\mu$ (and other similar formally invariant measures) basically reduces to proving almost sure local well-posedness on the support of $\mathrm{d}\mu$. As this support is very rough in high dimensions (namely $H^{1-\frac{d}{2}-}$ for (\ref{nls}) in dimension $d$), most known results are limited to one dimension, or require strong symmetry. For  the Schr\"{o}dinger equation (\ref{nls}) on the torus $\mathbb{T}^d$, Bourgain \cite{Bourgain94} solved the case $d=1$, and extended this to $d=2$ and $r=1$ in \cite{Bourgain}. These are the only results known for the nonlinear Schr\"odinger equation prior to the current paper. More is known for wave equations as noted in Remark \ref{2rem} (2) above; see \cite{Fre85, Zhi94, Bou99, OT2, BDNY22}.

Apart from the standard Schr\"{o}dinger and wave models on tori, there are many results, again mostly in one dimension or under radial symmetry, where the invariance of corresponding Gibbs measures (or of associated weighted Wiener measures) are justified for various dispersive models on various background manifolds (see e.g.  \cite{Tz0, Oh, Tz, TTz, NORS,NRSS, dS, Deng,Deng2,DTV, Giordi, Thomann,Sy} and references therein). We also mention the compactness method of Alveberio and Cruzeiro \cite{AC} which explores the tightness of the sequence of finite dimensional measures and applies the theorems of Prokhorov and Skorokhod to obtain existence of weak solutions (see e.g. \cite{BTT, NPST, OT, WangYue}). These are less related to the current paper and we will not elaborate further.

\subsubsection{Probabilistic well-posedness theory}\label{results2} It has long been known that PDEs with \emph{randomness} generally behave better in terms of local well-posedness (i.e. probabilistic well-posedness holds below the deterministic well-posedness threshold). Progress  has been made in two parallel directions: random initial data problems and stochastically forced problems.

The first results along this line are due to the seminal works by Bourgain \cite{Bourgain, Bourgain2}  in the random data setting  and later to Da Prato-Debussche \cite{DD, DD2}  in the stochastic setting. The idea in both cases is to make a linear-nonlinear decomposition and observe the effect of \emph{probabilistic smoothing}. For example, in \cite{Bourgain}, the equation (\ref{nls}) with $r=1$ on $\mathbb{T}^2$ is studied with random initial data in $H^{-\varepsilon}$ for some $0<\varepsilon\ll 1$, in which (\ref{nls}) is deterministically ill-posed. However with randomness we may construct solutions to (\ref{nls}) that have the form $u=e^{it\Delta}u(0)+v$, where $u(0)\in H^{-\varepsilon}$ is the random initial  data, and $v$ belongs to some \emph{positive} Sobolev space in which (\ref{nls}) is well-posed. In other words, the solution contains a rough random part $u_{\mathrm{lin}}:=e^{it\Delta}u(0)$ and a smooth remainder $v$. The point here is that, even though $u_{\mathrm{lin}}$ is rough, it has an  explicit random structure which allows us to control the nonlinear interactions between $u_{\mathrm{lin}}$ and $u_{\mathrm{lin}}$, and between $u_{\mathrm{lin}}$ and $v$, in a more regular space.

Until recently the method of Bourgain, as well as its higher order variants which include some nonlinear interactions  of $u_{\mathrm{lin}}$ with itself into the rough random part\footnote{This usually results in a finite or infinite tree expansion.}, has been the dominant strategy of exploiting randomness in the local well-posedness theory for dispersive and wave equations with random data. After Bourgain's pioneering work, there has been substantial success  using this method  (for a sample of works, we refer the readers to \cite{Bourgain, BTlocal, CoOh, Deng, BB4, NS, HYue, Bringmann, KMV, BOhP1, DLM0, KM} and references therein). However this method by itself has its limitations and does not lead to optimal results in most cases.

A few years ago, a series of important works emerged, which revolutionized the study of local well-posedness for stochastically forced PDEs, in fact reaching the optimal exponents in the parabolic case. These include the \emph{theory of regularity structures} of Hairer \cite{Hairer0, Hairer, Hairer2, Hairer3} and the \emph{para-controlled calculus} of Gubinelli-Imkeller-Perkowski \cite{GIP, GIP2}. A third method based on Wilsonian renormalization group analysis was independently proposed by Kupiainen in \cite{Kupia}. 

The theory of regularity structures is based on the local-in-space properties of solutions at fine scales (so it is particularly suitable for parabolic equations); it builds a general theory of distributions which includes the profiles coming from the noise, and allows for multiplications and analysis of the nonlinearity. Since its success with the KPZ equation \cite{Hairer0} and the $\Phi^4_3$ model \cite{Hairer}, this theory has been  further developed by Hairer and collaborators and is now powerful enough to solve a wide range of problems that are subcritical according to a suitable parabolic scaling. We will not get into the details, but we refer the reader to \cite{ FrH, Hairer3, Hairer4, Hairer5, HLab,ChW, MouWeXu, MouWe3} and references therein for nice expositions of these ideas.

The theory of para-controlled calculus, which is in spirit the point of departure of the present paper,  takes a different approach and is based on the following idea.
 In the approach of Bourgain and of Da Prato-Debussche mentioned above,  some nonlinear interactions between $u_{\mathrm{lin}}$ and $v$ may not have enough regularity despite  $v$ being more regular than $u_{\mathrm{lin}}$. However, a key observation is that, the only bad terms here are the high-low interactions where the high frequencies come from $u_{\mathrm{lin}}$ and the low frequencies come from $v$, and such terms  can be \emph{para-controlled} by the high-frequency inputs (which are nonlinear interactions of $u_{\mathrm{lin}}$ with itself).  Here  $f$ being para-controlled by $g$ simply means that $f$ equals the high-low paraproduct of $g$ with some other function $h$, up to a smoother remainder. With such structure, these para-controlled terms can be shown to have similar randomness structures as the nonlinear interactions of $u_{\mathrm{lin}}$ with itself, and thus can be handled similarly as in Bourgain's or Da Prato-Debussche's approach, leaving an even smoother remainder. The para-controlled calculus also has a higher order variant, see \cite{BaBe, BaBeFr}. We refer the reader to \cite{CCh,GP, GP2, GP3, GP4, MouWe3, Perk, BaBe, BaBe2, BaBeFr} and references therein for expositions of these ideas and some other recent developments on this method.

Finally, we would like to mention two recent results of Gubinelli-Koch-Oh \cite{GKO} and Bringmann \cite{Bringmann}. In \cite{GKO} the authors applied a version of para-controlled calculus to the stochastic wave equation setting, and obtained almost sure local well-posedness for a quadratic wave equation with additive white noise on $\mathbb{T}^3$. This relied on several new ingredients, including the analysis of a random operator (which is different from and unrelated to the random averaging operator introduced in the current paper). 

In \cite{Bringmann} the author studied the nonlinear wave equation with quadratic derivative nonlinearity on $\mathbb{R}^3$ and improved the known well-posedness threshold with random initial data, again by analyzing high-low interactions. This work introduced an improved para-controlled scheme that accounts for all the bad terms in the absence of linear smoothing effects, and also made the very important observation that the high-frequency and low-frequency parts can in fact be made \emph{independent}. See Section \ref{earlymethod} for more detailed discussions. 
\begin{rem}\label{discuss} We note from previous discussions that, in the same dimension and for the same nonlinearity, the probabilistic improvement (defined as the difference between exponents of the deterministic $H^s$ well-posedness threshold and the obtained probabilistic $H^s$ well-posedness threshold) is always much smaller for \emph{Schr\"{o}dinger equations} compared to wave and heat equations.

There are two reasons for this. First, heat equations are compatible with H\"{o}lder spaces $C^s$,  which scale much higher than $H^s$, but a function with independent Gaussian Fourier coefficients that belongs to $H^s$ will automatically belong to $C^{s-}$ due to Khintchine's inequality. This allows for a scaling at a higher regularity
and hence be in a better situation when studying heat equations. Such advantage cannot be exploited for  Schr\"{o}dinger and wave equations, since $C^s$ spaces are not compatible even with their linear evolution, and cannot be used in any well-posedness theory.

Second, the Duhamel evolution for the heat equation gains two derivatives, and for the wave equation gains one\footnote{Note the \emph{derivative} wave equation studied in \cite{Bringmann} behaves similarly to Schr\"{o}dinger as the one derivative gain is cancelled by the derivative nonlinearity. Therefore the probabilistic improvement obtained in \cite{Bringmann} is also a tiny amount compared to other results for wave equations.}. This allows for room to apply Sobolev embedding, and also reduces the task of controlling the nonlinearity to the task of making sense of products, which is still hard but at least more manageable. In comparison, the Schr\"{o}dinger Duhamel evolution has no smoothing effect, and it can be challenging to close the estimates even when the relevant products are well-defined.
\end{rem}
\subsection{Difficulties and the strategy}\label{method}

We now turn to the proof of Theorem \ref{main}. This proof consists of two parts: (a) proving almost sure local well-posedness for (\ref{nls}) on the support of the Gibbs measure, and (b) applying formal invariance to extend local solutions to global ones. Since part (b) is essentially an adaptation of Bourgain's classical proof \cite{Bourgain}, we shall focus on the local theory in part (a). For exposition simplicity we will also replace $W^{2r+1}(u)$ by the pure power $|u|^{2r}u$ in the discussion below.

The obvious difficulty here is that the Gibbs measure $\mathrm{d}\mu$ is supported in $H^{0-}$, while the (deterministic) scaling threshold, below which (\ref{nls}) is ill-posed, is $s_c=1-\frac{1}{r}\to 1$ as $r\to\infty$. In the language of Remark \ref{discuss}, we need to obtain a probabilistic improvement  $\approx 1$. Therefore, it is important to understand exactly how randomness allows us to beat scaling. This is contained in the notion of \emph{probabilistic scaling}, which we discuss below.
\subsubsection{The probabilistic scaling}\label{probsc} Consider the nonlinear Schr\"odinger equation
\begin{equation}\label{nlsex}(i\partial_t+\Delta)u=|u|^{2r}u,\quad u(0)=u_{\mathrm{in}}
\end{equation} on $\mathbb{T}^d$. The scaling critical threshold for (\ref{nlsex}) is
\begin{equation}\label{detersc} 
s_c=\frac{d}{2}-\frac{1}{r},
\end{equation} and (\ref{nlsex}) is expected to be locally well-posed in $H^s$ only if $s\geq s_c$. This can be demonstrated in multiple ways, but the one most relevant to us is as follows. Suppose the initial data $u_{\mathrm{in}}$ has Fourier transform $\mathcal{F}_xu_{\mathrm{in}}(k)$ supported in $|k|\sim N$ with $|\mathcal{F}_xu_{\mathrm{in}}(k)|\sim N^{-\alpha}$ with $\alpha=s+\frac{d}{2}$, then $\|u_{\mathrm{in}}\|_{H^s}\sim 1$. If local well-posedness holds then we should expect that the second iteration (say at time $t=1$),
\[u^{(1)}:=\int_0^1 e^{i(1-t')\Delta}(|e^{it'\Delta}u_{\mathrm{in}}|^{2r}e^{it'\Delta}u_{\mathrm{in}})\,\mathrm{d}t',\] satisfies $\|u^{(1)}\|_{H^s}\lesssim 1$. 
By performing Fourier expansions, we essentially get
\begin{equation}\label{naivedeter}\mathcal{F}_xu^{(1)}(k)\sim\sum_{\substack{k_1-\cdots+k_{2r+1}=k\\|k_j|\lesssim N}}\frac{1}{\langle\Sigma\rangle}\prod_{j=1}^{2r+1}\widehat{u_{\mathrm{in}}}(k_j),\quad \Sigma=|k|^2-|k_1|^2+\cdots-|k_{2r+1}|^2,\end{equation} where complex conjugates are omitted. In the worst scenario this gives, up to logarithmic factors,
\[|\mathcal{F}_xu^{(1)}(k)|\lesssim N^{-(2r+1)\alpha}\sup_{m\in\mathbb{Z}}\# S_m,\] where $S_m=\{(k_1,\cdots,k_{2r+1}):k_1-\cdots+k_{2r+1}=k, \, |k_j|\lesssim N, \, \Sigma=m \}.$  By dimension counting, we expect that $\#S_m\lesssim N^{2rd-2}$, so in order for $\|u^{(1)}\|_{H^s}\lesssim 1$ we need $-(2r+1)\alpha+2rd-2\leq -\alpha,$ or equivalently $s\geq s_c.$

Now, in the \emph{random data} setting, suppose the Fourier coefficients of initial data $\{\widehat{u_{\mathrm{in}}}(k)\}$ are \emph{independent Gaussians} of size $N^{-\alpha}$. The sum (\ref{naivedeter}) will then be a sum of products of independent Gaussian random variables, which is reminiscent of the classical Central Limit Theorem. Recall that in the latter we have a sum of $M$ independent random objects of unit size, and under certain general conditions, this sum scales only like $\sqrt{M}$ as opposed to $M$ if without randomness. In the same way, we would expect essentially a `square root gain' here, that is,
\[|\mathcal{F}_xu^{(1)}(k)|\lesssim N^{-(2r+1)\alpha}\sup_{m\in\mathbb{Z}}(\# S_m)^{\frac{1}{2}}\lesssim N^{-(2r+1)\alpha+rd-1},\] so in order for $\|u^{(1)}\|_{H^s}\lesssim 1$ it suffices to have $-(2r+1)\alpha+rd-1\leq -\alpha$, or equivalently
\begin{equation}\label{probscale}s\geq s_p:=-\frac{1}{2r}.
\end{equation} Note that $s_p$ is independent of the dimension and that we always have $s_p\leq s_c$. We will call this $s_p$ the critical threshold for \emph{probabilistic scaling}\footnote{This is associated with Gaussian random variables, but by the Central Limit Theorem, the scaling is the same for more general types of random variables.}.   We refer the reader to \cite{DNY24} for further discussions about the probabilistic scaling paradigm.

\begin{rem}\label{probscrem}
With the above heuristics, it is natural to expect that (\ref{nls}) will be almost surely locally well-posed with random initial data in $H^{s-}(\mathbb{T}^d)$, i.e. 
\[u_{\mathrm{in}}=\sum_{k\in\mathbb{Z}^d}\frac{g_k}{\langle k\rangle^\alpha}e^{ik\cdot x},\quad \alpha=s+\frac{d}{2},\] in any \emph{probabilistically subcritical} space with $s>s_p$.

Indeed, after the current paper was submitted, the authors have completed the paper \cite{DNY22} that establishes this result. The proof in \cite{DNY22} is based on the {\emph {theory of random tensors}}, which is the natural extension of the main technique of the current paper, i.e. the method of {\emph{random averaging operators}} introduced in Section \ref{rao} below.
\end{rem}
\subsubsection{Discussions of earlier methods}\label{earlymethod} With the scaling heuristics in Section \ref{probsc}, note that the support of the Gibbs measure $\mathrm{d}\mu$, which is $H^{0-}$, is \emph{above} the probabilistic critical space $H^{s_p-}$, where $s_p=-\frac{1}{2r}$. Therefore it is reasonable to believe that almost sure local well-posedness of (\ref{nls}) should hold in the support of $\mathrm{d}\mu$. However, the justification of such heuristics is far from trivial, due to the intrinsic difficulties associated with the Schr\"{o}dinger equation explained in Remark \ref{discuss}.

To  motivate the method of random averaging operators introduced in  this paper, we first discuss the possibility of applying existing methods\footnote{We will not discuss the regularity structure theory \cite{Hairer}, as it relies on local expansions in physical space and is thus not compatible with Schr\"{o}dinger or wave equations.} to the setting of (\ref{nls}), the ideas behind these other methods, and why these ideas do not work here. Below we set $u_{\mathrm{in}}=f(\omega)$ in (\ref{nls}).

\smallskip
\uwave{Method 1: Bourgain-Da Prato-Debussche.} We start with the classical idea of Bourgain \cite{Bourgain} and Da Prato-Debussche \cite{DD,DD2}. As discussed in Section \ref{results2}, this idea is based on the following observation:
\begin{itemize}
\item Observation 1: nonlinear components of random data solutions always have higher regularity than linear ones.
\end{itemize}

Now, suppose we apply this idea to (\ref{nls}), which leads to the ansatz $u(t)=e^{it\Delta}f(\omega)+w$, where $w$ belongs to $C_t^0H_x^s$, or more precisely $X^{s,\frac{1}{2}+}$ (see Section \ref{funcspace} for relevant definitions) for some positive $s$. In particular, this $w$ will contain components of form 
\begin{equation}\label{seconditer}
u^{(1)}(t)=I(|e^{it\Delta}f(\omega)|^{2r}e^{it\Delta}f(\omega)),\quad IF(t):=\int_0^t e^{i(t-t')\Delta}F(t')\,\mathrm{d}t'.\end{equation}However, it is shown in \cite{Bourgain} that even when $r=1$ (and obviously also for larger $r$), the $u^{(1)}$ defined in (\ref{seconditer}) belongs to $X^{s,\frac{1}{2}+}$ \emph{only for $s<\frac{1}{2}$}. As the space $X^{\frac{1}{2}-,\frac{1}{2}+}$ is still supercritical with respect to deterministic scaling for $d=2$ and $r\geq 2$, there will be no hope of solving (\ref{nls}) using the above ansatz. We may perform higher order Picard iterations, but it turns out that regardless of the order, there is always some contribution in the remainder that has regularity $X^{\frac{1}{2}-,\frac{1}{2}+}$, and the problem persists.

\smallskip
\uwave{Method 2: Para-controlled calculus.} Next we may try the idea of para-controlled calculus of Gubinelli-Imkeller-Perkowski \cite{GIP} (and Gubinelli-Koch-Oh \cite{GKO} for wave). This is based on the following observation, in addition to Observation 1 above:
\begin{itemize}
\item Observation 2: in probabilistically subcritical settings, the bad regularity of the nonlinear (and subsequent non-explicit) terms only come from high-low interactions.
\end{itemize}

Now, applying this idea to (\ref{nls}) would lead to the ansatz
\begin{equation}\label{ansatzintro2}u=u_{\mathrm{lin}}+X+Y,
\end{equation} where
\begin{equation}\label{ansatzintro2-2}u_{\mathrm{lin}}=e^{it\Delta }f(\omega),\quad X=\sum_{N}I(P_Nu_{\mathrm{lin}}\cdot |P_{\ll N}u|^{2r}),
\end{equation} $I$ is as in (\ref{seconditer}), and $P_N$ are the standard Littlewood-Paley projections. The term $X$ para-controlled by  $u_{\mathrm{lin}}$ will be constructed in some less regular space (say $X^{\frac{1}{2}-,\frac{1}{2}+}$), which allows the remainder $Y$ to be constructed in a more regular space (say $X^{1-,\frac{1}{2}+}$).

However, since the Duhamel operator $I$ gains no derivative, the term $Y$ will contain contributions of form
\begin{equation}\label{ansatz-defz}Z:=\sum_{N}I(P_N X\cdot |P_{\ll N}u|^{2r}),\end{equation} which in fact also has regularity $X^{\frac{1}{2}-,\frac{1}{2}+}$. This shows that the ansatz (\ref{ansatzintro2}) is still not sufficient for (\ref{nls}).

\smallskip
\uwave{Method 3: Bringmann's method.} Bringmann's paper \cite{Bringmann}  introduced a refinement to the para-controlled ansatz in \cite{GIP,GKO} designed for derivative wave equations (where the Duhamel evolution gains no derivative), based on the following observation in addition to Observations 1 and 2 above:
\begin{itemize}
\item Observation 3: the further terms of form $Z$ and similar as in (\ref{ansatz-defz}) can be packed in a para-linearized solution, and moreover the high and low frequency inputs are independent.
\end{itemize}

More precisely, \cite{Bringmann} suggests to modify the ansatz to
\begin{equation}\label{ansatzintro2-3}u=u_{\mathrm{lin}}+X+Y,
\end{equation} where $X$ is the solution to the para-linearized equation
\begin{equation}\label{ansatzintro2-4}X=\sum_{N}I(P_N(u_{\mathrm{lin}}+X)\cdot |P_{\ll N^\alpha}u|^{2r}),\end{equation} for some $\alpha\in(0,1)$, which is constructed in the less regular space $X^{\frac{1}{2}-,\frac{1}{2}+}$, and the remainder $Y$ is constructed in the more regular space $X^{1-,\frac{1}{2}+}$. Moreover, the low frequency input $P_{\ll N^\alpha}u$ is essentially \emph{independent} with $P_Nu_{\mathrm{lin}}$, which is a crucial observation first made in \cite{Bringmann}. This allows for an inductive construction, and also allows to use more powerful probabilistic tools such as multilinear Wiener chaos estimates.

However, the above ansatz is still not sufficient to control both $X$ and $Y$ in the desired regularities, due to the following reason. Since $Y$ contains high-high interactions, where by ``high" we mean frequencies $\gtrsim N^\alpha$, it is easy to see that $\alpha$ has to be chosen close to $1$ for large $r$, in order for $Y$ to be bounded in $X^{1-,\frac{1}{2}+}$; in fact a calculation similar to those in Section \ref{probsc} shows that $\alpha\geq 1-2/r$. But when $\alpha$ is close to $1$, the expression of $X$ will involve terms
\[\sum_{N}I(P_Nu_{\mathrm{lin}}\cdot |P_{\ll N^\alpha}X|^{2r}).\] To control this contribution, even with independence between $P_Nu_{\mathrm{lin}}$ and $P_{\ll N^\alpha}X$, we would still need to bound the power $|P_{\ll N^\alpha}X|^{2r}$ in a suitable space; however in the ansatz of \cite{Bringmann} we only know that $X$ is a function in $X^{\frac{1}{2}-,\frac{1}{2}+}$, which does not imply any useful control for $|P_{\ll N^\alpha}X|^{2r}$ due to super-criticality of this space. Of course we may also exploit the equation (\ref{ansatzintro2-4}) satisfied by $X$, but each iteration of this equation produces higher powers of $X$, which again cannot be bounded using the ansatz of \cite{Bringmann}.

Note that when $r=2$, it might be possible to carry out the scheme of \cite{Bringmann} with a small $\alpha$ by doing some refined analysis (which is by no means immediate in view of all the difficulties for the Schr\"odinger equation in Remark \ref{discuss}). Our approach, which is described below, allows instead for a \emph{unified} treatment for all values of $r$ by synthesizing the main underlying ideas and crucially capturing the true randomness structure of the solution.
\subsubsection{Random averaging operators}\label{rao} It is now clear that, to solve (\ref{nls}), it is necessary to exploit all three observations in Section \ref{earlymethod} (in particular it is very important to exploit the independence between high and low frequency inputs as first done by Bringmann in \cite{Bringmann}), but this is still not enough. In fact, the missing piece is the following observation, in addition to Observations 1--3 in Section \ref{earlymethod}:
\begin{itemize}
\item Observation 4: The input $P_{\ll N^\alpha}u$ in (\ref{ansatzintro2-4}) is \emph{not} an arbitrary function in $X^{\frac{1}{2}-,\frac{1}{2}+}$; it has its own randomness structure, which must be captured in order for the ansatz to be complete.
\end{itemize}

This randomness property of $P_{\ll N^\alpha}u$ cannot be captured by any norm of the \emph{para-controlled term} $X$ defined in (\ref{ansatzintro2-4}). Indeed, to see and exploit this randomness property, we need to perform a {\bf shift of paradigm}, by turning the focus from the para-controlled term $X$ to the \textbf{linear operator} defined by
\begin{equation}\label{randomop0}\mathcal{Q}:y\to z,\quad\mathrm{where}\quad z=\sum_NI(P_N(y+z)\cdot|P_{\ll N^\alpha}u|^{2r});\end{equation} note that by this definition we have $X=\mathcal{Q}(u_{\mathrm{lin}})$. In practice it is more convenient to consider the simpler operator
\begin{equation}\label{randomop}\mathcal{P}:y\mapsto \sum_NI(P_{N}y\cdot|P_{\ll N^\alpha}u|^{2r})
\end{equation} which obeys the same estimates; in fact from (\ref{randomop0}) and (\ref{randomop}) it is easy to see that $\mathcal{Q}=\mathcal{P}(1-\mathcal{P})^{-1}$.  Now we shall \emph{extract all the randomness properties of $P_{\ll N^\alpha}u$, as well as properties of the multilinear expression $|\cdot|^{2r}$ when applied to these random objects, and turn them into two particular norm bounds for the operator $\mathcal{P}$ defined in (\ref{randomop}) (or $\mathcal{Q}$ in (\ref{randomop0})).}

This is the key idea behind our method of {\bf random averaging operators.} The norms we choose are the operator norm $\|\cdot\|_{\mathrm{OP}}$ and the Hilbert-Schmidt norm $\|\cdot\|_{\mathrm{HS}}$ with $\mathcal{P}$ (or $\mathcal{Q}$) viewed as a linear operator from the space $X^{s,\frac{1}{2}+}$ to itself; this does not depend on $s$ so we may in fact choose $s=0$. When the input $y$ in (\ref{randomop}) is a linear Schr\"{o}dinger flow, we get an operator from $H^s$ to $X^{s,\frac{1}{2}+}$, and will also measure the corresponding operator and Hilbert-Schmidt norms. Suppose the maximum frequency of $P_{\ll N^\alpha}u$ in (\ref{randomop}) is $L\ll N^\alpha$, then roughly speaking, these norm bounds will look like
\begin{equation}\label{normintro}\|\mathcal{P}\|_{\mathrm{OP}}\lesssim L^{-\delta_0},\quad \|\mathcal{P}\|_{\mathrm{HS}}\lesssim N^{\frac{1}{2}+\delta_1}L^{-\frac{1}{2}},
\end{equation} where $\delta_1\ll\delta_0\ll 1$. Note that these bounds are obviously false for arbitrary functions $P_{\ll N^\alpha}u\in X^{\frac{1}{2}-,\frac{1}{2}+}$, so they indeed encode the (implicit) randomness structure of $P_{\ll N^\alpha}u$, which is reflected via the multilinear expression $|\cdot|^{2r}$.

We also remark that the choices of the operator and Hilbert-Schmidt norms and the bounds in (\ref{normintro}) are natural in the following sense: the operator norm bound in (\ref{normintro}) is what guarantees the solvability of the para-linearized equation (\ref{ansatzintro2-4}), and the Hilbert-Schmidt norm bound in (\ref{normintro}) is what guarantees that the term $X$ defined in (\ref{ansatzintro2-4}) belongs to $X^{\frac{1}{2}-,\frac{1}{2}+}$. Moreover, the estimate (\ref{normintro}) is preserved under multiplication (which is important in the proof), due to the simple inequalities
\[\|\mathcal{P}_1\mathcal{P}_2\|_{\mathrm{OP}}\leq \|\mathcal{P}_1\|_{\mathrm{OP}}\|\mathcal{P}_2\|_{\mathrm{OP}},\quad \|\mathcal{P}_1\mathcal{P}_2\|_{\mathrm{HS}}\leq \min(\|\mathcal{P}_1\|_{\mathrm{HS}}\|\mathcal{P}_2\|_{\mathrm{OP}},\|\mathcal{P}_1\|_{\mathrm{OP}}\|\mathcal{P}_2\|_{\mathrm{HS}}).\] See Proposition \ref{localmain} for details.

 These operators $\mathcal{P}$ and $\mathcal{Q}$, which will be of central importance in our proof, depend on the object $P_{\ll N^\alpha}u$ that has an implicit randomness structure.  Moreover, in the Fourier variables, this operator can be viewed as a weighted average taken over smaller scales $L\ll N^\alpha$. We thus call it a \emph{random averaging operator}, which explains the name of our method.
 \subsubsection{The full ansatz}\label{secansatz}
 We can now describe the full ansatz of the solution $u$ to (\ref{nls}), with random averaging operators. On the surface we consider the same framework as in \cite{Bringmann}, namely $u=u_{\mathrm{lin}}+X+Y$, where $X$ is defined as in (\ref{ansatzintro2-4}) and $Y\in X^{1-,\frac{1}{2}+}$ is a smooth remainder. However, in our ansatz, instead of merely estimating $X$ in $X^{\frac{1}{2}-,\frac{1}{2}+}$, we exploit the fact that $X$ equals a random averaging operator applied to $u_{\mathrm{lin}}$, where this random averaging operator satisfies a bound of form (\ref{normintro}). To be precise, choosing $\alpha<1$ close to $1$, we further write that  \begin{equation}\label{fullansatz}
X=\mathcal{Q}(u_{\mathrm{lin}}),
\end{equation} so we have the ansatz
\begin{equation}\label{fullansatz2}
u=u_{\mathrm{lin}}+\mathcal{Q}(u_{\mathrm{lin}})+Y.
\end{equation}Here the random averaging operator $\mathcal{Q}=\mathcal{P}(1-\mathcal{P})^{-1}$, and $\mathcal{P}$ has the form
 \[\mathcal{P}=\sum_{N}\sum_{L\ll N^\alpha}\mathcal{P}_{NL},\] where $\mathcal{P}_{NL}$ has the form (\ref{randomop}) with the maximum frequency of $P_{\ll N^\alpha}u$ in (\ref{randomop}) being $L$.  This $\mathcal{P}_{NL}$ is a Borel function of $(g_k(\omega))_{\langle k\rangle\leq L}$, which is independent with $e^{it\Delta}P_Nf(\omega)$, and satisfies (\ref{normintro}); the operator $\mathcal{Q}$ has a similar decomposition into $\mathcal{Q}_{NL}$ that satisfies the same bounds (\ref{normintro}). See Section \ref{decompsol} for the precise formulas.
 
 With the ansatz (\ref{fullansatz2}), the proof of local well-posedness then proceeds by inducting on frequencies to show (\ref{normintro}) and to bound $Y$ in $X^{1-,\frac{1}{2}+}$. In fact, suppose these are true for components of frequency $\ll N^\alpha$, then the bounds (\ref{normintro}) and high-low frequency independence imply that the part of $P_{\ll N^\alpha}u$ involving the random averaging operators really behaves like a linear Schr\"{o}dinger flow, so in $\mathcal{P}_{NL}$, see (\ref{randomop}), we can effectively assume that $P_{\ll N^\alpha}u$ is replaced by either a linear flow or a smooth function in $X^{1-,\frac{1}{2}+}$. Hence (\ref{normintro}) follows from large deviation estimates for multilinear Gaussians and a $T^*T$ (random matrix) argument like the one of Bourgain \cite{Bourgain}, and the estimate for $Y$ follows from standard contraction mapping arguments. See Sections \ref{structuresol} and \ref{multiest} for details.
\subsection{Further discussions} The notion of probabilistic scaling introduced in Section \ref{probsc} is a general philosophy and is not specific to (\ref{nls}); in fact it can be extended to more general situations. These include, but are not limited to, the following ones.
\subsubsection{Wave equations} For the wave equation (say with a power nonlinearity as in (\ref{nlsex})) we can apply the same heuristics as in Section \ref{probsc}. However, due to the gain of one derivative in the Duhamel fomrmula, instead of (\ref{naivedeter}) we essentially have that
\begin{equation}\label{naivedeter2}\mathcal{F}_xu^{(1)}(k)\sim\frac{1}{\langle k\rangle}\sum_{\substack{k_1-\cdots+k_{2r+1}=k\\|k_j|\lesssim N}}\frac{1}{\langle\Sigma\rangle}\prod_{j=1}^{2r+1}\widehat{u_{\mathrm{in}}}(k_j),\quad \Sigma=|k|-|k_1|+\cdots-|k_{2r+1}|.\end{equation}

Assume now $|k|\sim N$, then compared to (\ref{naivedeter}) we gain an extra factor $N^{-1}$ due to the antiderivative, while in the dimension counting argument we gain one less power of $N$ as $\Sigma$ is now linear instead of quadratic. In the deterministic setting this leads to the same scaling condition as the Schr\"{o}dinger equation, but in the probabilistic setting this  trade-off leads to a \emph{better} bound than  in the Schr\"{o}dinger case as the one-dimension disadvantage gets `square-rooted' by exploiting randomness as explained above. This then gives
\[|\mathcal{F}_xu^{(2)}(k)|\lesssim N^{-(2r+1)\alpha-1}N^{rd-\frac{1}{2}},\] which leads to a \emph{lower} probabilistic scaling threshold, namely  $s_p^{\mathrm{wave}}=-\frac{3}{4r}$.

However, unlike Schr\"{o}dinger, there is also a  `high-high to low' interaction, namely $|k|\sim 1$ in (\ref{naivedeter2}), that needs to be addressed. A similar calculation using randomness and counting bounds yields heuristically that
\[|\mathcal{F}_x{u^{(2)}}(k)|\lesssim N^{-(2r+1)\alpha}N^{rd-\frac{1}{2}},\] which leads to the restriction $s\geq s_p':=\frac{-d-1}{2(2r+1)}$. Thus it is reasonable to conjecture that the wave equation is almost surely locally well-posed in $H^s\times H^{s-1}$ for 
\[s>\max(s_p^{\mathrm{wave}},s_p')=\max\big(-\frac{3}{4r} , -\frac{d+1}{4r+2}\big).\] in particular when $(d,r)=(3,1)$ the conjectured threshold is $H^{-\frac{2}{3}}$, which is below $H^{-\frac{1}{2}-}$ where the Gibbs measure is supported, consistent with the recent positive result \cite{BDNY22}.

\subsubsection{Other dispersion relations and/or nonlinearities} For more general dispersion relations on $ \mathbb T^d$, say $\Lambda=\Lambda(k)$, the only thing above that changes is the counting bound for the set
\[S_m=\{(k_1,\cdots,k_{2r+1}):k_1-\cdots+k_{2r+1}=k,\,\,\Sigma:=\Lambda(k)-\Lambda(k_1)+\cdots-\Lambda(k_{2r+1})\in[m,m+1]\}.\] In contrast to parabolic equations (see \cite{Hairer}) where the exact form of  the elliptic operator is irrelevant once the order is fixed, here the properties of $S_m$ depend crucially on the choice of $\Lambda$, and have to be analyzed on a case by case basis. For simple dispersion relations like Schr\"{o}dinger, wave or gravity water wave (where $\Lambda(k)=\sqrt{|k|}$) this is doable, but when $\Lambda$ gets more complicated (say a high degree polynomial), determining the optimal local well-posedness threshold requires getting sharp bounds for $\#S_m$, which in itself may be a hard problem in analytic number theory.

For derivative nonlinearities, the scaling heuristics can still be carried out and the value of $s_p$ can be calculated in the same way as before (since such heuristics essentially take into account only the high-high interactions). However the actual almost sure well-posedness threshold may be strictly higher than $s_p$ due to high-low interactions and derivative loss (in the same way that the deterministic theory for quasilinear equations does not quite reach scaling, see e.g. \cite{KRS,ST}), which may be worth looking at first in some simple models. There is also the possibility of exponential nonlinearities but they are more of an `endpoint' nature and will not be discussed here.

\subsubsection{Stochastic equations} We may also consider wave and heat equations with additive noise (Schr\"{o}dinger is also possible but has worse behavior), say of form 
\begin{equation}\label{additive}(\partial_t^2-\Delta)u=|u|^{2r}u+\zeta,\quad\mathrm{or}\quad (\partial_t-\Delta)u=|u|^{2r}u+\zeta,
\end{equation} where $\zeta$ is the spacetime white noise which is essentially (after discretizing the time Fourier variable)
\[\zeta=\sum_{k,\xi}g_{k,\xi}e^{i(k\cdot x+\xi t)},\] where $g_{k,\xi}$ are independent normalized Gaussians.

The heat case of (\ref{additive}) has been studied extensively, see \cite{IW} and the references in Section \ref{results2}. In this case we can confirm that the scaling heuristics of Section \ref{probsc} are consistent with that of \cite{Hairer}. Indeed, note that for (\ref{additive}) the linear evolution $e^{it\Delta}u_{\mathrm{in}}$ in Section \ref{probsc} is replaced by the linear noise term
\[\psi(t)=\int_0^te^{(t-t')\Delta}\zeta(t')\,\mathrm{d}t'\sim\sum_{k,\xi}\frac{g_{k,\xi}}{|k|^2+|\xi|}e^{i(k\cdot x+\xi t)},\] which belongs to $C_t^0H_x^{s-}$ for $s=-\frac{d}{2}+1$. The goal would then be to guarantee that the second iteration
\[u^{(1)}(t)=\int_0^te^{(t-t')\Delta}(|\psi(s)|^{2r}\psi(t'))\,\mathrm{d}t'\] belongs to the same space. By similar arguments, this time also taking into account the time Fourier variable, we can show that this leads to the restriction $r(d-2)<2$, which coincides with the subcriticality condition introduced in \cite{Hairer} in the case of (\ref{additive}).

For the wave case of (\ref{additive}), similar calculations lead to the subcriticality condition $r(d-2)<\frac{3}{2}$, which is consistent with the results in \cite{GKO,GKO2}. In both cases, due to the particular choice of white noise, the high-to-low interactions studied in (1) above give the same condition on $(d,r)$.
\subsubsection{Other geometries} Considering more general geometries in addition to $\mathbb{T}^d$, will lead to different scenarios. For compact manifold the canonical randomization would be based on the spectral expansion of the Laplacian, in which case the probabilistic scaling depends on the \emph{global geometry} of the underlying manifold. This is because the randomized data is \emph{not} localized and has the same amplitude at each point of the domain. In comparison, the \emph{deterministic scaling} threshold does not depend on the geometry because it corresponds to the data zoomed out at a point, which is localized.

For non-compact manifolds (say $\mathbb{R}^d$), the canonical randomization based on Laplacian eigenfunctions (i.e. $e^{i\xi\cdot x}$) would lead to initial data with infinite $L^2$ mass, which is not compatible\footnote{It is however compatible with the wave equation due to finite speed of propagation; in particular the result of \cite{BDNY22} is expected to be true also for the Gibbs measure on $\mathbb{R}^3$. We will not discuss this here but see \cite{BDNY22}.} with the Schr\"{o}dinger equation (\ref{nls}). There is another kind of ``Wiener randomization" based on dividing the Fourier space into unit boxes and randomizing on each box (see e.g. \cite{ZhF, LM, BOP}), which produces localized initial data (which is not preserved by the linear flow). In the $\mathbb{R}^d$ case, this leads to the the critical threshold $s_{p}=-3/(4r)$ which is lower than $\mathbb{T}^d$.
\smallskip

\subsection{Notations and choice of parameters}\label{notations} In this section we collect some of the notations and conventions that will be used in the proof. Throughout the paper, the space and time Fourier transforms will be respectively fixed as
\begin{equation}\label{fouriert}(\mathcal{F}_xu)(k)=\frac{1}{(2\pi)^2}\int_{\mathbb{T}^2}e^{-ik\cdot x}f(x)\,\mathrm{d}x,\quad (\mathcal{F}_tu)(\xi)=\frac{1}{2\pi}\int_{\mathbb{R}}e^{-i\xi t}f(x)\,\mathrm{d}t.
\end{equation} We will be working in $(k,t)$ or $(k,\xi)$ variables, instead of the $x$ variable; so we will abbreviate $(\mathcal{F}_xu)(k)$ simply as $u_k$, and will abuse notations and write $u=u_k(t)$. The symbol $\widehat{u}$ will always represent time Fourier transform (or, for the second formula in (\ref{twistfourier2}) below, the corresponding two-dimensional time Fourier transform), so $(\mathcal{F}_{t,x}u)(k,\xi)=\widehat{u_k}(\xi)$.

Let the space mean $\mathcal{A}$ be defined by $\mathcal{A}u=  (\mathcal{F}_xu)(0) = u_0$ (this may depend on time $t$ if $u$ does). Define the twisted spacetime Fourier transform
\begin{equation}\label{twistfourier}\widetilde{u}_k(\lambda)=\widetilde{u}(k,\lambda)=\widehat{u_k}(\lambda-|k|^2).
\end{equation} We also need to study functions $h_{kk^*}(t)$ of variables $k,k^*\in\mathbb{Z}^2$ and $t\in\mathbb{R}$, and $\mathfrak{h}_{kk'}(t,t')$ of variables $k,k'\in\mathbb{Z}^2$ and $t,t'\in\mathbb{R}$; for these we will define
\begin{equation}\label{twistfourier2}\widetilde{h}_{kk^*}(\lambda)=\widehat{h_{kk^*}}(\lambda-|k|^2),\quad\mathrm{and}\quad \widetilde{\mathfrak{h}}_{kk'}(\lambda,\lambda')=2\pi \, \widehat{\mathfrak{h}_{kk'}}(\lambda-|k|^2,|k'|^2-\lambda'),
\end{equation}where $\lambda$ and $\lambda'$ are Fourier variables corresponding to $t$ and $t'$ respectively.

Recall that $\langle k\rangle=\sqrt{|k|^2+1}$, and $\mathbf{1}_P$ is the indicator function. The cardinality of a finite set $E$ will be denoted by $|E|$ or by $\#E$. We will be using smooth cutoff functions $\chi=\chi(z)$ which equal $1$ for $|z|\leq 1$ and equal $0$ for $|z|\geq 2$. For any Schwartz function $\varphi$ and any $0<\tau\ll 1$, we will define $\varphi_{\tau}(t)=\varphi(\tau^{-1}t)$.

For a complex number $z$ define $z^+:=z$ and $z^- :=\overline{z}$; we will also use the notation $z^{\iota}$ where $\iota$ will always be $\pm$. In the proof we will encounter tuples $(k_1,\cdots,k_n)$, or maybe $(k_1^*,\cdots,k_n^*)$, with associated signs $\iota_1,\cdots,\iota_n\in\{\pm\}$; they are usually linked by some equation $\iota_1k_1+\cdots+\iota_nk_n=d$ or $\iota_1|k_1|^2+\cdots+\iota_n|k_n|^2=\alpha$,
 where $d$ and $\alpha$ are given, or by some expression $g_{k_1^*}^{\iota_1}\cdots g_{k_n^*}^{\iota_n}$. 
\begin{df}\label{dfpairing}
In the above context we say $(k_i,k_j)$ is a \emph{pairing} in $\{k_1,\cdots,k_n\},$ if $k_i=k_j$ and $\iota_i=-\iota_j$. We say a pairing is \emph{over-paired} if $k_i=k_j=k_\ell$ for some $\ell\not\in\{i,j\}$. Pairings and over-pairings in $\{k_1^*,\cdots,k_n^*\}$ are defined similarly.
\end{df}   For example, suppose $k=k_1-k_2+k_3+d$. If $k_1=k_2$, then $(k_1,k_2)$ is a pairing in $\{k_1,k_2,k_3\}$; if $k=k_1$ then $(k,k_1)$ is a pairing in $\{k,k_1,k_2,k_3\}$. If $k=k_1=k_2\neq k_3$, then $(k_1,k_2)$ is over-paired if considered as a pairing in $\{k,k_1,k_2,k_3\}$, but not if considered a pairing in $\{k_1,k_2,k_3\}$.

Recall Remark \ref{powerof2} (2) that for the truncation $\Pi_N$ defined in (\ref{truncN}), $N$ will be a power of two that is also $\gtrsim 1$. The same applies to other capital letters like $M$, $L$, $R$, etc.. Define also $\Pi_N^{\perp}=\mathrm{Id}-\Pi_N$ and $\Delta_N=\Pi_N-\Pi_{\frac{N}{2}}$, so that
\begin{equation*}(\Delta_Nu)_k=\mathbf{1}_{N/2<\langle k\rangle\leq N}\cdot u_k.\end{equation*} Let $\mathcal{V}_N$ and $\mathcal{V}_N^\perp$ be the ranges of $\Pi_N$ and $\Pi_N^\perp$. For $N_1,\cdots,N_n$, we will define $\max^{(j)}(N_1,\cdots,N_n)$ to be the $j$-th maximal element among them, and denote it by $N^{(j)}$. 
\begin{df} \label{borel}
 For any $N$ as above, we denote by $\mathcal{B}_{\leq N}$  the $\sigma$-algebra generated by the random variables $g_k$ for $\langle k\rangle\leq N$, and by $\mathcal{B}_{\leq N}^+$  the smallest $\sigma$-algebra containing both $\mathcal{B}_{\leq N}$ and the $\sigma$-algebra generated by the random variables $|g_k|^2$ for $k\in\mathbb{Z}^2$.
\end{df} 

Recall that $\varepsilon$ is fixed by Remark \ref{fixep}. Let $1\gg\delta_0\gg\delta$ be two fixed small positive constants depending on $r$ and $\varepsilon$ (think of $\delta_0=\delta^{1/50}$). Define the parameters
\begin{equation}
\label{defparam}
\gamma=\delta^{\frac{3}{4}},\quad \gamma_0=\delta^{\frac{5}{4}},\quad\kappa=\delta^{-4},\quad b=\frac{1}{2}+\delta^4,\quad b_1=b+\delta^4,\quad b_2=b-\delta^6,\quad a_0=2b-10\delta^6,
\end{equation} then we have the following hierarchy:
\begin{equation}\label{hierarchy}\varepsilon\gg\delta_0\gg
\gamma\gg\delta\gg\gamma_0\gg \delta\gamma_0\gg b-\frac{1}{2}=b_1-b=\kappa^{-1}\gg\delta^6.
\end{equation} Denote by $\theta$ any positive quantity that is small enough depending on $\delta$ (for example $\theta\ll\delta^{50}$). This $\theta$ may take different values at different places. Let $C$ be any large absolute constant depending only on $r$, and $C_\theta$ be any  constant depending on $\theta$. Unless otherwise stated, the constants in the $\lesssim$, $\ll$ and $O(\cdot)$ symbols will depend on $C_\theta$. Finally,  if some statement $S$ about a random variable holds with probability $\mathbb{P}(S)\geq 1-C_\theta e^{-A^\theta}$ for some quantity $A>0$ and with given $\theta$ and $C_{\theta}$ independent of $A$,  we will say this $S$ is \emph{$A$-certain}. 

The rest of the paper is organized as follows. In Section \ref{prep0} we introduce the gauge transform and reduce to a favorable nonlinearity, and define the norms that will be used in the proof below. In Section \ref{structuresol} we identify the precise structure of the solution according to the ideas of Section \ref{method}, and reduce the local well-posedness to some multilinear estimates, namely Proposition \ref{multi0}. In Section \ref{prep2} we then set up the necessary tools (large deviation and counting estimates) needed in the proof of Proposition \ref{multi0}, and Section \ref{multiest} contains the proof itself. Finally in Section \ref{global} we apply an adapted version of Bourgain's argument to extend local solutions to global ones and finish the proof of Theorem \ref{main}.

\subsection{Acknowledgment} The second author thanks Hendrik Weber for helpful comments regarding references on the $\phi^4$ model.

\section{Equations, measures and norms}\label{prep0} In this section we make some preparations for the proof of Theorem \ref{main}. These include definitions of Wick ordering and gauge transform, properties of Gibbs and Gaussian measures and their finite dimensional truncations, and choices of function and operator norms and linear estimates. 
\subsection{Wick ordering and a gauge transform} We start by defining the Wick ordering and the gauge transform. Consider a general polynomial $\mathcal{M}_n(u)$ or $\mathcal{H}_n(u)$ of degree $n$, defined by
\begin{equation}\label{orderm}[\mathcal{M}_n(u)]_k=\sum_{\iota_1k_1+\cdots+\iota_n k_n=k}a_{kk_1\cdots k_n}u_{k_1}^{\iota_1}\cdots u_{k_n}^{\iota_n},
\end{equation} \begin{equation}\label{orderm2}[\mathcal{H}_n(u)]_{kk'}=\sum_{\iota_1k_1+\cdots+\iota_n k_n+\iota k'=k}a_{kk'k_1\cdots k_n}u_{k_1}^{\iota_1}\cdots u_{k_n}^{\iota_n},
\end{equation} where $a_{kk_1\cdots k_n}$ and $a_{kk'k_1\cdots k_n}$ are constants. Recall the definition of pairings in Definition \ref{dfpairing}.
\begin{df}
 We say the polynomial in \eqref{orderm} is \emph{input-simple}, if $a_{kk_1\cdots k_n}=0$ unless each pairing in $\{k_1,\cdots,k_n\}$ is over-paired. Similarly we say it is \emph{simple}, if $a_{kk_1\cdots k_n}=0$ unless each pairing in $\{k,k_1,\cdots,k_n\}$ is over-paired, and we say the polynomial in (\ref{orderm2}) is simple, if $a_{kk'k_1\cdots k_n}=0$ unless each pairing in $\{k,k',k_1,\cdots,k_n\}$ is over-paired. These notions also apply to multilinear forms.
 \end{df}
For $m: =\mathcal{A}|u|^2$, define the following polynomials of degree $n\in\{2p,2p+1\}$ (this $u$ may also be replaced by $v$):
\begin{equation}\label{pairfree}
\begin{aligned}:\mathrel{|u|^{2p}}:&=\sum_{j=0}^p(-1)^{p-j}{p\choose j}\frac{m^{p-j}p!}{j!}|u|^{2j},\\
:\mathrel{|u|^{2p}u}:&=\sum_{j=0}^p(-1)^{p-j}{{p+1}\choose {p-j}}\frac{m^{p-j}p!}{j!}|u|^{2j}u.
\end{aligned}
\end{equation} We will see in the proof of Proposition \ref{simpleterms} that each of these is input-simple.

Define a gauge transform $v_N=\mathcal{G}_Nu_N$ associated with (\ref{truncnls}) by \begin{equation}\label{gauge}v_N(t)=u_N(t)\cdot\exp\bigg((r+1)i\int_0^t\mathcal{A}[W_N^{2r}(u_N)]\,\mathrm{d}t'\bigg).
\end{equation} Then $u_N$ solves (\ref{truncnls}) if and only if $v_N$ solves the gauged equation
\begin{equation}\label{gauged}\left\{
\begin{split}(i\partial_t+\Delta)v_N&=\Pi_N\mathcal{Q}_N(v_N),\\
v_N(0)&=\Pi_Nu_{\mathrm{in}},
\end{split}
\right.
\end{equation} where
\begin{equation}\label{nonlinN}\mathcal{Q}_N(v)=W_N^{2r+1}(v)-(r+1)\mathcal{A}[W_N^{2r}(v)]v.
\end{equation} Since the gauge transform does not change the $t=0$ data, we will write $v_{\mathrm{in}}=u_{\mathrm{in}}$. The inverse of $\mathcal{G}_N$ is given by \begin{equation}\label{invgauge}u_N(t)=v_N(t)\cdot\exp\bigg(-(r+1)i\int_0^t\mathcal{A}[W_N^{2r}(v_N)]\,\mathrm{d}t'\bigg),
\end{equation} since by (\ref{wickpoly}), if $v=e^{i\alpha}u$ where $\alpha\in\mathbb{R}$, then $W^{n}(u)=W^{n}(v)$ for even $n$ and $W^n(u)=e^{i\alpha}W^n(v)$ for odd $n$. Now assume $v_N$ is a solution to (\ref{gauged}). Let $m_N$ be the truncated mass, which is conserved under (\ref{gauged}),
\begin{equation}\label{truncmass}m_N:=\mathcal{A}|v_N|^2=\sum_{\langle k\rangle\leq N}|(u_{\mathrm{in}})_k|^2,
\end{equation} and let $m_N^*:=m_N-\sigma_N$ where $\sigma_N$ is as in \eqref{expect}. Note that $m_N$ and $m_N^{\ast}$ are random terms if $u_{\mathrm{in}}=f(\omega)$ as in (\ref{random}). The following proposition give us a useful formula for $\mathcal{Q}_N$.
\begin{prop}\label{simpleterms} We have
\begin{equation}\label{formulaqn}\mathcal{Q}_N(v_N)=\sum_{l=0}^r{{r+1}\choose {r-l}}\frac{(m_N^*)^{r-l}r!}{l!}\mathcal{N}_{2l+1}(v_N),
\end{equation} where
\begin{equation}\label{quintcubic}\mathcal{N}_{2l+1}(v)=:\mathrel{|v|^{2l}v}:-(l+1)(\mathcal{A}:\mathrel{|v|^{2l}}:)v.
\end{equation} Here $\mathcal{N}_{2l+1}$ is a simple polynomial of degree $2l+1$. By standard procedure, we can define a $(2l+1)$-multilinear form, which we still denote by $\mathcal{N}_{2l+1}$, such that it reduces to $\mathcal{N}_{2l+1}(v)$ when all inputs equal to $v$.
\end{prop}
\begin{proof} First we prove (\ref{formulaqn}). By the definition of $\mathcal{Q}_N(v)$, see (\ref{nonlinN}), it will suffice to obtain that
\begin{equation}\label{prop3.2:eq1}
 W_N^{2r+1}(v_N) =\sum_{l=0}^r{{r+1}\choose {r-l}}\frac{(m_N^*)^{r-l}r!}{l!} :\mathrel{|v_N|^{2l}v_N}:
\end{equation} and
\begin{equation}\label{prop3.2:eq2}
(r+1)W_N^{2r}(v_N) =\sum_{l=0}^r{{r+1}\choose {r-l}}\frac{(m_N^*)^{r-l}r!}{l!}(l+1):\mathrel{|v_N|^{2l}}:.
\end{equation}

By the definition of $:\mathrel{|v|^{2l}v}:$, see (\ref{pairfree}), and combinatorial identities, we have
\begin{align}
\text{RHS of ({\ref{prop3.2:eq1}})} & = \sum_{l=0}^r  {r+1 \choose r-l} (m_N^*)^{r-l}r! \sum_{k=0}^l (-1)^{l-k} {l+1\choose l-k} \frac{m_N^{l-k}}{k!}\mathrel{|v_N|^{2k} v_N}\nonumber\\
\label{prop3.2:eq3}&= \sum_{k=0}^r (-1)^{r-k} {r+1\choose r-k}\frac{r!}{k!}\mathrel{|v_N|^{2k} v_N}
\sum_{l=k}^r {r-k\choose l-k}m_N^{l-k} (-m_N^*)^{r-l},
\end{align}
which implies (\ref{prop3.2:eq1}) due to binomial expansion.

Similarly we can calculate
\begin{align}
\text{RHS of ({\ref{prop3.2:eq2}})} & = \sum_{l=0}^r  {r+1 \choose r-l} (m_N^*)^{r-l}r! \, (l+1)\, \sum_{k=0}^l (-1)^{l-k} {l\choose k} \frac{m_N^{l-k}}{k!}\mathrel{|v_N|^{2k}}\nonumber\\
&= (r+1)\sum_{k=0}^r (-1)^{r-k} {r\choose k}\frac{r!}{k!}\mathrel{|v_N|^{2k}}
\sum_{l=k}^r {r-k\choose l-k}m_N^{l-k} (-m_N^*)^{r-l},\label{prop3.2:eq4}
\end{align}
which implies (\ref{prop3.2:eq2}).

Next we prove that $:\mathrel{|v|^{2p}}:$ and $:\mathrel{|v|^{2p}v}:$ are input-simple. Working in Fourier space, for any monomial
\[\mathfrak{X}:=(v_{k_1})^{a_1}(\overline{v_{k_1}})^{b_1}\cdots (v_{k_n})^{a_n}(\overline{v_{k_n}})^{b_n},\] where the $k_j$'s are different, $a_j$ and $b_j$ are nonnegative integers, we will calculate the coefficient of $\mathfrak{X}$ in the polynomial $:\mathrel{|v|^{2p}}:$ and $:\mathrel{|v|^{2p}v}:$, and will prove that this coefficient is zero provided $a_1=b_1=1$. Now clearly the coefficient of $\mathfrak{X}$ in $|v|^{2p}$ and $|v|^{2p}v$, denoted by $[\mathfrak{X}](|v|^{2p})$ and $[\mathfrak{X}](|v|^{2p}v)$, are
\begin{equation*}[\mathfrak{X}](|v|^{2p})=\frac{(p!)^2}{a_1!\cdots a_n!b_1!\cdots b_n!},\quad [\mathfrak{X}](|v|^{2p}v)=\frac{p!(p+1)!}{a_1!\cdots a_n!b_1!\cdots b_n!},
\end{equation*} under the assumptions $b_1+\cdots+b_n=p$ and $a_1+\cdots+a_n=p$ (or $a_1+\cdots+a_n=p+1$). Recall that $m=\mathcal{A}|v|^2$, we can calculate that
\begin{align}[\mathfrak{X}](:\mathrel{|v|^{2p}}:)&=\sum_{l=0}^p(-1)^{p-l}\frac{p!}{l!}{p\choose l}(p-l)!(l!)^2\sum_{c_1+\cdots+c_n=p-l}\, \prod_{s=1}^n\frac{1}{c_s!(a_s-c_s)!(b_s-c_s)!}\nonumber\\
&=(-1)^p(p!)^2 \sum_{l=0}^p(-1)^{l}\sum_{c_1+\cdots+c_n=p-l}\, \prod_{s=1}^n\frac{1}{c_s!(a_s-c_s)!(b_s-c_s)!}.\label{totalcoef}
\end{align}
 Now suppose $a_1=b_1=1$, then $c_1\in\{0,1\}$; clearly the terms for $l$ and with $c_1=0$ exactly cancel the terms for $l+1$ and with $c_1=1$, so $[\mathfrak{X}](:\mathrel{|v|^{2p}}:)=0$. Similarly we can prove $[\mathfrak{X}](:\mathrel{|v|^{2p}v}:)=0$.

Finally, we prove that $\mathcal{N}_{2p+1}(v)=:\mathrel{|v|^{2p}v}:-(p+1)(\mathcal{A}:\mathrel{|v|^{2p}}:)v$ is simple. By definition it suffices to prove that
\begin{equation*}\mathcal{A}(\mathcal{N}_{2p+1}(v)\overline{v})=\mathcal{A}(\overline{v}:\mathrel{|v|^{2p}v}:)-(p+1)m \, \mathcal{A}:\mathrel{|v|^{2p}}:
\end{equation*} is input-simple. We will actually show that this equals $\mathcal{A}:\mathrel{|v|}^{2p+2}:$ whence the result will follow. In fact, by (\ref{pairfree}) we have
\begin{equation*}
\begin{split}\mathcal{A}:\mathrel{|v|^{2p+2}}:&=\sum_{l=0}^{p+1}(-1)^{p-l+1}{{p+1}\choose l}\frac{(p+1)!}{l!}m^{p-l+1}\mathcal{A}|v|^{2l},\\
\mathcal{A}(\overline{v}:\mathrel{|v|^{2p}v}:)&=\sum_{l=1}^{p+1}(-1)^{p-l+1}{{p+1}\choose l}\frac{p!}{(l-1)!}m^{p-l+1}\mathcal{A}|v|^{2l},\\
-(p+1)m\, \mathcal{A}:\mathrel{|v|^{2p}}:&=\sum_{l=0}^{p}(-1)^{p-l+1}(p+1){{p}\choose l}\frac{p!}{l!}m^{p-l+1}\mathcal{A}|v|^{2l},
\end{split}
\end{equation*} so the first line equals the sum of the second and third lines by direct calculation.
\end{proof}
\begin{rem}\label{proper} Later on we will consider general multilinear forms $\mathcal{N}_{n}$ which are simple, and can be written as
\begin{equation}\label{multiform}[\mathcal{N}_{n}(v^{(1)},\cdots,v^{(n)})]_{k}=\sum_{\iota_1k_1+\cdots +\iota_nk_{n}=k}a_{kk_1\cdots k_{n}}(v_{k_1}^{(1)})^{\iota_1}\cdots (v_{k_{n}}^{(n)})^{\iota_n}.
\end{equation} 
We may assume the coefficient $a_{kk_1\cdots k_{n}}$ is symmetric in the $k_j$'s for which $\iota_j=+$, and also symmetric in the $k_j$'s for which $\iota_j=-$. Moreover, we assume that this coefficient only depends on the \emph{set of pairings} among $\{k,k_1,\cdots,k_{n}\}$.

The multilinear form $\mathcal{N}_{2l+1}$ corresponding to (\ref{quintcubic}) satisfies the above properties, and we will assume without any loss of generality that $\iota_j=+$ (i.e. $\mathcal{N}_{2l+1}$ is linear in $v^{(j)}$) for $j$ odd, and $\iota_j=-$ (i.e. $\mathcal{N}_{2l+1}$ is conjugate linear in $v^{(j)}$) for $j$ even.
\end{rem}
\subsection{Finite and infinite dimensional measures} We now summarize some properties of the infinite dimensional and finite dimensional (or truncated) Gaussian and Gibbs measures, that will be used later in the proof.

Recall that $\mathcal{V}_N$ and $\mathcal{V}_N^\perp$ are respectively the ranges of the projections $\Pi_N$ and $\Pi_N^\perp$. We will identify $\mathcal{V}$ with $\mathcal{V}_N\times\mathcal{V}_N^\perp$. Let $\mathrm{d}\rho_N$ and $\mathrm{d}\rho_N^\perp$ be the Gaussian measures defined on $\mathcal{V}_N$ and $\mathcal{V}_N^\perp$ respectively, such that $\mathrm{d}\rho=\mathrm{d}\rho_N\times\mathrm{d}\rho_N^\perp$. Define the measures $\mathrm{d}\mu_N^\circ$ on $\mathcal{V}_N$ and $\mathrm{d}\mu_N$ on $\mathcal{V}$ by
\begin{equation}\label{meastrunc}\mathrm{d}\mu_N^\circ=Z_N^{-1}e^{-V_N[u]}\,\mathrm{d}\rho_{N},\quad \mathrm{d}\mu_N=Z_N^{-1}e^{-V_N[u]}\,\mathrm{d}\rho;\quad Z_N=\int_{\mathcal{V}_N}e^{-V_N[u]}\,\mathrm{d}\rho_{N}(u),
\end{equation} then we have that $\mathrm{d}\mu_N=\mathrm{d}\mu_N^\circ\times\mathrm{d}\rho_N^\perp$. Recall also the measure $\mathrm{d}\mu$ defined in Proposition \ref{gibbsm}; all these are probability measures.
\begin{prop}\label{measurefact} When $N\to\infty$ we have $Z_N\to Z$, with $0<Z<\infty$. The sequence $Z_N^{-1}e^{-V_N[u]}$ converges to $Z^{-1}e^{-V[u]}$ almost surely, and also in $L^q(\mathrm{d}\rho)$ for any $1\leq q<\infty$. The measure $\mathrm{d}\mu_N$ converges to $\mathrm{d}\mu$ in the sense that the total variation of $\mu_N-\mu$ converges to $0$. Finally, the measure $\mathrm{d}\mu_N^\circ$ is invariant under the flows of (\ref{truncnls}) and (\ref{gauged}).
\end{prop}
\begin{proof} The convergence results are proved in \cite{OT}. The measure $\mathrm{d}\mu_N^\circ$ is invariant under (\ref{truncnls}), because the latter is a finite dimensional Hamiltonian system, and \[\mathrm{d}\mu_N^\circ(u_N)=\frac{1}{E_N}e^{-\mathcal{H}_N[u_N]-\mathcal{M}[u_N]}\,\mathrm{d}\mathcal{L}_N(u_N)\] is its Gibbs measure (weighted by another conserved quantity), where $E_N$ is some positive constant, $\mathcal{H}_N$ and $\mathcal{M}$ are as in (\ref{defham}), and $\mathrm{d}\mathcal{L}_N$ is the Lebesgue measure on the finite dimensional space $\mathcal{V}_N$.

To prove that $\mathrm{d}\mu_N^\circ$ is invariant under (\ref{gauged}), it suffices to show that it is preserved\footnote{For fixed time $t$, we can view $\mathcal{G}_N$ as a mapping from $\mathcal{V}_N$ to itself, by requiring $u_N$ to solve (\ref{truncnls}).} by the gauge transform $\mathcal{G}_N$. In fact, by (\ref{wickpoly}) and (\ref{defham}) we know $\mathcal{H}_N[u_N]=\mathcal{H}_N[v_N]$ and $\mathcal{M}[u_N]=\mathcal{M}[v_N]$, so it suffices to prove that $\mathcal{G}_N$ preserves the Lebesgue measure $\mathrm{d}\mathcal{L}_N$. Working in the coordinates $(r_k,\theta_k)_{\langle k\rangle\leq N}$ and $(r_k^*,\theta_k^*)_{\langle k\rangle\leq N}$, which are defined by $(u_N)_k=r_k e^{i\theta_k}$ and $(v_N)_k=r_k^* e^{i\theta_k^*}$, we can write the measure $\mathrm{d}\mathcal{L}_N$ as
\begin{equation}\label{measureN}\mathrm{d}\mathcal{L}_N=\prod_{\langle k\rangle\leq N}r_k\mathrm{d}r_k\mathrm{d}\theta_k.
\end{equation} If $v_N=\mathcal{G}_Nu_N$, then we have $r_k^*=r_k$ and $\theta_k^*=\theta_k+F((r_j,\theta_j)_{\langle j\rangle\leq N})$, where $F$ may also depend on $t$, but does not depend on $k$. Moreover, by (\ref{wickpoly}) and (\ref{gauge}) we know that $F$ actually depends only on $r_j$ and on the differences $\theta_j-\theta_\ell$, which are invariant under the mapping $\theta_k\mapsto\theta_k^*$. It then follows that the transformation $(r_k,\theta_k)\mapsto (r_k^*,\theta_k^*)$ preserves the measure (\ref{measureN}), by a simple calculation of its Jacobian.
\end{proof}
\subsection{Function spaces and linear estimates}\label{funcspace} From now on we will work with the equation (\ref{gauged}) with the nonlinearity defined by (\ref{formulaqn}) and (\ref{quintcubic}), which has the form (\ref{multiform}). Recall the well-known $X^{s,b}$ spaces (where $b$ may be replaced by $b_1$ or $b_2$)
\begin{equation}\label{xsb}\|u\|_{X^{s,b}}=\|\langle k\rangle^s\langle \lambda\rangle^b\,\widetilde{u}_k(\lambda)\|_{\ell_k^2L_\lambda^2}.
\end{equation} We will mostly consider $s=0$ and will denote $X^{0,b}=X^b$. In addition we introduce matrix norms which measure the functions $h=h_{kk^*}(t)$ and $\mathfrak{h}=\mathfrak{h}_{kk'}(t,s)$, namely 
\begin{align}
\label{matrixnorm1}\|h\|_{Y^{b}}&=\|\langle\lambda\rangle^{b}\, \widetilde{h}_{kk^*}(\lambda)\|_{\ell_{k^*}^2\to\ell_{k}^2L_\lambda^2},&\|\mathfrak{h}\|_{Y^{b,b}}&=\|\langle\lambda\rangle^{b}\langle \lambda'\rangle^{-b}\, \widetilde{\mathfrak{h}}_{kk'}(\lambda,\lambda')\|_{\ell_{k'}^2L_{\lambda'}^2\to\ell_{k}^2L_\lambda^2},\\
\label{matrixnorm2}
\|h\|_{Z^{b}}&=\|\langle\lambda\rangle^{b} \, \widetilde{h}_{kk^*}(\lambda)\|_{\ell_{k,k^*}^2L_\lambda^{2}},&\|\mathfrak{h}\|_{Z^{\widetilde{b},b}}&=\|\langle\lambda\rangle^{\widetilde{b}}\langle \lambda'\rangle^{-b} \, \widetilde{\mathfrak{h}}_{kk'}(\lambda,\lambda')\|_{\ell_{k,k'}^2L_{\lambda,\lambda'}^2},
\end{align} where $\widetilde{b}\in\{b,b_1\}$, $\|\cdot\|_{\ell_{k^*}^2\to\ell_{k}^2L_\lambda^2}$ and $\|\cdot\|_{\ell_{k'}^2L_\lambda'^2\to\ell_{k}^2L_\lambda^2}$ represent the operator norms of linear operators with the given kernels, for example
\begin{equation}\label{lptolp}\|\mathfrak{h}\|_{Y^{b,b}}=\sup\bigg\{\bigg\|\sum_{k'}\int\mathrm{d}\mu\cdot\langle \lambda\rangle^{b}\langle\lambda'\rangle^{-b}\, \widetilde{\mathfrak{h}}_{kk'}(\lambda,\lambda')y_{k'}(\lambda')\bigg\|_{\ell_k^2L_\lambda^2}\,:\, \|y_{k'}(\lambda')\|_{\ell_{k'}^2L_{\lambda'}^2}=1\bigg\}.
\end{equation} By definition we can verify that \begin{equation}\label{matrixnorm0}\|\mathfrak{h}\|_{Y^{b,b}}=\sup_{\|y\|_{X^b} \,=\,1}\bigg\|\sum_{k'}\int\mathrm{d}t'\cdot \mathfrak{h}_{kk'}(t,t')y_{k'}(t')\bigg\|_{X^{b}}.
\end{equation}For any of the above spaces, we can localize them in the standard way to a time interval $J$,
\begin{equation}\label{localize}\|u\|_{\mathcal{Z}(J)}=\inf\{\|v\|_{\mathcal{Z}}:v\equiv u\textrm{ on }J\}.
\end{equation} We will need the following simple estimates.
\begin{prop} The norms $\|h\|_{Y^{b}}$ and $\|\mathfrak{h}\|_{Y^{b,b}}$ do not increase, when $\widetilde{h}$ or $\widetilde{\mathfrak{h}}$ is multiplied by a function of $(k,\lambda)$, or a function of $k^*$ (or $(k',\lambda')$ for $\widetilde{\mathfrak{h}}$), which is at most $1$ in the $l^\infty L^\infty$ or $l^\infty$ norms.

Next, if $H$ is defined by
\begin{equation}\label{operatorbd4}\widetilde{H}_{kk^*}(\lambda) =\sum_{k'}\int \mathrm{d}\lambda'\cdot\widetilde{\mathfrak{h}}_{kk'}(\lambda,\lambda')\, \widetilde{h}_{k'k^*}(\lambda'),
\end{equation} where $\widetilde{\mathfrak{h}}_{kk'}(\lambda,\lambda')$ is supported in $|k-k'|\lesssim L$, then for any $\alpha>0$ we have \begin{equation}\label{operatorbd3}\bigg\|\bigg(1+\frac{|k-k^*|}{L}\bigg)^{\alpha}H\bigg\|_{Z^{b}}\lesssim \|\mathfrak{h}\|_{Y^{b,b}}\cdot\bigg\|\bigg(1+\frac{|k'-k^*|}{L}\bigg)^{\alpha}h\bigg\|_{Z^b}.
\end{equation}
\end{prop}
\begin{proof} The first statement follows directly from definition (\ref{lptolp}). Now let us prove (\ref{operatorbd3}). We may fix $k^*$ and by translation invariance, we may assume $k^*=0$. Relabeling  $\widetilde{H}_{k0}(\lambda) =:\widetilde{H}_k(\lambda)$ and $\widetilde{h}_{k'0}(\lambda) =:\widetilde{h}_{k'}(\lambda)$ we may decompose \begin{equation*}\widetilde{H}_k(\lambda)=\sum_{M\geq L}(\widetilde{H}^M)_k(\lambda);\quad (\widetilde{H}^M)_k(\lambda)=\left\{
\begin{aligned}&\mathbf{1}_{|k|\sim M}\widetilde{H}_k(\lambda),&M&> L,\\
&\mathbf{1}_{|k|\lesssim L}\widetilde{H}_k(\lambda),&M&= L,
\end{aligned}
\right.\end{equation*} and similarly for $\widetilde{h}$, so that we have
\begin{equation}\label{dydsum} \bigg\|\langle\lambda\rangle^b\bigg(1+\frac{|k|}{L}\bigg)^{\alpha}\widetilde{H}_{k}(\lambda)\bigg\|_{\ell_{k}^2L_\lambda^2}^2\sim\sum_{M\geq L}L^{-2\alpha}M^{2\alpha}\|\langle\lambda\rangle^b(\widetilde{H}^M)_k(\lambda)\|_{\ell_k^2L_\lambda^2}^2
\end{equation} and similarly for $h$. Since $\widetilde{\mathfrak{h}}$ is supported in $|k-k'|\lesssim L$, we have
\begin{equation*}|(\widetilde{H}^M)_k(\lambda)|\leq\sum_{M'\sim M}\bigg|\sum_{k'}\int\mathrm{d}\lambda'\cdot\widetilde{\mathfrak{h}}_{kk'}(\lambda,\lambda')(\widetilde{h}^{M'})_{k'}(\lambda')\bigg|,
\end{equation*} therefore
\begin{equation*}\|\langle \lambda\rangle^{b}(\widetilde{H}^M)_k(\lambda)\|_{\ell_k^2L_\lambda^2}^2\lesssim\|\langle \lambda\rangle^{b}\langle\lambda'\rangle^{-b}\widetilde{\mathfrak{h}}_{kk'}(\lambda,\lambda')\|_{\ell_{k'}^2L_{\lambda'}^2\to\ell_k^2L_\lambda^2}^2\sum_{M'\sim M}\|\langle \lambda'\rangle^{b}(\widetilde{h}^{M'})_{k'}(\lambda')\|_{\ell_{k'}^2L_\lambda'^2}^2,
\end{equation*} which, combined with (\ref{dydsum}), implies (\ref{operatorbd3}).
\end{proof}
Let $\chi$ be a smooth cutoff as in Section \ref{notations}, and define the time truncated Duhamel operator
\begin{equation}\label{duhameloper}\mathcal{I}F(t):=\chi(t)\int_0^te^{i(t-t')\Delta}\chi(t')F(t')\,\mathrm{d}t'.
\end{equation}
\begin{lem}\label{duhamelest} We have $2\, \mathcal{I}F(t)=\mathcal{J}F(t)-\chi(t)e^{it\Delta}\mathcal{J}F(0)$, where $\mathcal{J}$ is defined by
\begin{equation}\label{defjop}\mathcal{J}F(t)=\chi(t)\bigg(\int_{-\infty}^t-\int_t^{\infty}\bigg)e^{i(t-t')\Delta}\chi(t')F(t')\,\mathrm{d}t'.
\end{equation} Moreover we have the formula
\begin{equation}\label{trunckernel}\widetilde{\mathcal{J}F}(k,\lambda)=\int_{\mathbb{R}}\mathcal{J}(\lambda,\mu)\widetilde{F}(k,\mu)\,\mathrm{d}\mu,\quad |\partial_{\lambda,\mu}^\alpha\mathcal{J}(\lambda,\mu)|\lesssim_{\alpha,A}\frac{1}{\langle \lambda-\mu\rangle^{A}}\frac{1}{\langle \mu\rangle}.
\end{equation}
\end{lem} For the proof of Lemma \ref{duhamelest}, see the calculations in \cite{DNY}, Lemma 3.1.
\begin{prop}\label{sttime} Let $\varphi$ be any Schwartz function, recall that $\varphi_\tau(t)=\varphi(\tau^{-1}t)$ for any $0<\tau\ll1$. Then for any $u=u_k(t)$ and $\mathfrak{h}=\mathfrak{h}_{kk'}(t,t')$ we have
\begin{equation}\label{sttime1}\|\varphi_{\tau}\cdot u\|_{X^{s,b}}\lesssim {\tau}^{b_1-b}\|u\|_{X^{s,b_1}},\quad \|\varphi_{\tau}(t)\cdot \mathfrak{h}\|_{Z^{b,b}}\lesssim {\tau}^{b_1-b}\|\mathfrak{h}\|_{Z^{b_1,b}}\,,
\end{equation} provided that $u_k(0)=\mathfrak{h}_{kk'}(0,t')=0$.
\end{prop}
\begin{proof} Using the definition of the $Z^{\widetilde{b},b}$ norms and fixing the $(k',\lambda')$ variables, we can reduce the second inequality in (\ref{sttime1}) to the first, and by fixing $k$ and conjugating by the linear Schr\"{o}dinger flow, we can reduce the first to
\begin{equation*}\|\langle\xi\rangle^{b}(\widehat{\varphi_{\tau}}*\widehat{v})(\xi)\|_{L^2}\lesssim {\tau}^{b_1-b} \, \|\langle \eta\rangle^{b_1}\widehat{v}(\eta)\|_{L^{2}}\end{equation*} for $v$ satisfying $v(0)=0$. Let $\widehat{v}=g_1+g_2$ where
\[g_1(\xi)=\mathbf{1}_{|\xi|\geq {\tau}^{-1}}(\xi)\widehat{v}(\xi),\quad g_2(\xi)=\mathbf{1}_{|\xi|< {\tau}^{-1}}(\xi)\widehat{v}(\xi).\]We will prove that
\begin{equation}\label{st21}\|\langle\xi\rangle^{b}(\widehat{\varphi_{\tau}}*g_j)(\xi)\|_{L^2}\lesssim {\tau}^{b_1-b} \, \|\langle \eta\rangle^{b_1}\widehat{v}(\eta)\|_{L^{2}}
\end{equation} for $j\in\{1,2\}$. To prove (\ref{st21}) for $j=1$, we can reduce it to the $L^{2}\to L^2$ bound for the operator
\[g(\eta)\mapsto\int_{\mathbb{R}}R(\xi,\eta)g(\eta)\,\mathrm{d}\eta,\quad R(\xi,\eta)=\mathbf{1}_{|\eta|\geq {\tau}^{-1}}\cdot \tau\widehat{\varphi}(\tau(\xi-\eta))\frac{\langle\xi\rangle^{b}}{\langle \eta\rangle^{b_1}}.\] Since 
\[\mathbf{1}_{|\eta|\geq {\tau}^{-1}}\cdot\frac{\langle\xi\rangle^{b}}{\langle \eta\rangle^{b_1}}\lesssim {\tau}^{b_1-b}\frac{\langle T\xi\rangle^{b}}{\langle \tau\eta\rangle^{b_1}}\lesssim {\tau}^{b_1-b}\langle \tau(\xi-\eta)\rangle^{b},\] it follows from Schur's estimate that this $L^{2}\to L^2$ bound is at most\[\tau^{b_1-b}\|\tau\widehat{\varphi}(\tau\zeta)\langle \tau\zeta\rangle^{b}\|_{L_\zeta^1}\lesssim {\tau}^{b_1-b},\] which proves (\ref{st21}) for $j=1$.

To prove (\ref{st21}) for $j=2$, note that since $v(0)=0$ we have $\int_{\mathbb{R}}\widehat{v}(\eta)\,\mathrm{d}\eta=0$, so
\begin{align*}|(\widehat{\varphi_{\tau}}*g_2)(\xi)|&=\bigg|-\tau\widehat{\varphi}(\tau\xi)\int_{|\eta|\geq \tau^{-1}}\widehat{v}(\eta)\,\mathrm{d}\eta-\int_{|\eta|<\tau^{-1}}\tau\widehat{v}(\eta)\big[\widehat{\varphi}(\tau\xi)-\widehat{\varphi}(\tau(\xi-\eta))\big]\,\mathrm{d}\eta\bigg|\\
&\lesssim \tau\langle \tau\xi\rangle^{-4}\int_{\mathbb{R}}\min(1,|\tau\eta|)|\widehat{v}(\eta)|\,\mathrm{d}\eta,
\end{align*} and by H\"{o}lder's inequality we have
\[\int_{\mathbb{R}}\min(1,|\tau\eta|)|\widehat{v}(\eta)|\,\mathrm{d}\eta\lesssim \|\langle \eta\rangle^{b_1}\widehat{v}(\eta)\|_{L^{2}}\cdot\|\min(1,|\tau\eta|)\langle\eta\rangle^{-b_1}\|_{L^2}\lesssim {\tau}^{b_1- \frac{1}{2}}\|\langle \eta\rangle^{b_1}\widehat{v}(\eta)\|_{L^{2}}.\] Using also the elementary bound
\[\|\tau\langle \tau\xi\rangle^{-4}\langle\xi\rangle^{b}\|_{L^2}\lesssim {\tau}^{\frac{1}{2}-b},\] we deduce (\ref{st21}) for $j=2$, and hence (\ref{sttime1}).
\end{proof} 
\section{Structure of the solution: random averaging operators} \label{structuresol} We now fix a short time $0<\tau\ll 1$, and focus on the local theory for (\ref{gauged}), with initial data distributed according to the Gaussian measure $\mathrm{d}\rho_{N}$, on $J:=[-\tau,\tau]$. By definition, this is equivalent to considering (\ref{gauged}) with random initial data $u_{\mathrm{in}}=v_{\mathrm{in}}=f(\omega)$, which we will assume from now on, until the end of Section \ref{multiest}. Most functions that appear in the proof will be random (i.e. depend on $\omega$), whether or not we explicitly write $\omega$ in their expressions.
\subsection{The decomposition}\label{decompsol} We start by writing down the ansatz of the solution to (\ref{gauged}). Recall that the truncated mass $m_N$ defined in (\ref{truncmass}) and the corresponding $m_N^*$ are random variables given by
\begin{equation}\label{truncmass2}m_N=\sum_{\langle k\rangle\leq N}\frac{|g_k|^2}{\langle k\rangle^2},\quad m_N^*=\sum_{\langle k\rangle\leq N}\frac{|g_k|^2-1}{\langle k\rangle^2}.
\end{equation} Note that they are Borel functions of $|g_k|^2$ for $\langle k\rangle\leq N$. Let $\nu_N:=m_N^*-m_{\frac{N}{2}}^*$. By standard large deviation estimates we have\begin{equation}\label{massdiff}\mathbb{P}(|\nu_N|\geq AN^{-1})\leq Ce^{-C^{-1}A}
\end{equation}for any $A>0$, where $C$ is an absolute constant. In particular, by removing a set of measure $\leq C_\theta e^{-{\tau}^{-\theta}}$ (which will be done before proving any estimates) we may assume the following bounds, which are used below without any further mentioning:
\begin{equation}\label{simplebd}|g_k|\lesssim {\tau}^{-\theta}\langle k\rangle^\theta,\quad |m_N^*|\lesssim  {\tau}^{-\theta},\quad |\nu_N|\lesssim  {\tau}^{-\theta} N^{-1+\theta}.
\end{equation}

Our goal here is to obtain a quantitative estimate for the difference $y_N:=v_N-v_{\frac{N}{2}}$. By (\ref{gauged}), this $y_N$ satisfies the equation
\begin{equation}\label{nlsdiff}
\left\{
\begin{split}(i\partial_t+\Delta)y_N&=\Pi_N\mathcal{Q}_N(y_N+v_{\frac{N}{2}})-\Pi_{\frac{N}{2}}\mathcal{Q}_{\frac{N}{2}}(v_{\frac{N}{2}}),\\
y_N(0)&=\Delta_Nf(\omega).
\end{split}
\right.
\end{equation}
By (\ref{formulaqn}) we can rewrite the above equation as
\begin{equation}\label{nlsdiff2}
\left\{
\begin{aligned}
(i\partial_t+\Delta)y_N&=\sum_{l=0}^rc_{rl} \, (m_N^*)^{r-l}\bigg\{\Pi_N\big[\mathcal{N}_{2l+1}(y_N+v_{\frac{N}{2}})-\mathcal{N}_{2l+1}(v_{\frac{N}{2}})\big]+\Delta_N\mathcal{N}_{2l+1}(v_{\frac{N}{2}})\bigg\}\\
&+\sum_{j=0}^rc_{rl}\, \big[(m_{\frac{N}{2}}^*+\nu_{N})^{r-l}-(m_{\frac{N}{2}}^*)^{r-l}\big]\cdot\Pi_\frac{N}{2}\mathcal{N}_{2l+1}(v_{\frac{N}{2}}),\\
y_N(0)&=\Delta_Nf(\omega),
\end{aligned}
\right.
\end{equation} where $c_{rl}$ are constants that will not be important in the proof.

Define the set
\begin{equation}\label{setNL}\mathcal{K}:=\{(N,L)\in(2^{\mathbb{Z}})^2:2^{-1}\leq L< N^{1-\delta}\}.
\end{equation} For each $(N,L)\in\mathcal{K}$, we define the function $\psi_{N,L}$ as the solution to the (linear) equation
\begin{equation}\label{defpsi}
\left\{
\begin{aligned}
(i\partial_t+\Delta)\psi_{N,L}&=\sum_{l=0}^r(l+1) \, c_{rl}\, (m_N^*)^{r-l} \, \Pi_N\mathcal{N}_{2l+1}(\psi_{N,L},v_{L},\cdots,v_{L}),\\
\psi_{N,L}(0)&=\Delta_Nf(\omega).
\end{aligned}
\right.
\end{equation}
It is important to place $\psi_{N, L}$ in the first position of $\mathcal{N}_{2l+1}$ in (\ref{defpsi}), see Remark \ref{gamma2}. By linearity we have,
\begin{equation}\label{linearity}(\psi_{N,L})_k=\sum_{k^*}\, H_{kk^*}^{N,L}\, \frac{g_{k^*}(\omega)}{\langle k^*\rangle},
\end{equation} where for $\frac{N}{2}<\langle k^*\rangle\leq N$ and $\langle k\rangle\leq N$, $H_{kk^*}^{N,L}=\varphi_k$ is the $k$-th mode of the solution $\varphi$ to the equation
\begin{equation}\label{defpsi2}
\left\{
\begin{aligned}
(i\partial_t+\Delta)\varphi&=\sum_{l=0}^r(l+1)c_{rl}(m_N^*)^{r-l}\, \Pi_N\mathcal{N}_{2l+1}(\varphi,v_{L},\cdots,v_{L}),\\
\varphi(0)&=e^{ik^*\cdot x},
\end{aligned}
\right.
\end{equation} and for other $(k,k^*)\in(\mathbb{Z}^2)^2$ define $H_{kk^*}^{N,L}=0$. By definition these $H_{kk^*}^{N,L}$, as well as the $h_{kk^*}^{N,L}$ defined below, are $\mathcal{B}_{\leq N}$ measurable and $\mathcal{B}_{\leq L}^+$ measurable in the sense of Definition \ref{borel}.

\smallskip
For any $N$, let $L_0$ be the largest $L$ satisfying $(N,L)\in\mathcal{K}$. We further define
\begin{equation}\label{matrices}\zeta_{N,L}: =\psi_{N,L} -\psi_{N,\frac{L}{2}},\qquad h^{N,L}: =H^{N,L}-H^{N,\frac{L}{2}}; \qquad z_N: =y_N-\psi_{N,L_0}.
\end{equation} Note that $\psi_{N,\frac{1}{2}}=e^{it\Delta}(\Delta_Nf(\omega))$, and that $H_{kk^*}^{N,\frac{1}{2}}$ is $e^{-i|k|^2t}\mathbf{1}_{k=k^*}$ restricted to the frequency band $\frac{N}{2}<\langle k\rangle\leq N$. Moreover $z_N$ is $\mathcal{B}_{\leq N}$ measurable, $z_N(0)=0$ and satisfies the equation
\begin{align}
(i\partial_t+\Delta)z_N&=\sum_{l=0}^rc_{rl}(m_N^*)^{r-l}\cdot\Pi_N\big[\mathcal{N}_{2l+1}(z_N+\psi_{N,L_0}+v_{\frac{N}{2}})-\mathcal{N}_{2l+1}(v_{\frac{N}{2}})+\Delta_N\mathcal{N}_{2l+1}(v_{\frac{N}{2}})\big]\nonumber\\
&-\sum_{l=0}^rc_{rl}(m_N^*)^{r-l}\cdot\Pi_N\big[(l+1)\mathcal{N}_{2l+1}(\psi_{N,L_0},v_{L_0},\cdots,v_{L_0})\big]\nonumber\\
\label{equationz}&+\sum_{l=0}^rc_{rl}\big[(m_{\frac{N}{2}}^*+\nu_N)^{r-l}-(m_{\frac{N}{2}}^*)^{r-l}\big]\cdot\Pi_{\frac{N}{2}}\mathcal{N}_{2l+1}(v_{\frac{N}{2}}).
\end{align}
\begin{rem} With the above construction, if we let $v=\lim_{N\to\infty}v_N$ be the gauged version of the solution $u$ to (\ref{nls}), we then have
\begin{equation}\label{ansatz1}
v = e^{it\Delta}f(\omega)+\sum_{(N,L)\in\mathcal{K}}\zeta_{N,L}+z,\quad \text{where} \quad z=\sum_Nz_N.
\end{equation} This is the ansatz (\ref{fullansatz2}) in Section \ref{secansatz}, where $\zeta_{N,L}$ can be viewed as a random averaging operator $\mathcal{P}_{NL}$, whose kernel is essentially given by $h^{N,L}$, applied to the Gaussian free field $e^{it\Delta}f(\omega)$.
There are however two differences: (1) our $\mathcal{P}_{NL}$ is not exactly the one in (\ref{randomop}), but an infinite iteration of the latter, because (\ref{randomop}) has no smoothing effect; and (2) our $\mathcal{P}_{NL}$ is not exactly a Borel function of $(g_k)_{\langle k\rangle\leq L}$ as it also depends on $m_N^*$, but as it turns out this does not affect any probabilistic estimates, see Lemma \ref{largedev}.
\end{rem}
\subsection{The a priori bounds}\label{apriori} We now state the local well-posedness result for (\ref{gauged}). Its proof will occupy the rest of this section and Sections \ref{prep2} and \ref{multiest}. 
\begin{prop}\label{localmain} Recall the relevant constants defined in (\ref{defparam}), and that $\tau\ll 1$, $J=[-\tau,\tau]$. Then, ${\tau}^{-1}$-certainly, i.e. with probability bigger than or equal to $1-C_{\theta}e^{-{\tau}^{-\theta}},$ the following estimates hold for all $(N,L)\in\mathcal{K}$:
\begin{equation}\label{aprioriest}
\|h^{N,L}\|_{Y^{b}(J)}\leq L^{-\delta_0},\quad\|h^{N,L}\|_{Z^{b}(J)}\leq N^{\frac{1}{2}+\delta^\frac{5}{4}}L^{-\frac{1}{2}},\quad
\|z_N\|_{X^{b}(J)}\leq N^{-1+\gamma}.
\end{equation}
\end{prop}
\subsubsection{The extensions}\label{extension}
In proving Proposition \ref{localmain} we will restrict $z_N$ and $h^{N,L}$ to $J$ and construct extensions of these restrictions that are defined for all time. This has to be done carefully so as to maintain the correct independence properties. We define these extensions inductively, as follows.

First, let $z_1^\dagger(t):=z_1(t)\chi_{\tau}(t)$ and $ \psi_{N,\frac{1}{2}}^\dagger(t): =\chi(t)e^{it\Delta}(\Delta_Nf(\omega))$, and define $H^{N,\frac{1}{2},\dagger}$ accordingly. Suppose $M\geq 1$ is a dyadic number and that we have defined $z_N^\dagger$ for all $N\leq M$ and $h^{N,L,\dagger}$ for all $(N,L)\in\mathcal{K}$ and $L<M$, then for $L\leq M$ we may define
\begin{equation*}v_L^\dagger= \sum_{L'\leq L}y_{L'}^\dagger,\quad \text{ where } \quad y_L^\dagger=z_{L}^\dagger\, +\, \chi(t)e^{it\Delta}(\Delta_Lf(\omega))+\sum_{(L,R)\in\mathcal{K}}\zeta_{L,R}^{\dagger},
\end{equation*} which is acceptable since for $(L,R)\in\mathcal{K}$ we must have $R<L^{1-\delta}\leq M$, so $h^{L,R,\dagger}$ and $H^{L,R,\dagger}$ are well-defined, hence $\zeta_{L,R}^\dagger$ and $\psi_{L,R}^\dagger$ can be defined by (\ref{linearity}) and (\ref{matrices}).

\smallskip

Next, for $(N,L)\in\mathcal{K}$ and $L=M$, we can define $H^{N,M,\dagger}$ such that for $\frac{N}{2}<\langle k^*\rangle\leq N$ and $\langle k\rangle\leq N$, $H_{kk^*}^{N,M,\dagger}=\varphi_k^\dagger$ is the $k$-th mode of the solution $\varphi^\dagger$ to the equation
\begin{equation}\label{daggered}\varphi^{\dagger}(t)=\chi(t)e^{it\Delta}(e^{ik^*\cdot x})-i\chi_{\tau}(t)\sum_{l=0}^r(l+1)\,c_{rl}\,(m_N^*)^{r-l}\cdot\mathcal{I}\Pi_N\mathcal{N}_{2l+1}\big(\varphi^\dagger,v_{M}^\dagger,\cdots,v_{M}^\dagger\big),
\end{equation}provided this solution exists and is unique; otherwise simply define $H^{N,M,\dagger}=H^{N,M}\cdot\chi_\tau(t)$. This defines $H^{N,M,\dagger}$ and hence also $h^{N,M,\dagger}$, $\psi_{N,M}^\dagger$ and $\zeta_{N,M}^\dagger$.

Finally we will define $z_{2M}$. As $\psi_{2M,L_0}^\dagger$ is already defined, where $L_0\leq M$ is the largest $L$ such that $(2M,L)\in\mathcal{K}$, we can define $z_{2M}^\dagger$ to be the unique fixed point of the mapping
\begin{align}z\mapsto &-i\chi_{\tau}(t)\sum_{l=0}^rc_{rl}\, (m_{2M}^*)^{r-l}\cdot\mathcal{I}\Pi_{2M}\bigg\{\mathcal{N}_{2l+1}\big(z+\psi_{2M,L_0}^\dagger+v_{M}^\dagger\big)-\mathcal{N}_{2l+1}\big(v_{M}^\dagger\big)\bigg\}\nonumber\\
&+i\chi_{\tau}(t)\sum_{l=0}^r(l+1)c_{rl}\, (m_N^*)^{r-l}\cdot\mathcal{I}\Pi_{2M}\mathcal{N}_{2l+1}\big(\psi_{2M,L_0}^\dagger,v_{L_0}^{\dagger},\cdots,v_{L_0}^{\dagger}\big)
\nonumber\\&-i\chi_{\tau}(t)\sum_{l=0}^rc_{rl}\, \big[(m_M^*+\nu_{2M})^{r-l}-(m_M^*)^{r-l}\big]\cdot\mathcal{I}\Pi_{M}\mathcal{N}_{2l+1}\big(v_M^{\dagger}\big)\nonumber\\
&-i\chi_{\tau}(t)\sum_{l=0}^rc_{rl}\, (m_{2M}^*)^{r-l}\cdot\mathcal{I}\Delta_{2M}\mathcal{N}_{2j+1}\big(v_M^\dagger\big)\label{fixedpoint}
\end{align}
 on the set $\mathcal{Z}=\{z:\|z\|_{X^b}\leq (2M)^{-1+\gamma}\}$, provided that this mapping is a contraction mapping from $\mathcal{Z}$ to itself. If it is not a contraction mapping, then simply define $z_{2M}^\dagger=z_{2M}\cdot\chi_{\tau}(t)$. This completes the inductive construction. We may then easily verify that:
\begin{itemize}
\item The $z_N^\dagger$ and $h^{N,L,\dagger}$ we constructed indeed coincide with $z_N$ and $h^{N,L}$ on $J$;

\item The $z_N^\dagger$ is supported in $\langle k\rangle\leq N$, and $h_{kk^*}^{N,L,\dagger}$ is supported in $\langle k\rangle\leq N$ and $\frac{N}{2}<\langle k^*\rangle\leq N$;

\item The random variable $h^{N,L,\dagger}$ is $\mathcal{B}_{\leq L}^+$ measurable, and $z_N^\dagger$ and $h^{N,L,\dagger}$ are $\mathcal{B}_{\leq N}$ measurable;
\item All the above are smooth and compactly supported in time $t\in[-2,2]$.
\end{itemize}
We will prove Proposition \ref{localmain} by induction in $M$, but in the process we will need some auxiliary estimates. More precisely, we will prove the following result, which contains Proposition \ref{localmain}:
\begin{prop}\label{localmain2} Recall the relevant constants defined in (\ref{defparam}), and that $\tau\ll 1$. Consider the following statement which we call  $ \mathtt{Loc}(M) $ for $M\geq 1$: for any $(N,L)\in\mathcal{K}$ with $L<M$, we have
\begin{align}
\label{induct1}
&\|h^{N,L,\dagger}\|_{Y^b}\leq L^{-\delta_0};\\
\label{induct2}
&\|h^{N,L,\dagger}\|_{Z^b}\leq N^{\frac{1}{2}+\gamma_0}L^{-\frac{1}{2}};\\
\label{induct3}&\bigg\|\bigg(1+\frac{|k-k^*|}{L}\bigg)^{\kappa}h_{kk^*}^{N,L,\dagger}\bigg\|_{Z^{b}}\leq N.
\end{align} Define the operators\footnote{In fact we will prove stronger bounds where the low frequency inputs in (\ref{devop1}) are replaced by $v_{L_1}^\dagger,\cdots v_{L_{2r}}^{\dagger}$ with $\max(L_j)=L$, and similarly for (\ref{defop2}). But for simplicity we will just write (\ref{devop1}) and (\ref{defop2}).} (where $0\leq l\leq r$)
\begin{align}\label{devop1}\mathcal{P}^+(w)&:=\chi_{\tau}(t)\cdot\mathcal{I}\Pi_N\big[\mathcal{N}_{2l+1}(w,v_L^\dagger,\cdots,v_L^\dagger)-\mathcal{N}_{2l+1}(w,v_\frac{L}{2}^\dagger,\cdots,v_{\frac{L}{2}}^\dagger)\big],\\
\label{defop2}\mathcal{P}^-(w)&:=\chi_{\tau}(t)\cdot\mathcal{I}\Pi_N\big[\mathcal{N}_{2l+1}(v_L^\dagger,w,v_L^\dagger,\cdots,v_L^\dagger)-\mathcal{N}_{2l+1}(v_\frac{L}{2}^\dagger,w,v_{\frac{L}{2}}^\dagger,\cdots,v_{\frac{L}{2}}^\dagger)\big],
\end{align} then for any $(N,L)\in\mathcal{K}$ as defined in \eqref{setNL} with $L<M$ we have
\begin{equation}\label{induct4}\|\mathcal{P}^{\pm}\|_{X^{b}\to X^{b}}\leq \tau^{\theta}L^{-\delta_0^{\frac{1}{2}}}.
\end{equation} Let the kernel of $\mathcal{P}^+$ be $\mathfrak{h}_{kk'}^{N,L}(t,t')$, then for any $(N,L)\in\mathcal{K}$ with $L<M$ we have
\begin{equation}\label{induct5}\|\mathbf{1}_{|k|,|k'|\geq \frac{N}{4}}\cdot\mathfrak{h}_{kk'}^{N,L}(t,t')\|_{Z^{b,b}}\leq  \tau^{\theta}N^{\frac{1}{2}+\gamma_0-\delta^3}L^{-\frac{1}{2}}.
\end{equation} Finally for any $N\leq M$ we have
\begin{equation}\label{induct6}\|z_N^\dagger\|_{X^b}\leq N^{-1+\gamma}.
\end{equation}

Now suppose the statement $\mathtt{Loc}(M)$ is true for $\omega\in\Xi$, where $\Xi$ is a set, then the statement $\mathtt{Loc}(2M)$ is true for $\omega\in \Xi'$ where $\Xi'$ is another set such that $\mathbb{P}(\Xi\backslash\Xi')\leq C_{\theta}e^{-(\tau^{-1}M)^\theta}$. In particular, apart from a set of $\omega$ with probability $\leq C_{\theta}e^{-\tau^{-\theta}}$, the statement $\mathtt{Loc}(M)$ is true for all $M$.
\end{prop}
\subsection{The proof of Proposition \ref{localmain2}: reduction to multilinear estimates}\label{reductmulti} The heart of the proof of Proposition \ref{localmain2} is a collection of (probabilistic) multilinear estimates for $\mathcal{N}_{2l+1}$. We will state them in Proposition \ref{multi0} below and show that they imply Proposition \ref{localmain2}. We leave the proof of Proposition \ref{multi0} to Section \ref{multiest}.

\begin{prop}\label{multi0} Recall the relevant constants defined in (\ref{defparam}), and that $\tau\ll 1$. Let the multilinear form $\mathcal{N}_{n}$ be as in (\ref{multiform}), where $1\leq n\leq 2r+1$. We will also consider $\mathcal{N}_{n+1}$, in which we assume $\iota_1=+$. For each $1\leq j\leq n$, the input function $v^{(j)}$ satisfies one of the followings:

(i) Type (G), where we define $L_j=1$, and \begin{equation}\label{input1}(\widetilde{v^{(j)}})_{k_j}(\lambda_j)=\mathbf{1}_{N_j/2< \langle k_j\rangle\leq N_j}\frac{g_{k_j}(\omega)}{\langle k_j\rangle}\widehat{\chi}(\lambda_j).
\end{equation}

(ii) Type (C), where
\begin{equation}\label{input2}(\widetilde{v^{(j)}})_{k_j}(\lambda_j)=\sum_{N_j/2<\langle k_j^*\rangle\leq N_j}h_{k_jk_j^*}^{(j)}(\lambda_j,\omega)\frac{g_{k_j^*}(\omega)}{\langle k_j^*\rangle},\end{equation} with $h_{k_jk_j^*}^{(j)}(\lambda_j,\omega)$ supported in the set $\big\{\langle k_j\rangle\leq N_j,\frac{N_j}{2}<\langle k_j^*\rangle\leq N_j\big\}$, $\mathcal{B}_{\leq N_j}$ mesurable and $\mathcal{B}_{\leq L_j}^+$ measurable for some $L_j\leq N_j^{1-\delta}$, and satisfying the bounds
\begin{equation}\label{input2+}
\begin{aligned}\|\langle \lambda_j\rangle^{b}h_{k_jk_j^*}^{(j)}(\lambda_j)\|_{\ell_{k_j^*}^2\to\ell_{k_j}^2L_{\lambda_j}^2}\lesssim L_j^{-\delta_0},\quad \| \langle \lambda_j\rangle^{b}h_{k_jk_j^*}^{(j)}(\lambda_j)\|_{\ell_{k_j,k_j^*}^2L_{\lambda_j}^2}&\lesssim N_j^{\frac{1}{2}+\gamma_0}L_j^{-\frac{1}{2}},\\\bigg\|\langle \lambda_j\rangle^{b}\bigg(1+\frac{|k_j-k_j^*|}{L_j}\bigg)^\kappa h_{k_jk_j^*}^{(j)}(\lambda_j)\bigg\|_{\ell_{k_j,k_j^*}^2L_{\lambda_j}^2}&\lesssim N_j.
\end{aligned}\end{equation}

(iii) Type (D), where $(\widetilde{v^{(j)}})_{k_j}(\lambda_j)$ is supported in $\{|k_j|\lesssim N_j\}$, and satisfies
\begin{equation}\label{input3}\|\langle\lambda_j\rangle^b(\widetilde{v^{(j)}})_{k_j}(\lambda_j)\|_{\ell_{k_j}^2L_{\lambda_j}^2}\lesssim N_j^{-(1-\gamma)}.
\end{equation} In each case, we will assume that derivatives in $\lambda_j$ of these functions satisfy the same bounds. This can always be guaranteed, since in practice everything will be compactly supported in time.

\smallskip

Assume for $n_1\leq n$ that $v^{(j)}$ are of type (D) for $n_1+1\leq j\leq n$, and of type (G) or (C) for $1\leq j\leq n_1$. Let $\mathcal{G}$ and $\mathcal{C}$ be the sets of $j$ such that $v^{(j)}$ are of type (G) and (C) respectively, similarly denote by $\mathcal{D}:=\{n_1+1,\cdots,n\}$. Let $N^{(j)}=\max^{(j)}(N_1,\cdots,N_n)$ as before, and let $1\leq a\leq n$ be such that $N^{(1)}\sim N_a$. Given $N_*\geq 1$, the followings hold $\tau^{-1}N_*$-certainly. We emphasize that the exceptional set of $\omega$ removed does \emph{not} depend on the choice of the functions $v_j(j\geq n_1+1)$.

\smallskip

(1) If $a\geq n_1+1$ (say $a=n$) and $N_*\gtrsim N^{(2)}$, then we have (recall $b_1=b+\delta^4$)
\begin{equation}\label{mainmult1}\|\mathcal{I}\mathcal{N}_{n}(v^{(1)},\cdots,v^{(n)})\|_{X^{b_1}}\lesssim \tau^{-\theta}(N_*)^{C\kappa^{-1}} (N^{(1)})^{-1+\gamma}(N^{(2)})^{-\delta_0^{\frac{1}{3}}}.
\end{equation} Here the exceptional set does not depend on $N^{(1)}$.
\smallskip

(2) If $a\leq n_1$ and $N_*\gtrsim N^{(1)}$, then we have \begin{equation}\label{mainmult2}\|\mathcal{I}\mathcal{N}_{n}(v^{(1)},\cdots,v^{(n)})\|_{X^{b_1}}\lesssim \tau^{-\theta}(N_*)^{C\kappa^{-1}} (N^{(1)}N^{(2)})^{-\frac{1}{2}(1-\gamma_0)}.
\end{equation} If moreover $\iota_a=-$, then we have the stronger bound
\begin{equation}\label{mainmult3}\|\mathcal{I}\mathcal{N}_{n}(v^{(1)},\cdots,v^{(n)})\|_{X^{b_1}}\lesssim \tau^{-\theta}(N_*)^{C\kappa^{-1}} (N^{(1)})^{-(1-\gamma_0)}.
\end{equation} If moreover $\iota_a=+$ and $N^{(2)}\lesssim (N^{(1)})^{1-\delta}$, then we have stronger bound for the projected term
\begin{equation}\label{mainmult4}\|\mathcal{I}\Pi_{N^{(1)}}^\perp\mathcal{N}_{n}(v^{(1)},\cdots,v^{(n)})\|_{X^{b_1}}\lesssim \tau^{-\theta}(N_*)^{C\kappa^{-1}} (N^{(1)})^{-(1-\frac{4\gamma}{5})}.
\end{equation}
\smallskip
(3) Now consider the operator
\begin{equation}\mathcal{Q}^+(w):=\mathcal{I}\Pi_{N_0}\mathcal{N}_{n+1}(w,v^{(1)},\cdots,v^{(n)}),
\end{equation} and let its kernel be $\mathfrak{h}_{kk'}(t,t')$. If $N^{(1)}\lesssim N_0^{1-\delta}$ and $N_*\gtrsim N_0$, then we have
\begin{equation}\label{mainmult5}\|\mathbf{1}_{|k|,|k'|\geq \frac{N_0}{4}}\cdot\mathfrak{h}_{kk'}(t,t')\|_{Z^{b_1,b}}\lesssim \tau^{-\theta}(N_*)^{C\kappa^{-1}}N_0^{\frac{1}{2}}(N^{(1)})^{-\frac{1}{2}+\gamma_0}.
\end{equation}
\end{prop}
\begin{rem}\label{gamma2} The improvement (\ref{mainmult4}) is due to the exact projection $\Pi_{N^{(1)}}^\perp$. In fact this implies that in the expression (\ref{multiform}) there exists some $1\leq a\leq n$ and some $\Gamma$, namely $\Gamma=(N^{(1)})^2-1$, such that \begin{equation}\label{gammacon}|k|^2\geq \Gamma\geq |k_a|^2\quad\textrm{or}\quad |k|^2\leq \Gamma\leq |k_a|^2,\quad\textrm{and }N^{(1)}\sim N_a.\end{equation} We call (\ref{gammacon}) the \emph{$\Gamma$-condition}. If we put some other projections in $\mathcal{N}_n$ that also guarantee (\ref{gammacon}), for example $\Pi_M\mathcal{N}_{n}(\cdots,\Pi_M^\perp v^{(a)},\cdots)$ where $N^{(1)}\sim N_a$, then the same improvement (\ref{mainmult4}) will remain true.

In the proof below we will see that the $\Gamma$ condition provides the needed improvements in the case $N^{(1)}\sim N_a$ and $\iota_a=+$. This is the reason why we place $\psi_{N, L}$ in the first position of $\mathcal{N}_{2l+1}$ in (\ref{defpsi}). On the other hand, the term where $\psi_{N, L}$ is placed in the second position can be handled using the improvement (\ref{mainmult3}).
\end{rem}

\begin{proof}[Proof of Proposition \ref{localmain2} assuming Proposition \ref{multi0}] To prove the statement $\mathtt{Loc}(2M)$ we start with (\ref{induct4}) and (\ref{induct5}), and may assume $L=M$. The proof for $\mathcal{P}^-$ in (\ref{induct4}) is similar, so let us consider $\mathcal{P}^+$. Since $v_M^\dagger=\sum_{L'\leq M}y_{L'}^\dagger$, by definition we can write $\mathcal{P}^+$ as a superposition of forms 
\[w\mapsto\chi_{\tau}(t)\cdot\mathcal{I}\Pi_N\mathcal{N}_{2l+1}(w,y_{N_2}^\dagger,\cdots,y_{N_{2l+1}}^\dagger),\] where $\max(N_2,\cdots,N_{2l+1})=M$. As we have the decomposition
\begin{equation} y_{N_j}^\dagger\,=\,\chi(t)e^{it\Delta}(\Delta_{N_j}f(\omega))\,+\, \sum _{L_j:(N_j,L_j)\in\mathcal{K}}\zeta_{N_j,L_j}^\dagger\,+\, z_{N_j}^\dagger
\end{equation}by $\mathtt{Loc}(M)$ we know that each $y_{N_j}^{\dagger}$ can be decomposed into terms of type (G), type (C) (corresponding to some $L_j\lesssim N_j^{1-\delta}$), and type (D). The bound (\ref{induct4}) is then a consequence of (\ref{sttime1}) and (\ref{mainmult1}), after removing a set of $\omega$ with measure $\leq C_{\theta}e^{-(\tau^{-1}M)^\theta}$ that is independent of $N$. Note that by (\ref{matrixnorm0}), the $L=M$ case of (\ref{induct4}) is equivalent to 
\begin{equation}\label{inductmed}\|\mathfrak{h}^{N,M}\|_{Y^{b,b}}\leq \tau^{\theta}M^{-\delta_0^{\frac{1}{2}}}.
\end{equation}Similarly (\ref{induct5}) follows from (\ref{sttime1}) and (\ref{mainmult5}), because we have
\[\tau^{b_1-b}\tau^{-\theta}N^{C\kappa^{-1}}N^{\frac{1}{2}}M^{-\frac{1}{2}+\gamma_0}\ll \tau^{\theta} N^{\frac{1}{2}+\gamma_0-2\delta^3}M^{-\frac{1}{2}}
\] using the fact that $M\lesssim N^{1-\delta}$. The set of $\omega$ removed here will depend on $N$, but it will have measure $\leq C_{\theta}e^{-(\tau^{-1}N)^\theta}$, so summing in $N\geq M$ we still get a set of measure $\leq C_{\theta}e^{-(\tau^{-1}M)^\theta}$.
\smallskip

Next we prove \eqref{induct1}--\eqref{induct3}, again assuming $L=M$.  By (\ref{induct4}) and $\mathtt{Loc}(M)$
we already know that the right hand side of (\ref{daggered}) gives a contraction mapping in $X^b$,
so (\ref{daggered}) does have a unique solution. Subtracting the equations (\ref{daggered}) with $M$ and with $\frac{M}{2}$ instead of $M$, we deduce that
\begin{multline}\label{matrixprod}h_{kk^*}^{N,M,\dagger}(t) = -i\sum_{l=0}^r(l+1)\,c_{rl}\, (m_N^*)^{r-l}\bigg\{\sum_{L\leq M}\sum_{k'}\int\mathrm{d}t'\cdot \mathfrak{h}_{kk'}^{N,L}(t,t')h_{k'k^*}^{N,M,\dagger}(t')\\+\sum_{L<M}\sum_{k'}\int\mathrm{d}t'\cdot \mathfrak{h}_{kk'}^{N,M}(t,t')h_{k'k^*}^{N,L,\dagger}(t')+\int\mathrm{d}t'\cdot \mathfrak{h}_{kk^*}^{N,M}(t,t')H_{k^*k^*}^{N,\frac{1}{2},\dagger}(t')\bigg\},
\end{multline} where $\mathfrak{h}_{kk'}^{N,L}(t,t')$ is the kernel corresponding to $\mathcal{P}^+$ in the $(k,k',t,t')$ variables. Recall that we are already in a set where (\ref{simplebd}) is true, which allows us to control $m_N^*$. Now by the definition of $Y^b$ and $Y^{b,b}$ norms, the statement $\mathtt{Loc}(M)$ and (\ref{inductmed}), we conclude that
\begin{align*}\|h^{N,M,\dagger}\|_{Y^b}&\lesssim\sum_{L\leq M}\|\mathfrak{h}^{N,L}\|_{Y^{b,b}}\cdot\|h^{N,M,\dagger}\|_{Y^b}+\|\mathfrak{h}^{N,M}\|_{Y^{b,b}}\bigg(\sum_{L<M}\|h^{N,L,\dagger}\|_{Y^b}+1\bigg)\nonumber\\
&\lesssim \|h^{N,M,\dagger}\|_{Y^b}\cdot \sum_{L\leq M} \tau^{\theta}L^{-\delta_0^{\frac{1}{2}}}+\sum_{L<M}\tau^{\theta}M^{-\delta_0^{\frac{1}{2}}}L^{-\delta_0}\lesssim \tau^\theta \|h^{N,M,\dagger}\|_{Y^b}+\tau^\theta M^{-\delta_0^{\frac{1}{2}}}\nonumber,
\end{align*}
 which implies (\ref{induct1}) as desired.  In the same way we can prove (\ref{induct3}) by using (\ref{operatorbd3}), noting that $\widetilde{\mathfrak{h}^{N,L}}$ is supported in $|k-k'|\lesssim L$. 

As for (\ref{induct2}), recall that for
\[\widetilde{H}_{kk^*}(\lambda)=\sum_{k'}\int\mathrm{d}\lambda'\cdot \mathfrak{h}_{kk'}(\lambda,\lambda')h_{k'k^*}(\lambda')
\]we have, by definition of the relevant norms, that
\[\|H\|_{l_{k,k^*}^2L_{\lambda}^2}\leq\min(\|\mathfrak{h}\|_{l_{k,k'}^2L_{\lambda,\lambda'}^2}\|h\|_{\ell_{k^*}^2\to \ell_{k'}^2L_{\lambda'}^2},\|\mathfrak{h}\|_{l_{k,k'}^2\to L_{\lambda,\lambda'}^2}\|h\|_{\ell_{k',k^*}^2L_{\lambda'}^2}).
\] Now in (\ref{matrixprod}) we may assume $|k-k^*|\leq 2^{-10}N$ and $|k'-k^*|\leq 2^{-10}N$ (otherwise the bound follows trivially from (\ref{induct3}) which we just proved), so in particular $|k|,|k'|\geq \frac{N}{4}$ as $|k^*|\geq\frac{N}{2}$. Using the statement $\mathtt{Loc}(M)$ and (\ref{induct5}) we get
\begin{align*}\|h^{N,M,\dagger}\|_{Z^b}&\lesssim\sum_{L\leq M}\|\mathfrak{h}^{N,L}\|_{Y^{b,b}}\cdot\|h^{N,M,\dagger}\|_{Z^b}+\|\mathbf{1}_{|k|,|k'|\geq \frac{N}{4}}\cdot\mathfrak{h}_{kk'}^{N,M}(t,t')\|_{Z^{b,b}}\bigg(\sum_{L<M}\|h^{N,L,\dagger}\|_{Y^b}+1\bigg)\\
&\lesssim \|h^{N,M,\dagger}\|_{Z^b}\cdot \sum_{L\leq M} \tau^{\theta}L^{-\delta_0^{\frac{1}{2}}}+\sum_{L<M}\tau^{\theta}N^{\frac{1}{2}+\gamma_0-\delta^3}M^{-\frac{1}{2}}L^{-\delta_0}\\&\lesssim \tau^\theta \|h^{N,M,\dagger}\|_{Y^b}+\tau^\theta N^{\frac{1}{2}+\gamma_0-\delta^3}M^{-\frac{1}{2}},
\end{align*}
which proves (\ref{induct2}).

\smallskip

Finally we prove (\ref{induct6}) with $L=2M$, by showing that the mapping defined in (\ref{fixedpoint}) is indeed a contraction mapping from the given set $\mathcal{Z}=\{z:\|z\|_{X^b}\leq (2M)^{-1+\gamma}\}$ to itself. Actually we will only prove that this mapping sends $\mathcal{Z}$ to $\mathcal{Z}$, as the difference estimate is done in the same way. 

We will separate the right hand side of (\ref{fixedpoint}) into six groups, each of which has the form
\[\chi_{\tau}(t)\cdot(m_{2M}^*)^{r-l}\mathcal{I}\Pi_{2M}\mathcal{N}_{2l+1}(v^{(1)},\cdots,v^{(2l+1)}),\] where
\begin{enumerate}[label=(\alph*)]
\item At least two of the $v^{(j)}$ are equal to $z+\psi_{2M,L_0}^{\dagger}$, and others are either $z+\psi_{2M,L_0}^{\dagger}$ or $v_M^\dagger$;
\item We have $v^{(2)}=\psi_{2M,L_0}^{\dagger}$, and all others equal $v_M^\dagger$;
\item One of $v^{(1)}$ or $v^{(2)}$ equals $z$, and all others equal $v_M^\dagger$;
\item We have $v^{(1)}=\psi_{2M,L_0}^{\dagger}$, another $v^{(j)}$ equals $v_{M}^\dagger-v_{L_0}^\dagger$, and all others equal either $v_M^\dagger$ or $v_{L_0}^\dagger$;
\item The factor $(m_{2M}^*)^{r-l}$ is replaced by $(m_{M}^*+\nu_{2M})^{r-l}-(m_M^*)^{r-l}$ and all $v^{(j)}$ equal $v_M^\dagger$;
\item Same as (a), but with $\Delta_{2M}$ instead of $\Pi_{2M}$, and all $v^{(j)}$ equal $v_M^\dagger$.
\end{enumerate}

By (\ref{sttime1}), it suffices to prove that each of these terms in (a) through (f), but without the $\chi_{\tau}(t)$ factor, is bounded in $X^{b_1}$ by $\tau^{-\theta}(2M)^{-1+\gamma}$.  Let one such term be denoted by $\mathcal{M}$, and notice that we can decompose
\[v_M^\dagger=\sum_{L\leq M}y_L^\dagger,\quad v_{L_0}^\dagger=\sum_{L\leq L_0}y_L^\dagger,\quad v_M-v_{L_0}=\sum_{L_0<L\leq M}y_L^\dagger,\]\[\psi_{2M,L_0}^\dagger=\chi(t)e^{it\Delta}(\Delta_{2M}f(\omega))+\sum_{L\leq L_0}\zeta_{2M,L}.\]Moreover by what we have proved so far, we know that $y_L^\dagger$ for $L\leq M$ can be decomposed into terms of types (G), (C) and (D), and that $\chi(t)e^{it\Delta}(\Delta_{2M}f(\omega))$ is of type (G), $\zeta_{2M,L}$ is of type (C), and $z$ is of type (D). By such decomposition we can reduce $\mathcal{M}$ to the terms studied in Proposition \ref{multi0}, with various choices of $N_j$ and $L_j$. We now proceed case by case.

Case (a): Here we have at least two inputs $v^{(j)}$ with $N_j=2M$, so by either (\ref{mainmult1}) or (\ref{mainmult2}) we can bound
\[\|\mathcal{M}\|_{X^{b_1}}\lesssim \tau^{-\theta}(2M)^{-1+\gamma_0+C\kappa^{-1}}\] by removing a set of measure $\leq C_{\theta}e^{-(\tau^{-1}M)^\theta}$, which suffices.

Case (b): Here we have $v^{(2)}=N^{(1)}=2M$, while $\iota_2=-$. By (\ref{mainmult3}) we have the same bound as above.

Case (c): This term, \emph{with the $\chi_{\tau}(t)$ factor}, can be written as
\[\sum_{L\leq M}\mathcal{P}_L^{\pm}(z),\] where $\mathcal{P}_L^{\pm}$ are defined as in (\ref{devop1}) and (\ref{defop2}) (such expressions can be defined even if $(2M,L)\not\in\mathcal{K}$, and we use subscript $L$ to indicate $L$ dependence). If $L\leq L_0$ then by (\ref{induct4}) we can bound
\[\|\chi_{\tau}(t)\cdot \mathcal{M}\|_{X^b}\leq(2M)^{-1+\gamma}\sum_{L\leq L_0}\tau^\theta L^{-\delta_0^{\frac{1}{2}}}\lesssim \tau^\theta(2M)^{-1+\gamma},\] which suffices. Note that here no further set of $\omega$ needs to be removed. If  $L>L_0$ then this term can be bounded in the same way as in case (d) below, by removing a set of measure $\leq C_{\theta}e^{-(\tau^{-1}M)^\theta}$.

Case (d): Here we have, due to the factor $v_M^\dagger-v_{L_0}^\dagger$, that $N^{(1)}=2M$ and $N^{(2)}\gtrsim M^{1-\delta}$; so by either (\ref{mainmult1}) or (\ref{mainmult2}) we can bound \[\|\mathcal{M}\|_{X^{b_1}}\lesssim \tau^{-\theta}M^{-1+\frac{\delta}{2}+C\gamma_0}\] by removing a set of measure $\leq C_{\theta}e^{-(\tau^{-1}M)^\theta}$, which suffices.

Case (e): The bound for this term follows from the bound for $\nu_{2M}$ and the trivial bound (say (\ref{mainmult1}) or (\ref{mainmult2})) for the $\mathcal{N}_{2l+1}$ term.

Case (f): We may assume $N^{(1)}=N_a\sim M$. If either $v^{(a)}$ is of type (D) or $N^{(2)}\gtrsim M^{1-\delta}$ or $\iota_a=-$, we can reduce to one of the previous cases (namely (c) or (d) or (b)) and close as before; if $v^{(a)}$ is of type (G) or (C), $\iota_a=+$ and $N^{(2)}\ll M^{1-\delta}$, then the $\Delta_{2M}$ projection allows us to apply the improvement (\ref{mainmult4}), which leads to \[\|\mathcal{M}\|_{X^{b_1}}\lesssim \tau^{-\theta}M^{-1+\frac{4}{5}\gamma+C\kappa^{-1}}\] by removing a set of measure $\leq C_{\theta}e^{-(\tau^{-1}M)^\theta}$, which suffices. This completes the proof.
\end{proof}
\section{Large deviation and counting estimates}\label{prep2} Proposition \ref{multi0} will be proved in Section \ref{multiest}. In this section we make some preparations for the proof, namely we introduce two large deviation estimates and some counting estimates for integer lattice points.
\subsection{Large deviation estimates} We first prove the following large deviation estimate for multilinear Gaussians, which as far as we know is new.
\begin{lem}\label{largedev} Let $E\subset\mathbb{Z}^2$ be a finite subset, and let $\mathcal{B}$ be the $\sigma$-algebra generated by $\{g_k:k\in E\}$. Let $\mathcal{C}$ be a $\sigma$-algebra independent with $\mathcal{B}$, and let $\mathcal{C}^+$ be the smallest $\sigma$-algebra containing both $\mathcal{C}$ and the $\sigma$-algebra generated by $\{|g_k|^2:k\in E\}$. Consider the expression
\begin{equation}\label{indp}F(\omega)=\sum_{(k_1,\cdots,k_n)\in E^n}a_{k_1\cdots k_n}(\omega)\prod_{j=1}^ng_{k_j}(\omega)^{\iota_j},
\end{equation} where $n\leq 2r+1$, $\iota_j\in\{\pm\}$ and the coefficients $a_{k_1\cdots k_n}(\omega)$ are $\mathcal{C}^+$ measurable. Let $A\geq \#E$, then $A$-certainly we have 
\begin{equation}\label{firstbound} |F(\omega)|\leq A^\theta M(\omega)^{\frac{1}{2}}, \end{equation} where
\begin{equation}\label{ortho}M(\omega)=\sum_{(X,Y)}\sum_{(k_{m}):m\not\in X\cup Y}\bigg(\sum_{\mathrm{pairing}\,(k_{i_s},k_{j_s}):1\leq s\leq p}|a_{k_1\cdots k_n}(\omega)|\bigg)^2.
\end{equation} In the summation (\ref{ortho}) we require that all $k_j\in E$, and that $X:=\{i_1,\cdots,i_p\}$ and $Y:=\{j_1,\cdots,j_p\}$ are two disjoint subsets of $\{1,2,\cdots,n\}$. Recall also the definition of pairing in Definition \ref{dfpairing}.
\end{lem}
\begin{proof} Write in polar coordinates $g_k(\omega)=\rho_k(\omega)\eta_k(\omega)$ where $\rho_k=|g_k|$ and $\eta_k=\rho_k^{-1}g_k$, then all the $\rho_k$ and $\eta_k$ are independent, and each $\eta_k$ is uniformly distributed on the unit circle of $\mathbb{C}$. We may write
\begin{equation}\label{newexp1}F(\omega)=\sum_{(k_1,\cdots,k_n)\in E^n}b_{k_1\cdots k_n}(\omega)\prod_{j=1}^n\eta_{k_j}(\omega)^{\iota_j},\quad b_{k_1\cdots k_n}(\omega):=a_{k_1\cdots k_n}(\omega)\prod_{j=1}^n\rho_{k_j}(\omega).
\end{equation} Since $a_{k_1\cdots k_n}(\omega)$ are $\mathcal{C}^+$ measurable, we know that the collection $\{b_{k_1\cdots k_n}\}$ is independent with the collection $\{\eta_k:k\in E\}$. The goal is to prove that
\begin{equation}\label{largedevnew}\mathbb{P}(|F(\omega)|\geq BM_1(\omega)^{\frac{1}{2}})\leq Ce^{-B^{1/n}},
\end{equation} where $C$ is an absolute constant, and $M_1(\omega)$ is the same as $M(\omega)$ but with the coefficients $a$  replaced by the coefficients $b$. In fact, as $A\geq \#E$ we have $A$-certainly that $|b_{k_1\cdots k_n}(\omega)|\leq A^\theta|a_{k_1\cdots k_n}(\omega)|$, so (\ref{largedevnew}) implies the desired bound.

We now prove (\ref{largedevnew}). By independence, we may condition on the $\sigma$-algebra generated by $\{b_{k_1\cdots k_n}\}$ and prove (\ref{largedevnew}) for the conditional probability, then take another expectation; therefore we may assume that $b_{k_1\cdots k_n}$ are constants (so $M_1(\omega)=M_1$ is a constant). Now let $\{h_k:k\in E\}$ be another set of i.i.d. normalized complex Gaussian random variables and define
\begin{equation}\label{newexp2}G=\sum_{(k_1,\cdots,k_n)\in E^n}|b_{k_1\cdots k_n}|\prod_{j=1}^nh_{k_j}^{\iota_j},
\end{equation} we want to compare $F$ and $G$ and show $\mathbb{E}|F|^{2d}\leq \mathbb{E}|G|^{2d}$ for any positive integer $d$. In fact,
\begin{align}\label{exp1}\mathbb{E}(|F|^{2d})&=\sum_{(k_j^i,\ell_j^i:1\leq i\leq d,1\leq j\leq n)}\prod_{i=1}^db_{k_1^i\cdots k_n^i}\overline{b_{\ell_1^i\cdots\ell_n^i}}\mathbb{E}\bigg(\prod_{i=1}^d\prod_{j=1}^n\eta_{k_j^i}^{\iota_j}\overline{\eta_{\ell_j^i}^{\iota_j}}\bigg),\\
\label{exp2}\mathbb{E}(|G|^{2d})&=\sum_{(k_j^i,\ell_j^i:1\leq i\leq d,1\leq j\leq n)}\prod_{i=1}^d|b_{k_1^i\cdots k_n^i}||b_{\ell_1^i\cdots\ell_n^i}|\mathbb{E}\bigg(\prod_{i=1}^d\prod_{j=1}^nh_{k_j^i}^{\iota_j}\overline{h_{\ell_j^i}^{\iota_j}}\bigg).
\end{align}
The point is that we always have
\[\bigg|\mathbb{E}\bigg(\prod_{i=1}^d\prod_{j=1}^n\eta_{k_j^i}^{\iota_j}\overline{\eta_{\ell_j^i}^{\iota_j}}\bigg)\bigg|\leq \mathrm{Re}\,\mathbb{E}\bigg(\prod_{i=1}^d\prod_{j=1}^nh_{k_j^i}^{\iota_j}\overline{h_{\ell_j^i}^{\iota_j}}\bigg).\] In fact, by collecting all different factors we can write the expectations as
\[\mathbb{E}\bigg(\prod_{\alpha}(\eta_{k^{(\alpha)}})^{x_\alpha}(\overline{\eta_{k^{(\alpha)}}})^{y_\alpha}\bigg)\quad\mathrm{and}\quad \mathbb{E}\bigg(\prod_{\alpha}(h_{k^{(\alpha)}})^{x_\alpha}(\overline{h_{k^{(\alpha)}}})^{y_\alpha}\bigg),\] where the $k^{(\alpha)}$ are pairwise distinct. If $x_\alpha\neq y_\alpha$ for some $\alpha$, both expectations will be $0$; if $x_\alpha=y_\alpha$ for each $\alpha$, then the first expectation will be $1$ and the second expectation will be $\prod_{\alpha}x_\alpha!\geq 1$.

Now, since $G$ is an exact multilinear Gaussian expression, by the standard hypercontractivity estimate, see \cite{OT}, we have
\[\mathbb{E}|F|^{2d}\leq \mathbb{E}|G|^{2d}\leq(2d-1)^{nd}(\mathbb{E}|G|^2)^d,\] so for any $D>0$ by using Chebyshev's inequality and optimizing in $d$ we have
\[\mathbb{P}(|F(\omega)|\geq D)\leq \min_d\bigg\{(2d-1)^{nd}\big(\frac{\mathbb{E}|G|^2}{D^2}\big)^d\bigg\}\leq C\exp\bigg\{\frac{-1}{2e}\big(\frac{\mathbb{E}|G|^2}{D^2}\big)^{\frac{1}{n}}\bigg\}\] with some constant $C$ depending only on $n$. It then suffices to prove $\mathbb{E}|G|^2\lesssim M_1$ with constants depending only on $n$.

By dividing the sum (\ref{newexp2}) into finitely many terms and rearranging the subscripts, we may assume
\begin{equation*}k_{1}=\cdots=k_{j_1},\,\, k_{j_1+1}=\cdots=k_{j_2},\cdots,k_{j_{r-1}+1}=\cdots=k_{j_r},\,\,1\leq j_1<\cdots <j_r=n,
\end{equation*} and the $k_{j_s}$ are different for $1\leq s\leq r$. Such a monomial that appears in (\ref{newexp2}) has the form
\[\prod_{s=1}^rh_{k_{j_s}}^{\beta_s}(\overline{h_{k_{j_s}}})^{\gamma_s},\quad \beta_s+\gamma_s=j_s-j_{s-1}\,\,(j_0=0),\] where the factors for different $s$ are independent. We may also assume $\beta_s=\gamma_s$ for $1\leq s\leq q$ and $\beta_s\neq \gamma_s$ for $q+1\leq s\leq r$, and that $\iota_j$ has the same sign as $(-1)^j$ for $1\leq j\leq j_q$. Then we can further rewrite this monomial as a linear combination of
\[\prod_{s=1}^p\beta_s!\prod_{s=p+1}^q(|h_{k_{j_s}}|^{2\beta_s}-\beta_s!)\prod_{s=q+1}^rh_{k_{j_s}}^{\beta_s}(\overline{h_{k_{j_s}}})^{\gamma_s}\] for $1\leq p\leq q$. Therefore, $G$ is a finite linear combination of expressions of  the form
\[\sum_{k_{j_1},\cdots,k_{j_r}}|b_{k_{j_1},\cdots,k_{j_1},\cdots,k_{j_r},\cdots k_{j_r}}|\prod_{s=1}^p\beta_s!\prod_{s=p+1}^q(|h_{k_{j_s}}|^{2\beta_s}-\gamma_s!)\prod_{s=q+1}^rh_{k_{j_s}}^{\beta_s}(\overline{h_{k_{j_s}}})^{\gamma_s}.\] Due to independence and the fact that $\mathbb{E}(|h|^{2\beta}-\beta!)=\mathbb{E}(h^\beta(\overline{h})^\gamma)=0$ for a normalized Gaussian $h$ and $\beta\neq \gamma$, we conclude that
\begin{equation*}\mathbb{E}|G|^2\lesssim\sum_{k_{j_{p+1}},\cdots k_{j_r}}\bigg(\sum_{k_{j_1},\cdots,k_{j_p}}|b_{k_{j_1},\cdots,k_{j_1},\cdots,k_{j_r},\cdots k_{j_r}}|\bigg)^2,
\end{equation*} which is bounded by $M_1$ choosing $X=\{1,3,\cdots,j_p-1\}$ and $Y=\{2,4,\cdots,j_p\}$, since by our assumptions $(k_{2i-1},k_{2i})$ is a pairing for $2i\leq j_p$. This completes the proof.
\end{proof}
For the purpose of Section \ref{multiest} we will also need the following lemma, which is a more general large deviation-type estimate restricted to the no-pairing case:
\begin{lem}\label{lemma:4.2} Let $\delta$ be as in Section \ref{notations}, $n\leq 2r+1$ and consider the following expression
\begin{equation}\label{multgauss}M(\omega)=\sum_{(k_1,\cdots,k_n)}\sum_{(k_1^*,\cdots,k_n^*)}\int \mathrm{d}\lambda_1\cdots\mathrm{d}\lambda_n\cdot a_{k_1\cdots k_n}(\lambda_1,\cdots,\lambda_n)\prod_{j=1}^ng_{k_j^*}(\omega)^{\iota_j}h^{(j)}_{k_jk_j^*}(\lambda_j,\omega)^{\pm},
\end{equation} where $a_{k_1\cdots k_n}(\lambda_1,\cdots,\lambda_n)$ is a given (deterministic) function of $(k_1,\cdots,k_n,\lambda_1,\cdots,\lambda_n)$. Moreover in the summation we assume that there are no pairings among $\{k_1^*,\cdots,k_n^*\}$, that $\langle k_j\rangle \leq N_j$ and $\frac{N_j}{2}<\langle k_j^*\rangle\leq N_j$, and that $h_{k_jk_j^*}^{(j)}(\lambda_j,\omega)$, as a random variable, is $\mathcal{B}_{\leq N_j^{1-\delta}}^+$ measurable. Let $N_*\geq\max( N_1,\cdots,N_n)$, then $N_*$-certainly (the exceptional set removed will depend on the coefficients $a$), we have
\begin{equation}\label{multgauss1}|M(\omega)|\lesssim(N_*)^{\theta}\prod_{j=1}^n\|h_{k_jk_j^*}^{(j)}(\lambda_j,\omega)\|_{\ell_{k_j^*}^2\to\ell_{k_j}^2L_{\lambda_j}^2}\cdot\|a_{k_1\cdots k_n}(\lambda_1,\cdots\lambda_n)\|_{\mathcal{L}},
\end{equation} where $\mathcal{L}$ is an auxiliary norm defined by (where $\partial_\lambda=(\partial_{\lambda_1},\cdots,\partial_{\lambda_n})$)
\begin{equation}\label{auxnorm}\|a_{k_1\cdots k_n}(\lambda_1,\cdots\lambda_n)\|_{\mathcal{L}}^2:=\sum_{k_1,\cdots k_n}\int\mathrm{d}\lambda_1\cdots\mathrm{d}\lambda_n\cdot\big(\max_{1\leq j\leq n}\langle\lambda_j\rangle\big)^{\delta^6}(|a|^2+|\partial_{\lambda}a|^2).
\end{equation}
\end{lem}
\begin{proof} Consider the big box $\{|\lambda_j|\leq (N_*)^{\delta^{-7}}\}$ and divide it into small boxes of size $(N_*)^{-\delta^{-1}}$. By exploiting the weight $(\max_{1\leq j\leq n}\langle\lambda_j\rangle)^{\delta^6}$ in (\ref{auxnorm}) and using Poincar\'{e}'s inequality, we can find a function $b$ which is supported in the big box and is constant on each small box, such that
\begin{equation}\sup_{k_j,k_j^*}\|a-b\|_{L_{\lambda_1,\cdots,\lambda_n}^2}\lesssim (N_*)^{-\delta^{-1}}\|a\|_{\mathcal{L}}.\end{equation} Exploiting this $(N_*)^{-\delta^{-1}}$ gain, summing over $|k_j|,|k_j^*|\lesssim N_*$ and using the simple bound 
\begin{equation*}\sup_{k_j,k_j^*}\bigg|\int \mathrm{d}\lambda_1\cdots\mathrm{d}\lambda_n\cdot a_{k_1\cdots k_n}(\lambda_1,\cdots,\lambda_n)\prod_{j=1}^n h_{k_jk_j^*}^{(j)}(\lambda_j)^\pm\bigg|\lesssim \sup_{k_j,k_j^*}\|a\|_{L_{\lambda_1,\cdots,\lambda_n}^2}\prod_{j=1}^n\sup_{k_j,k_j^*}\|h_{k_jk_j^*}^{(j)}(\lambda_j)\|_{L_{\lambda_j}^2}\end{equation*} suffices to bound the contribution with $a$ replaced by $a-b$; thus we may now replace $a$ by $b$ (or equivalently, assume $a$ is supported in the big box and is constant on each small box) and will prove  (\ref{multgauss1})  $N_*$-certainly, with the $\mathcal{L}$ norm replaced by the $l^2L^2$ norm, by induction.

By symmetry we may assume $N_1\geq\cdots\geq N_n$. Choose the smallest $q$ such that $N_q>2^{10}N_{q+1}$, then $N_1\sim N_q$ with constant depending only on $n$. Unless $N_1\leq C$, in which case (\ref{multgauss1}) is trivial, we can conclude that
\[h_{k_jk_j^*}^{(j)}(\lambda_j,\omega)^\pm,\,1\leq j\leq n,\,\text{are $\mathcal{B}_{\leq N_1^{1-\delta}}^+$ measurable; $N_1^{1-\delta}\leq 2^{-10}N_q$,}\]
\[g_{k_j^*}(\omega)^{\pm},\,q+1\leq j\leq n,\,\text{are $\mathcal{B}_{\leq N_{q+1}}$ measurable; $N_{q+1}\leq 2^{-10}N_q$.}\] Note that in this case, there is no pairing among $\{k_1^*,\cdots,k_n^*\}$ if and only if there is no pairing among $\{k_1^*,\cdots,k_q^*\}$ and no pairing among $\{k_{q+1}^*,\cdots,k_n^*\}$. We can then write $M(\omega)$ as 
\begin{equation}M(\omega)=\sum_{k_1^*,\cdots,k_q^*}b_{k_1^*\cdots k_q^*}(\omega)\cdot\prod_{j=1}^q g_{k_j^*}(\omega)^{\iota_j},
\end{equation} where
\begin{multline}\label{defofb}b_{k_1^*\cdots k_q^*}(\omega)=\sum_{k_1,\cdots,k_q}\int\mathrm{d}\lambda_1\cdots\mathrm{d}\lambda_q\prod_{j=1}^qh_{k_jk_j^*}^{(j)}(\lambda_j,\omega)^\pm\\\times\sum_{(k_{q+1},\cdots k_n)}\sum_{(k_{q+1}^*,\cdots,k_n^*)}\int \mathrm{d}\lambda_{q+1}\cdots\mathrm{d}\lambda_n\cdot a_{k_1\cdots k_n}(\lambda_1,\cdots,\lambda_n)\prod_{j=q+1}^ng_{k_j^*}(\omega)^{\iota_j}h_{k_jk_j^*}^{(j)}(\lambda_j,\omega)^\pm
\end{multline} are $\mathcal{B}_{\leq 2^{-10}N_q}^+$ measurable. We then apply Lemma \ref{largedev} and conclude that, after removing a set of $\omega$ with probability $\leq C_{\theta}e^{-(N_*)^{\theta}}$, we have
\begin{equation}\label{gaussprob}|M(\omega)|^2\lesssim (N_*)^\theta\sum_{k_1^*,\cdots k_q^*}|b_{k_1^*\cdots k_q^*}(\omega)|^2.
\end{equation} Now by (\ref{defofb}) we have
\begin{multline}\label{l2tol2}\sum_{k_1^*,\cdots k_q^*}|b_{k_1^*\cdots k_q^*}(\omega)|^2\lesssim\prod_{j=1}^q\|h_{k_jk_j^*}^{(j)}(\lambda_j,\omega)\|_{\ell^2\to\ell^2L^2}^2\sum_{k_1,\cdots,k_q}\int\mathrm{d}\lambda_1\cdots\mathrm{d}\lambda_q\\\times\bigg|\sum_{(k_{q+1},\cdots k_n)}\sum_{(k_{q+1}^*,\cdots,k_n^*)}\int \mathrm{d}\lambda_{q+1}\cdots\mathrm{d}\lambda_n\cdot a_{k_1\cdots k_n}(\lambda_1,\cdots,\lambda_n)\prod_{j=q+1}^ng_{k_j^*}(\omega)^{\iota_j}h_{k_jk_j^*}^{(j)}(\lambda_j,\omega)^\pm\bigg|^2;
\end{multline} by induction hypothesis, we get that
\begin{multline}\label{fixedpt}\bigg|\sum_{(k_{q+1},\cdots k_n)}\sum_{(k_{q+1}^*,\cdots,k_n^*)}\int\mathrm{d}\lambda_{q+1}\cdots\mathrm{d}\lambda_n a_{k_1\cdots k_n}(\lambda_1,\cdots,\lambda_n)\prod_{j=q+1}^ng_{k_j^*}(\omega)^{\iota_j}h_{k_jk_j^*}^{(j)}(\lambda_j,\omega)^\pm\bigg|^2\\\lesssim (N_*)^{\theta}\prod_{j=q+1}^n\|h_{k_jk_j^*}^{(j)}(\lambda_j,\omega)\|_{\ell_{k_j^*}^2\to\ell_{k_j}^2L_{\lambda_j}^2}^2\sum_{k_{q+1},\cdots k_n}\int \mathrm{d}\lambda_{q+1}\cdots\mathrm{d}\lambda_n\cdot|a_{k_1\cdots k_n}(\lambda_1,\cdots,\lambda_n)|^2,
\end{multline} up to a set of $\omega$ with probability $\leq C_{\theta}e^{-(N_*)^{\theta}}$, for any \emph{fixed} $(k_j,\lambda_j)$ for $1\leq j\leq q$. By our assumption on the coefficients $a$, the function \[(k_{q+1},\cdots, k_n,\lambda_{q+1},\cdots,\lambda_n)\mapsto a_{k_1\cdots k_n}(\lambda_1,\cdots,\lambda_n)\] which depends on the parameters $(k_j,\lambda_j)$ for $1\leq j\leq q$, has only $(N_*)^{C\delta^{-7}}$ different possibilities, so by removing a set of $\omega$ with probability $\leq C_{\theta}e^{-(N_*)^{\theta}}$, we may assume (\ref{fixedpt}) holds for \emph{all} $(k_j,\lambda_j)$, $1\leq j\leq q$. Thus we can sum (\ref{fixedpt}) over $k_j$ and integrate over $\lambda_j$, and combine with (\ref{gaussprob}) and (\ref{l2tol2}) to get that
\begin{equation}|M(\omega)|^2\lesssim(N_*)^{\theta}\prod_{j=1}^n\|h_{k_jk_j^*}^{(j)}(\lambda_j,\omega)\|_{\ell_{k_j^*}^2\to\ell_{k_j}^2L_{\lambda_j}^2}^2\sum_{k_1,\cdots,k_n}\int\mathrm{d}\lambda_1\cdots\mathrm{d}\lambda_n\cdot|a_{k_1\cdots k_n}(\lambda_1,\cdots,\lambda_n)|^2.
\end{equation} This completes the proof.
\end{proof}
\begin{rem}\label{mesh} In the proof of Lemma \ref{lemma:4.2} above, the first step involves approximating the function $a$ by another function $b$ which is supported in a big box of size $(N_*)^{\delta^{-7}}$ and is constant in each small box of size $(N_*)^{-\delta^{-1}}$. This reduces the infinitely many choices for the $\lambda_j$ variables to essentially at most $(N_*)^{2\delta^{-7}}$ choices, which allows to bound the probability of the tail event in question by $C_\theta e^{-(N_*)^\theta}$. This trick, which we refer to as the \emph{meshing argument}, will be used frequently below (especially in Section \ref{multiest}) without  further explanations.
\end{rem}
\subsection{Counting estimates for lattice points} We start with a simple lemma, and then state the main integer lattice point counting bounds that will be used in the proof below.
\begin{lem}\label{lem:counting}(1) Let $\mathcal{R}=\mathbb{Z}$ or $\mathbb{Z}[i]$. Then, given $0\neq m\in\mathcal{R}$, and $a_0,b_0\in\mathbb{C}$, the number of choices for $(a,b)\in\mathcal{R}^2$ that satisfy
\begin{equation}m=ab,\,\,|a-a_0|\leq M,\,\,|b-b_0|\leq N
\end{equation}is $O(M^\theta N^\theta)$ with constant depending only on $\theta>0$.

(2) Given dyadic numbers $N_1\gtrsim N_2\gtrsim N_3$, consider the set
\begin{multline}\label{count}
S=\big\{(x,y,z)\in(\mathbb{Z}^2)^3:\iota_1 x+\iota_2y+\iota_3 z=d,\,\,\iota_1|x|^2+\iota_2 |y|^2+\iota_3 |z|^2=\alpha,\\|x-a|\lesssim N_1,|y-b|\lesssim N_2,|z-c|\lesssim N_3\big\}.
\end{multline} Assume also there is no pairing in $S$. Then, uniformly in $(a,b,c,d)\in(\mathbb{Z}^2)^4$ and $\alpha\in\mathbb{Z}$, we have $\#S\lesssim N_2^{1+\theta}N_3$. Moreover, if $\iota_1=\iota_2$, then we have the stronger bound $\#S\lesssim N_2^\theta N_3^2$.
\end{lem}
\begin{proof} (1) This strengthened divisor estimate is essentially proved in \cite{DNY}, Lemma 3.4. We know that $\mathcal{R}$ has unique factorization and satisfies the standard divisor estimate, namely the number of divisors of $0\neq m\in\mathcal{R}$ is $O(|m|^{\theta})$. Now suppose $\max(|a_0|,M)\geq\max(|b_0|,N)$, then $|m|\lesssim\max(|a_0|,M)^2$. We may assume $M_1\sim|a_0|\gg M^4$, and hence $|m|\lesssim M_1^2$.

We then claim that the number of divisors $a$ of $m$ that satisfies $|a-a_0|\leq M$ is at most two. In fact, suppose $a,b,c$ are different divisors of $m$ that belong to the ball $|x-a_0|\leq M$, then by unique factorization we have $\mathrm{lcm}(a,b,c)|m$, hence
\[\frac{abc}{\gcd(a,b)\gcd(b,c)\gcd(c,a)}\] divides $m$. As $|a|\sim M_1$ etc., and $|\gcd(a,b)|\leq |a-b|\lesssim M$ etc., we conclude that
\[M_1^2\gtrsim|m|\geq\bigg|\frac{abc}{\gcd(a,b)\gcd(b,c)\gcd(c,a)}\bigg|\gtrsim M_1^3M^{-3},\] contradicting the assumption $M_1\gg M^4$.

(2) Let $x=(x_1,x_2)$, etc. If $\iota_1=\iota_2$, then with fixed $z$ (which has $O(N_3^2)$ choices), $x+y$ will be constant. Let $x-y=w$, then
\[(w_1+iw_2)(w_1-iw_2)=|w|^2=2(|x|^2+|y|^2)-|x+y|^2\] is constant. As $w$ belongs to a ball of radius $O(N_2)$ in $\mathbb{R}^2$, by (1) we know that the number of choices for $w$ is $O(N_2^\theta)$, hence $\#S=O(N_2^\theta N_3^2)$. Below we will assume that $\iota_1=+$ and $\iota_2=-$.

(a) Suppose $\iota_3=+$, then we have that
\begin{equation*}(d_1-z_1)(z_1-y_1)+(d_2-z_2)(z_2-y_2)=(d-z)\cdot(z-y)=\frac{|d|^2-\alpha}{2}
\end{equation*} is constant. If $(d_1-z_1)(z_1-y_1)\neq 0$ (or similarly if $(d_2-z_2)(z_2-y_2)\neq 0$), then with fixed $(y_2,z_2)$ (which has $O(N_2N_3)$ choices), $(d_1-z_1)(z_1-y_1)$ will be constant. As $d_1-z_1$ belongs to an interval of size $O(N_3)$ in $\mathbb{R}$, and $z_1-y_1$ belongs to an interval of size $O(N_2)$ in $\mathbb{R}$, by (1) we know that the number of choices for $(y_1,z_1)$ is $O(N_2^\theta)$, so $\#S\lesssim N_2^\theta N_2N_3$.

If $(d_1-z_1)(z_1-y_1)=0$ and $(d_2-z_2)(z_2-y_2)=0$, as there is no pairing, we may assume that $d_1=z_1$ and $z_2=y_2$ (or $z_1=y_1$ and $d_2=z_2$, which is treated similarly), so $z_1=d_1$ and $x_2=d_2$ are fixed, $z_2$ has $O(N_3)$ choices and $x_1=y_1$ has $O(N_2)$ choices, which implies $\#S\lesssim N_2N_3$.

(b) Suppose $\iota_3=-$, then similarly we have that
\begin{equation*}(d_1+z_1)(d_1+y_1)+(d_2+z_2)(d_2+y_2)=(d+z)\cdot(d+y)=\frac{|d|^2+\alpha}{2}
\end{equation*} is constant. If $(d_1+z_1)(d_1+y_1)\neq 0$ (or similarly if $(d_2+z_2)(d_2+y_2)\neq 0$), then with fixed $(y_2,z_2)$ (which has $O(N_2N_3)$ choices), $(d_1+z_1)(d_1+y_1)$ will be constant. As $d_1+z_1$ belongs to an interval of size $O(N_3)$ in $\mathbb{R}$, and $d_1+y_1$ belongs to an interval of size $O(N_2)$ in $\mathbb{R}$, by (1) we know that the number of choices for $(y_1,z_1)$ is $O(N_2^\theta)$, so $\#S\lesssim N_2^\theta N_2N_3$.

If $(d_1+z_1)(d_1+y_1)=0$ and $(d_2+z_2)(d_2+y_2)=0$, as there is no pairing, we may assume that $d_1+z_1=0$ and $d_2+y_2=0$ (or $d_2+z_2$ and $d_1+y_1=0$, which is treated similarly), so $z_1=-d_1$ and $y_2=-d_2$ are fixed, $z_2$ has $O(N_3)$ choices and $y_1$ has $O(N_2)$ choices, which implies $\#S\lesssim N_2N_3$.
\end{proof}
\begin{prop}\label{counting1} Recall the relevant constants defined in (\ref{defparam}). The following bounds are uniform in all parameters. Given $d,d',k^0\in\mathbb{Z}^2$ and $k_j^0\in\mathbb{Z}^2(1\leq j\leq n)$, $\alpha,\Gamma\in\mathbb{R}$, $\iota,\iota_j\in\{\pm\}(1\leq j\leq n)$, and $2p\leq n$, as well as $M$, $N_j (0\leq j\leq n)$ and $R_i(1\leq i\leq p)$, such that for $1\leq i\leq p$ we have $N_{2i-1}\sim N_{2i}$, $\iota_{2i-1}=-\iota_{2i}$ and $R_{i}\lesssim N_{2i-1}^{1-\delta}$. Let $N^{(j)}=\max^{(j)}(N_1,\cdots,N_n)$, $N_{PR}=\max(N_1,\cdots,N_{2p})$ and let $N_*\gtrsim\max(N_0,N^{(1)})$. Also fix a subset $A$ of $\{1,\cdots,n\}$ that contains $\{1,\cdots,2p\}$, and recall the definition of the \emph{$\Gamma$-condition} (\ref{gammacon}). Consider the sets \begin{multline}\label{defset}\
S_1=\bigg\{(k,k_1,\cdots,k_n)\in(\mathbb{Z}^2)^{n+1}:\sum_{j=1}^n\iota_jk_j=k+d,\quad\sum_{j=1}^n\iota_j|k_j|^2=|k|^2+\alpha,\\|k_j-k_j^0|\lesssim N_j\,(1\leq j\leq n),\quad|k_{2i-1}-k_{2i}|\lesssim R_i(N_*)^{C\kappa^{-1}}\,(1\leq i\leq p)\bigg\},
\end{multline}
\begin{multline}\label{defset2}
S_2=\bigg\{(k,k',k_1,\cdots,k_n)\in(\mathbb{Z}^2)^{n+2}:\iota k'+\sum_{j=1}^n\iota_jk_j=k+d,\quad\iota|k'|^2+\sum_{j=1}^n\iota_j|k_j|^2=|k|^2+\alpha,\\|k|,|k'|\lesssim N_0,\,\, |k_j-k_j^0|\lesssim N_j\,(1\leq j\leq n),\,\,|k_{2i-1}-k_{2i}|\lesssim R_i(N_*)^{C\kappa^{-1}}\,(1\leq i\leq p)\bigg\},
\end{multline}
\begin{equation}\label{extraeqn}S^{+}=\bigg\{(k,k_1,\cdots,k_{n})\in(\mathbb{Z}^2)^{n+1}{\rm{\ and\ }}(k,k',k_1,\cdots,k_{n})\in(\mathbb{Z}^2)^{n+2}:\sum_{j\in A}\iota_jk_j=d'\bigg\},
\end{equation}
\begin{multline}\label{defset4}
S_3=\bigg\{(k,k_1,\cdots,k_n)\in(\mathbb{Z}^2)^{n+1}:\sum_{j=1}^n\iota_jk_j=k+d,\quad\bigg||k|^2-\sum_{j=1}^n\iota_j|k_j|^2-\alpha\bigg|\lesssim M,\quad|k|\lesssim N_0,\\|k_j|\lesssim N_j \, (1\leq j\leq n),\quad|k_{2i-1}-k_{2i}|\lesssim R_i(N_*)^{C\kappa^{-1}}\,(1\leq i\leq p),\quad {\rm{and\ }}(\ref{gammacon}){\rm{\ holds}}\bigg\}.
\end{multline}

Assume that there is no pairing among the variables $k$, $k'$ and $k_j$ in the sets above. Let $S_j^+=S_j\cap S^+$. Then, for $S_1$ we have
\begin{equation}\label{bdset1}
(\#S_1)\cdot\prod_{i=1}^p\frac{N_{2i-1}^{1+2\gamma_0}}{R_i}\lesssim (N_{PR})^{2\gamma_0}(N_*)^{C\kappa^{-1}}(N^{(1)}N^{(2)})^{-1}\prod_{j=1}^nN_j^2.
\end{equation} If $N^{(1)}\sim N_a$ and $\iota_a=-$, then we have
\begin{equation}\label{bdset2}
(\#S_1)\cdot\prod_{i=1}^p\frac{N_{2i-1}^{1+2\gamma_0}}{R_i}\lesssim (N_{PR})^{2\gamma_0}(N_*)^{C\kappa^{-1}}(N^{(1)})^{-2}\prod_{j=1}^nN_j^2.
\end{equation} For $S_2$ and $S_3$ we have \begin{equation}\label{bdset3}
(\#S_2)\cdot\prod_{i=1}^p\frac{N_{2i-1}^{1+2\gamma_0}}{R_i}\lesssim (N_{PR})^{2\gamma_0}(N_*)^{C\kappa^{-1}}N_0(N^{(1)})^{-1}\prod_{j=1}^nN_j^2,
\end{equation}\begin{equation}\label{bdset5}
(\#S_3)\cdot\prod_{i=1}^p\frac{N_{2i-1}^{1+2\gamma_0}}{R_i}\lesssim (N_{PR})^{2\gamma_0}(N_*)^{C\kappa^{-1}}M\frac{\max((N^{(2)})^2,|\alpha|)}{(N^{(2)})^2}(N^{(1)})^{-2}\prod_{j=1}^nN_j^2.
\end{equation}

Finally, suppose we replace any of these $S_j$ by the set $S_j^+$. Then (\ref{bdset1})--(\ref{bdset3}) hold with the right hand side multiplied by an extra factor
\begin{equation}\label{exfactor}\big[\min\big(N^{(2)},\max_{j\in A,j\geq 2p+1}N_j\big)\big]^{-1}.\end{equation} If $N^{(1)}\sim N_a$ and $a\in A$, then the stronger bound (\ref{bdset2}) holds for $S_1^+$ with right hand side multiplied by an extra factor (\ref{exfactor}), regardless of whether $\iota_a=-$ or not. If $N^{(1)}\sim N_a$ and $2p+1\leq a\in A$, then (\ref{bdset3}) holds for $S_2^+$ with right hand side multiplied by an extra factor $(N^{(1)})^{-1}$. As for $S_3^+$, either it satisfies (\ref{bdset5}) with the same extra factor (\ref{exfactor}), or it satisfies
\begin{multline}\label{bdset6}(\#S_3^+)\cdot\prod_{i=1}^p\frac{N_{2i-1}^{1+2\gamma_0}}{R_i}\lesssim (N_*)^{C\kappa^{-1}}M(N_{PR})^{2\gamma_0}\min\bigg(\frac{\max((N^{(2)})^2,|\alpha|)}{(N^{(1)}N^{(2)})^2}\\\times\big(\max^{(2)}\{N_j:2p+1\leq j\in A\}\big)^{-1},(N^{(1)})^{-1}(N^{(2)})^{-2}\bigg)\prod_{j=1}^nN_j^2.
\end{multline}
\end{prop}
\begin{proof} Let $a,b,c$ be such that $N^{(1)}\sim N_a$, $N^{(2)}\sim N_b$, and $\max(\{N_j:2p+1\leq j\in A\})\sim N_c$. In the proof below any factor that is $\lesssim (N_*)^{C\kappa^{-1}}$ will be negligible, so we will pretend they are $1$. For simplicity, let us first also ignore all $N_{2i-1}^{2\gamma_0}$ factors; at the end of the proof we will explain how to put them back.

(1) We start with (\ref{bdset1}). If $a,b\geq 2p+1$, we may fix all $k_j(j\not\in\{a,b\})$, and then apply Lemma \ref{lem:counting} (2) to count the triple $(k,k_{a},k_{b})$. This gives
\begin{equation}(\#S_1)\prod_{i=1}^p\frac{N_{2i-1}}{R_i}\lesssim\bigg(\prod_{i=1}^pN_{2i-1}^3R_i\prod_{2p+1\leq j\not\in\{a,b\}}N_j^2\bigg)N_{a}N_{b},\end{equation} which proves (\ref{bdset1}) as $R_{j}\lesssim N_{2j-1}^{1-\delta}$. If $a\geq 2p+1$ and $b\leq 2p$ (say $b=1$), we may fix all $k_j\,(j\not\in\{1,a\})$, and then apply Lemma \ref{lem:counting} (2) to count the triple $(k,k_{1},k_{a})$, noticing that $k_1$ belongs to a disc of radius $O(R_1(N_*)^{C\kappa^{-1}})$ once $k_2$ is fixed. This gives
\begin{equation}
(\#S_1)\prod_{i=1}^p\frac{N_{2i-1}}{R_i}\lesssim\bigg(N_1^2\prod_{i=2}^pN_{2i-1}^3R_i\prod_{2p+1\leq j\neq a}N_j^2\bigg)N_{a}R_1\frac{N_1}{R_1},
\end{equation} which proves (\ref{bdset1}). Finally, if $a\leq 2p$ (say $a=1$) then we may assume $b=2$. We may fix all $k_j(j\geq 3)$, and then apply Lemma \ref{lem:counting} (2) to count the triple $(k,k_{1},k_{2})$, noticing that $k$ belongs to a disc of radius $O(R_1(N_*)^{C\kappa^{-1}})$ once all $k_j(j\geq 3)$ are fixed. This gives
\begin{equation}(\#S_1)\prod_{i=1}^p\frac{N_{2i-1}}{R_i}\lesssim \bigg(\prod_{i=2}^pN_{2i-1}^3R_i\prod_{j\geq 2p+1}N_j^2\bigg)N_1R_1\frac{N_1}{R_1},
\end{equation} which proves (\ref{bdset1}).

As for (\ref{bdset2}) and (\ref{bdset3}) we only need to consider $a$. If $a\geq 2p+1$ we may fix all $k_j(j\neq a)$, and then apply Lemma \ref{lem:counting} (2) to count the pair $(k,k_{a})$ for (\ref{bdset2}) (using the fact $\iota_{a}=-$) and the triple $(k,k',k_{a})$ for (\ref{bdset3}), and get
\begin{equation}(\#S_1)\prod_{i=1}^p\frac{N_{2i-1}}{R_i}\lesssim \bigg(\prod_{i=1}^pN_{2i-1}^3R_i\prod_{2p+1\leq j\neq a}N_j^2\bigg) ,
\end{equation}
\begin{equation}(\#S_2)\prod_{i=1}^p\frac{N_{2i-1}}{R_i}\lesssim\bigg(\prod_{i=1}^pN_{2i-1}^3R_i\prod_{2p+1\leq j\neq a}N_j^2\bigg) N_{0}N_{a},\end{equation} which proves (\ref{bdset2}) and (\ref{bdset3}). If $a\leq 2p$ (say $a=1$) we may assume $b=2$, in particular $N^{(1)}\sim N^{(2)}$ and (\ref{bdset2}) follows from (\ref{bdset1}); for (\ref{bdset3}) we may fix all $k_j(j\geq 2)$ and then apply Lemma \ref{lem:counting} (2) to count the triple $(k,k',k_{1})$, noticing that $k_1$ belongs to a disc of radius $O(R_1(N_*)^{C\kappa^{-1}})$ once $k_2$ is fixed, and get
\begin{equation}(\#S_2)\prod_{i=1}^p\frac{N_{2i-1}}{R_i}\lesssim \bigg(N_1^2\prod_{i=2}^pN_{2i-1}^3R_i\prod_{j\geq 2p+1}N_j^2\bigg)N_0R_1\frac{N_1}{R_1},
\end{equation} which proves (\ref{bdset3}).

(2) Next we prove the improvements to (\ref{bdset1})--(\ref{bdset3}) for $S_j^{+}$. We start with (\ref{bdset1}). If $a,b\not\in A$, we may fix all $k_j (c\neq j\in A)$ and apply (\ref{bdset1}) to the rest variables and get
\begin{equation}(\#S_1^+)\prod_{i=1}^p\frac{N_{2i-1}}{R_i}\lesssim (N_aN_b)^{-1}\prod_{j\not\in A}N_j^2\prod_{i=1}^p N_{2i-1}^3R_i\prod_{j\in A,c\neq j\geq 2p+1}N_j^2,
\end{equation}
which gains a factor $N_c^{-2}$ upon (\ref{bdset1}). If $a\not\in A$ and $2p+1\leq b\in A$, we may fix all $k_j (b\neq j\in A)$ and apply (\ref{bdset1}) to the rest variables and get
\begin{equation}(\#S_1^+)\prod_{i=1}^p\frac{N_{2i-1}}{R_i}\lesssim N_a^{-1}\prod_{j\not\in A}N_j^2\prod_{i=1}^p N_{2i-1}^3R_i\prod_{j\in A,b\neq j\geq 2p+1}N_j^2,
\end{equation}  which gains a factor $N_b^{-1}$ upon (\ref{bdset1}). If $a\not\in A$ and $b\leq 2p$ (say $b=1$), we may fix all $k_j(2\leq j\neq a)$, noticing that $k_c$ belongs to a ball of radius $\min(N_c,R_1(N_*)^{C\kappa^{-1}})$ once all $k_j(3\leq j\not\in \{a,c\})$ are fixed, and then apply Lemma \ref{lem:counting} (2) to count the pair $(k,k_a)$ and get 
\begin{equation}(\#S_1^+)\prod_{i=1}^p\frac{N_{2i-1}}{R_i}\lesssim\bigg(N_1^2\min(N_c^2,R_1^2)\prod_{i=2}^pN_{2i-1}^3R_i\prod_{2p+1\leq j\not\in\{a, c\}}N_j^2\bigg)N_a\frac{N_1}{R_1},
\end{equation}which gains a factor $N_c^{-1}$ upon (\ref{bdset1}).

Now if $2p+1\leq a\in A$ and either $b\not\in A$ or $2p+1\leq b\in A$, we may fix all $k_j(j\not\in\{a,b\})$ and apply Lemma \ref{lem:counting} (2) to count the pair $(k,k_b)$ (if $b\not\in A$) or $(k_a,k_b)$ (if $2p+1\leq b\in A$), and get
\begin{equation}(\#S_1^+)\prod_{i=1}^p\frac{N_{2i-1}}{R_i}\lesssim \bigg(\prod_{i=1}^pN_{2i-1}^3R_i\prod_{2p+1\leq j\not\in\{a,b\}}N_j^2\bigg)N_b,
\end{equation} which gains a factor $N_b^{-1}$ upon the stronger bound (\ref{bdset2}). If $2p+1\leq a\in A$ and $b\leq 2p$ (say $b=1$), we may fix all $k_j(j\not\in\{a,1,2\})$ and apply Lemma \ref{lem:counting} (2) to count the triple $(k_a,k_1,k_2)$, noticing that $k_a$ belongs to a disc of radius $O(R_1(N_*)^{C\kappa^{-1}})$ once all $k_j(j\not\in\{a,1,2\})$ are fixed, and get
\begin{equation}(\#S_1^+)\prod_{i=1}^p\frac{N_{2i-1}}{R_i}\lesssim \bigg(\prod_{i=2}^pN_{2i-1}^3R_i\prod_{2p+1\leq j\neq a}N_j^2\bigg)N_1R_1\frac{N_1}{R_1},
\end{equation} which gains a factor $N_b^{-2}$ upon the stronger bound (\ref{bdset2}). Finally, if $a\leq 2p$ (say $a=1$) then we may assume $b=2$. We may fix all $k_j(j\geq 3)$, noticing that $k_c$ belongs to a disc of radius $\min(N_c,R_1(N_*)^{C\kappa^{-1}})$ once all $k_j(3\leq j\neq c)$ are fixed, and then apply Lemma \ref{lem:counting} (2) to count the pair $(k_{1},k_{2})$. This gives
\begin{equation}(\#S_1^+)\prod_{i=1}^p\frac{N_{2i-1}}{R_i}\lesssim \bigg(\min(N_c,R_1)^2\prod_{i=2}^pN_{2i-1}^3R_i\prod_{2p+1\leq j\neq c}N_j^2\bigg)N_1\frac{N_1}{R_1},
\end{equation} which gains a factor $N_c^{-1}$ upon the stronger bound (\ref{bdset2}).

As for (\ref{bdset2}) and (\ref{bdset3}) we only need to consider $a$. If $a\not\in A$ we may fix all $k_j (c\neq j\in A)$ and apply (\ref{bdset2}) (if $\iota_a=-$) or (\ref{bdset3}) to the rest variables and get
\begin{equation}(\#S_1^+)\prod_{i=1}^p\frac{N_{2i-1}}{R_i}\lesssim N_a^{-2}\prod_{j\not\in A}N_j^2\prod_{i=1}^p N_{2i-1}^3R_i\prod_{j\in A,c\neq j\geq 2p+1}N_j^2,
\end{equation}
\begin{equation}(\#S_2^+)\prod_{i=1}^p\frac{N_{2i-1}}{R_i}\lesssim N_0N_a^{-1}\prod_{j\not\in A}N_j^2\prod_{i=1}^p N_{2i-1}^3R_i\prod_{j\in A,c\neq j\geq 2p+1}N_j^2,
\end{equation}
which gains a factor $N_c^{-2}$ upon (\ref{bdset2}) or (\ref{bdset3}). If $a\in A$ then (\ref{bdset2}) follows from the above proof for (\ref{bdset1}); for (\ref{bdset3}), if $2p+1\leq a\in A$ we may fix all $k_j (a\neq j\in A)$ and apply (\ref{bdset3}) to the rest variables and get
\begin{equation}(\#S_2^+)\prod_{i=1}^p\frac{N_{2i-1}}{R_i}\lesssim N_0\prod_{j\not\in A}N_j^2\prod_{i=1}^p N_{2i-1}^3R_i\prod_{j\in A,a\neq j\geq 2p+1}N_j^2,
\end{equation} which gains a factor $N_a^{-1}$ upon (\ref{bdset3}); if $a\leq 2p$ (say $a=1$) we may fix all $k_j(2\leq j\in A)$, noticing that $k_c$ belongs to a ball of radius $\min(N_c,R_1(N_*)^{C\kappa^{-1}})$ once all $k_j(j\in A\backslash\{1,2,c\})$ are fixed, and apply (\ref{bdset3}) to the rest variables and get\begin{equation}(\#S_2^+)\prod_{i=1}^p\frac{N_{2i-1}}{R_i}\lesssim\min(N_c,R_1)^2N_1^2\frac{N_1}{R_1}N_0\prod_{j\not\in A}N_j^2\prod_{i=2}^p N_{2i-1}^3R_i\prod_{j\in A,c\neq j\geq 2p+1}N_j^2,
\end{equation} which gains a factor $N_c^{-1}$ upon (\ref{bdset3}).

(3) Now we consider (\ref{bdset5}) and its improvement. We may assume $\iota_a=+$ and $N^{(1)}\gg N^{(2)}$ (so $a\geq 2p+1$), since otherwise (\ref{bdset5}) follows from (\ref{bdset2}) or (\ref{bdset1}) and similarly for the improvement. Now let $M_0=\max(|\alpha|,(N^{(2)})^2)$, if $M\gg M_0$ then we have
\[\big||k|^2-|k_{a}|^2\big|\leq|\alpha|+\sum_{j\neq a}|k_j|^2+M\lesssim M;\] combining with (\ref{gammacon}) and Lemma \ref{lem:counting} (1) we conclude that the number of choices for $|k_a|^2$, and thus $k_a$, is $O(M)$. We may fix $k_a$ and then count $k_j(j\neq a)$ to get
\begin{equation}(\#S_3)\prod_{i=1}^p\frac{N_{2i-1}}{R_i}\lesssim M\prod_{i=1}^p N_{2i-1}^3R_i\prod_{2p+1\leq j\neq a}N_j^2,
\end{equation} which proves (\ref{bdset5}). As for $S_3^+$, if $a\in A$ then the improvement of (\ref{bdset5}) follows from the improvement of (\ref{bdset2}); if $a\not\in A$ we may fix $k_a$ and count $k_j(j\not\in\{a,c\})$ to get \begin{equation}(\#S_3^+)\prod_{i=1}^p\frac{N_{2i-1}}{R_i}\lesssim M\prod_{i=1}^p N_{2i-1}^3R_i\prod_{2p+1\leq j\not\in \{a,c\}}N_j^2,
\end{equation} which gains a factor $N_c^{-2}$ upon (\ref{bdset5}).

Assume now $M\lesssim M_0$, then just like above we have $\big||k|^2-|k_a|^2\big|\lesssim M_0$, so $k$ has at most $O(M_0)$ choices, similarly $k_a$ has at most $O(M_0)$ choices. If $b\geq 2p+1$, we may assume $\iota_b=+$ (otherwise switch the roles of $k$ and $k_a$), then fix $k$ and $k_j(j\not\in\{a, b\})$ and apply Lemma \ref{lem:counting} (2) to count the pair $(k_a,k_b)$ to get
\begin{equation}(\#S_3)\prod_{i=1}^p\frac{N_{2i-1}}{R_i}\lesssim M_0\bigg(\prod_{i=1}^p N_{2i-1}^3R_i\prod_{2p+1\leq j\not\in\{a,b\}}N_j^2\bigg)M,
\end{equation} which proves (\ref{bdset5}); if $b\leq 2p$ (say $b=1$), we may fix $k$ and $k_j(j\not\in\{1,2,a\})$ and apply Lemma \ref{lem:counting} (2) to count the triple $(k_1,k_2,k_a)$, noticing that $k_a$ belongs to a disc of radius $O(R_1(N_*)^{C\kappa^{-1}})$ once $k$ and $k_j(j\not\in\{1,2,a\})$ are fixed, and get
\begin{equation}(\#S_3)\prod_{i=1}^p\frac{N_{2i-1}}{R_i}\lesssim \bigg(M_0\prod_{i=2}^p N_{2i-1}^3R_i\prod_{2p+1\leq j\neq a}N_j^2\bigg)N_1R_1M\frac{N_1}{R_1},
\end{equation} which proves (\ref{bdset5}).

It remains to prove the improvement of (\ref{bdset5}) for $S_3^+$. We may assume $a\not\in A$, since otherwise it follows from the improvement of (\ref{bdset2}). Now if $b\not\in A$ we may fix all $k_j (c\neq j\in A)$ and apply (\ref{bdset5}) to the rest variables and get
\begin{equation}(\#S_3^+)\prod_{i=1}^p\frac{N_{2i-1}}{R_i}\lesssim MM_0(N_aN_b)^{-2}\prod_{j\not\in A}N_j^2\prod_{i=1}^pN_{2i-1}^3R_i\prod_{2p+1\leq j\neq c,\,j\in A}N_j^2,
\end{equation} which gains a factor $N_c^{-2}$ upon (\ref{bdset5}). If $b\leq 2p$, say $b=1$, we may fix $k$ and $k_j(3\leq j\neq a)$, noticing that $k_c$ belongs to a disc of radius $\min(N_c,R_1(N_*)^{C\kappa^{-1}})$ once all $k_j(3\leq j\not\in\{ a,c\})$ are fixed, and then apply Lemma \ref{lem:counting} (2) to count the pair $(k_1,k_2)$ and get
\begin{equation}(\#S_3^+)\prod_{i=1}^p\frac{N_{2i-1}}{R_i}\lesssim MM_0\bigg(\min(N_c,R_1)^2\prod_{i=2}^pN_{2i-1}^3R_i\prod_{2p+1\leq j\not\in\{ a,c\}}N_j^2\bigg)N_1\frac{N_1}{R_1},
\end{equation} which gains a factor $N_c^{-1}$ upon (\ref{bdset5}). Finally, assume $2p+1\leq b\in A$, then we will prove (\ref{bdset6}). Let $\max^{(2)}\{N_j:2p+1\leq j\in A\}\sim N_d$, we may fix $k$ and $k_j(j\not\in \{a,b,d\})$, then apply Lemma \ref{lem:counting} (2) to count the pair $(k_b,k_d)$ and get
\begin{equation}\label{lastone}(\#S_3^+)\prod_{i=1}^p\frac{N_{2i-1}}{R_i}\lesssim MM_0\bigg(\prod_{i=1}^pN_{2i-1}^3R_i\prod_{2p+1\leq j\not\in\{ a,b,d\}}N_j^2\bigg)N_d;
\end{equation}alternatively we may choose to fix $k_j(j\not\in\{a,b\})$ then apply Lemma \ref{lem:counting} (2) to count the pair $(k,k_a)$ and get \begin{equation}\label{lasttwo}(\#S_3^+)\prod_{i=1}^p\frac{N_{2i-1}}{R_i}\lesssim M\bigg(\prod_{i=1}^pN_{2i-1}^3R_i\prod_{2p+1\leq j\not\in\{ a,b\}}N_j^2\bigg)N_a,
\end{equation} and combining (\ref{lastone}) and (\ref{lasttwo}) yields (\ref{bdset6}).

In the last part we will explain how to put back the powers $N_{2i-1}^{2\gamma_0}$. In fact, in each estimate above we have the product $\prod_{i\geq 2}N_{2i-1}^3R_i$. As $R_i\lesssim N_{2i-1}^{1-\delta}$ and $N_{2i-1}\sim N_{2i}$ we have
\[N_{2i-1}^3R_i\lesssim N_{2i-1}^2N_{2i}^2 \cdot N_{2i-1}^{-2\gamma_0},\] which allows us to incorporate the extra factor $N_{2i-1}^{2\gamma_0}$ for $i\geq 2$. Thus we lose at most a factor $N_{1}^{2\gamma_0}$, which is acceptable as $N_1\lesssim N_{PR}$.
\end{proof}
\begin{cor}\label{corcounting} Recall $a_0>1$ defined in (\ref{defparam}), and let all the parameters ($d$, $N_j$, $\iota_j$ etc.) be as in Proposition \ref{counting1}. From the sets $S_j(1\leq j\leq 3)$ in Proposition \ref{counting1} we may construct the quantities $\mathcal{E}_j$ as follows: each $\mathcal{E}_j$ is a sum over a set ${S}_j^{\mathrm{lin}}$. This ${S}_j^{\mathrm{lin}}$ is formed from $S_j$ by removing from its defining properties the one that involves the quadratic algebraic sum $\Sigma$ (this $\Sigma$ is $\iota_1|k_1|^2+\cdots +\iota_n|k_n|^2-|k|^2$ for $S_1$ and $S_3$, and $\iota_1|k_1|^2+\cdots +\iota_n|k_n|^2+\iota|k'|^2-|k|^2$ for $S_2$), and the summand is simply $\langle \Sigma-\alpha \rangle^{-a_0}$. Similarly define $\mathcal{E}_j^+$ by replacing $S_j$ with $S_j^+$.

Then, the inequalities (\ref{bdset1})--(\ref{bdset6}), as well as their improvements, hold with $\# S_j$ replaced by $\mathcal{E}_j$ $(\#S_j^+$ replaced by $\mathcal{E}_j^{+}$), and with the factor $M$ on the right hand sides of (\ref{bdset5}) and (\ref{bdset6}) removed.
\end{cor}
\begin{proof} This is straightforward, by applying Proposition \ref{counting1} for each value of $\Sigma$ and summing up using $a_0>1$ for (\ref{bdset1})--(\ref{bdset3}), and by dyadically decomposing $\langle\Sigma-\alpha\rangle$ and applying Proposition \ref{counting1} for each dyadic piece for (\ref{bdset5})--(\ref{bdset6}).
\end{proof}

\section{Proof of the multilinear estimates}\label{multiest} In this section we will prove Proposition \ref{multi0} thus completing the local theory. We start with an estimate for general multilinear forms without pairing. 

Given $d\in\mathbb{Z}^2$ and $\alpha\in\mathbb{R}$, consider the following expressions: 
\begin{equation}\label{mainexp1}
\mathcal{X}:=\sum_{\substack{(k,k_1,\cdots,k_n)\\\iota_1 k_1+\cdots+\iota_n k_n=k+d}}\int\mathrm{d}\lambda\mathrm{d}\lambda_1\cdots\mathrm{d}\lambda_n\cdot\eta\bigg(\lambda,\lambda-|k|^2-\sum_{j=1}^n\iota_j(\lambda_j-|k_j|^2)-\alpha\bigg)\overline{v_{k}(\lambda)}\prod_{j=1}^n[v_{k_j}^{(j)}(\lambda_j)]^{\iota_j},
\end{equation} 
\begin{multline}\label{mainexp2}
\mathcal{Y}:=\sum_{\substack{(k,k',k_1,\cdots,k_n)\\\iota_1 k_1+\cdots+\iota_n k_n+\iota k'=k+d}}\int \mathrm{d}\lambda\mathrm{d}\lambda'\mathrm{d}\lambda_1\cdots\mathrm{d}\lambda_n\\\times\eta\bigg(\lambda,\lambda-|k|^2-\iota(\lambda'-|k'|^2)-\sum_{j=1}^n\iota_j(\lambda_j-|k_j|^2)-\alpha\bigg)y_{kk'}(\lambda,\lambda')\prod_{j=1}^n[v_{k_j}^{(j)}(\lambda_j)]^{\iota_j},
\end{multline} where $d\in\mathbb{Z}^2$ and $\alpha\in\mathbb{R}$ are fixed, $\eta$ is a function that satisfies
\begin{equation}\label{kernelest2}|\eta(\lambda,\mu)|+|\partial_{\lambda,\mu}\eta(\lambda,\mu)|\lesssim\langle\mu\rangle^{-10}.
\end{equation} In the summation we always assume that there is \emph{no pairing}\footnote{This requirement appears in the form of coefficients which are indicator functions of sets of form $\{k_j\neq k_l\}$. Such coefficients may lead to slightly different multilinear Gaussian expressions in the estimates below, but there will be at most $(N_*)^C$ possibilities where $N_*$ is a parameter to be defined below, and will not affect any estimates since our exceptional sets will always have measure at most $C_\theta e^{-(N_*)^\theta}$.}  among the variables $k$, $k'$ and $k_j$.

We assume that the input functions $v^{(j)}$ are as in Proposition \ref{multi0}, where $v^{(j)}$ are of type (G) or (C) for $1\leq j\leq n_1$, and of type (D) for $n_1+1\leq j\leq n$. Since we are working exclusively in the $\lambda_j$ spaces, we will abuse notations here and write $(v_{k_j}^{(j)})(\lambda_j)$ instead of $(\widetilde{v^{(j)}})_{k_j}(\lambda_j)$.  

Let the parameters $N_j$, $L_j$, $N^{(j)}$ etc., and the sets $\mathcal{G}$ and $\mathcal{C}$ be as in that proposition. We further assume that the functions $v_k(\lambda)$ and $y_{kk'}(\lambda,\lambda')$ satisfy \begin{equation}\label{input0}\|\langle\lambda\rangle^bv_{k}(\lambda)\|_{\ell_{k}^2L_{\lambda}^2}\lesssim 1,\quad  \|\langle\lambda\rangle^b\langle \lambda'\rangle^by_{kk'}(\lambda,\lambda')\|_{\ell_{k,k'}^2L_{\lambda,\lambda'}^2}\lesssim 1,
\end{equation} and that $v_{k}(\lambda)$ is supported in $\{|k|\lesssim N_0\}$ and $y_{kk'}(\lambda,\lambda')$ is supported in $\{|k|,|k'|\lesssim N_0\}$.
\begin{prop}\label{general} Recall the relevant constants defined in (\ref{defparam}), and that $\tau\ll 1$. Under all the above assumptions, there exist $p$ and $q$, and $N_{2p+l}\gtrsim R_{2p+l}\gtrsim L_{2p+l}(1\leq l\leq q)$ such that $2p+q\leq n_1$, that for $1\leq i\leq p$ we must have $N_{2i-1}\sim N_{2i}$ and $\iota_{2i-1}=-\iota_{2i}$, and that $2i-1$ and $2i$ do not both belong to $\mathcal{G}$. Define $R_i=\max(L_{2i-1},L_{2i})$ and let $N_*$ be fixed, then the following estimates hold $\tau^{-1}N_*$-certainly. Here, as in Proposition \ref{multi0}, the exceptional set of $\omega$ removed does \emph{not} depend on the choice of the functions $v^{(j)}(j\geq n_1+1)$ or $v$ or $w$.
\smallskip

(1) Assume $N_*\gtrsim\max(N_0,N^{(1)})$. Then we have
\begin{equation}\label{estimate1}|\mathcal{X}|^2\lesssim \tau^{-\theta}(N_*)^{C\kappa^{-1}}\mathcal{E}_1\prod_{i=1}^p\frac{N_{2i-1}^{1+2\gamma_0}}{R_i}\prod_{j=1}^nN_j^{-2}\prod_{j\geq n_1+1}N_j^{2\gamma}\prod_{2p+1\leq j\leq n_1}L_j^{-2\delta_0},
\end{equation} and similarly
\begin{equation}\label{estimate2}|\mathcal{Y}|^2\lesssim \tau^{-\theta}(N_*)^{C\kappa^{-1}}\mathcal{E}_2\prod_{i=1}^p\frac{N_{2i-1}^{1+2\gamma_0}}{R_i}\prod_{j=1}^nN_j^{-2}\prod_{j\geq n_1+1}N_j^{2\gamma}\prod_{2p+1\leq j\leq n_1}L_j^{-2\delta_0},
\end{equation} where $\mathcal{E}_1$ and $\mathcal{E}_2$ are the quantities defined in Corollary \ref{corcounting}, with $a_0=2b-10\delta^6$, and for \emph{some} choice of the parameters in that corollary that do not appear in the assumptions of the current proposition. Moreover, if in the sum defining $\mathcal{X}$ we also assume the $\Gamma$-condition (\ref{gammacon}), then (\ref{estimate1}) holds with $\mathcal{E}_1$ replaced by $\mathcal{E}_3$ (see Corollary \ref{corcounting} for the relevant definitions).
\smallskip

(2) Assume $N_*\gtrsim\max(N_0,N^{(1)})$. Then we have
\begin{equation}\label{estimate3}|\mathcal{X}|^4\lesssim \tau^{-\theta}(N_*)^{C\kappa^{-1}} \mathcal{E}_1\mathcal{E}_{1}^{+}\bigg(\prod_{i=1}^p\frac{N_{2i-1}^{1+2\gamma_0}}{R_i}\bigg)^2\prod_{j=1}^nN_j^{-4}\prod_{j\geq n_1+1}N_j^{4\gamma}\prod_{2p+1\leq j\leq n_1}L_j^{40n^2},
\end{equation} and similarly
\begin{equation}\label{estimate4}|\mathcal{Y}|^4\lesssim  \tau^{-\theta}(N_*)^{C\kappa^{-1}}\mathcal{E}_2\mathcal{E}_{2}^{+}\bigg(\prod_{i=1}^p\frac{N_{2i-1}^{1+2\gamma_0}}{R_i}\bigg)^2\prod_{j=1}^nN_j^{-4}\prod_{j\geq n_1+1}N_j^{4\gamma}\prod_{2p+1\leq j\leq n_1}L_j^{40n^2},
\end{equation} where $\mathcal{E}_j$ and $\mathcal{E}_j^{+}$ are the quantities defined in Corollary \ref{corcounting}, again for some choice of the parameters in that corollary that do not appear in the assumptions of the current proposition. In the set $S^{+}$ in (\ref{extraeqn}) the set $A$ will contain $\{1,2,\cdots,2p\}\cup \{n_1+1,\cdots,n\}$. Moreover, if in the sum defining $\mathcal{X}$ we also assume the $\Gamma$-condition (\ref{gammacon}), then (\ref{estimate3}) holds with $\mathcal{E}_1\mathcal{E}_{1}^{+}$ replaced by $\mathcal{E}_3\mathcal{E}_{3}^{+}$.

(3) Assume in addition that $N^{(1)}\sim N_n$ and $N_*\gtrsim N^{(2)}$. Then (\ref{estimate3}) is true, with $N_n$ replaced by $N^{(2)}$ in both quantities $\mathcal{E}_1$ and $\mathcal{E}_1^{+}$. Moreover we have
\begin{equation}\label{estimate5}|\mathcal{X}|^4\lesssim \tau^{-\theta}(N_*)^{C\kappa^{-1}}(N^{(1)})^{-4(1-\gamma)}(N^{(2)})^{C\gamma}\widetilde{\mathcal{E}_1}\widetilde{\mathcal{E}_{1}^{+}}\bigg(\prod_{i=1}^p\frac{N_{2i-1}^{1+2\gamma_0}}{R_i}\bigg)^2\prod_{j=1}^{n-1}N_j^{-4}\prod_{j=2p+1}^{2p+q}N_j^{2}R_j^{-2},
\end{equation}where $\widetilde{\mathcal{E}_1}$ is the quantity defined in Corollary \ref{corcounting}, for some choice of the parameters in that corollary that do not appear in the assumptions of the current proposition, but with $N_n$ replaced by $N^{(2)}$. Similarly $\widetilde{\mathcal{E}_{1}^{+}}$ is the quantity $\mathcal{E}_1^+$ defined in Corollary \ref{corcounting} with $A=\{1,\cdots,n\}$, but with $N_n$ replaced by $N^{(2)}$ and $N_{2p+l}$ replaced by $R_{2p+l}(N_*)^{C\kappa^{-1}}$ for $1\leq l\leq q$. Moreover, the exceptional set of $\omega$ removed is independent of $N^{(1)}$.
\end{prop} 
\subsection{Proof of Proposition \ref{general}} We will prove Proposition \ref{general} in this section. We will only prove the bounds for $\mathcal{X}$ without $\Gamma$-condition; with obvious modifications the proof also works for $\mathcal{Y}$ and for the version with $\Gamma$-condition. For simplicity we will omit the $\omega$ dependence, and may ignore any factors that are $\lesssim \tau^{-\theta}(N_*)^{C\kappa^{-1}}$.

Our proof will roughly follow an algorithm, indicated by the following steps. (1) Distinguish between the inputs $j\in\mathcal{D}$, where $v^{(j)}$ are bounded in $\ell^2L^2$, with $j\in\mathcal{G}\cup\mathcal{C}$. (2) Identify the pairings among $k_j(j\in\mathcal{G})$ and $k_j^*(j\in\mathcal{C})$ and reduce the sum of products of the $h^{(j)}$ functions over the paired variables to some functions $P^{(i)}$, see (\ref{contraction}), that are also bounded in $\ell^2L^2$. (3) Estimate the sum in unpaired variables using Lemma \ref{lemma:4.2} (in Section \ref{case2} we will skip step (2) and estimate the whole sum including paired and unpaired variables using Lemma \ref{largedev}). (4) Apply Cauchy-Schwartz to handle all the factors in $\ell^2L^2$, and {then reduce to the $\mathcal{E}_j$ type quantities in Corollary \ref{corcounting}. (5) When necessary, apply a $\mathcal{T}^*\mathcal{T}$ argument and repeat the previous steps for the resulting kernel.

As the proof will be notation heavy, the reader may do a first reading making the following simplifications without missing the core parts of the proof: (1) omit integration in any $\lambda_j$ and pretend $\lambda_j=0$ (so $v^{(j)}$ is a function of $k_j$ only and $h^{(j)}$ is a function of $k_j$ and $k_j^*$ only); (2) when identifying pairings, restrict to only simple pairings where (say) $k_i^*=k_j^*$ and does not equal any other $k_l^*$. These will make formulas like (\ref{contbd}) simpler and the proofs more transparent.

Throughout the proof we will fix the sets $U=\{1,2,\cdots,n\}$ and $V=\{1,2,\cdots,n-1\}$. We will (in this section only) introduce a shorthand notation for vectors: for a finite set $X$, define $k_{[X]}$ to be the vector $(k_j:j\in X)$; similarly define $\lambda_{[X]}$, $k_{[X]}^*$, etc., and define $\mathrm{d}\lambda_{[X]}=\prod_{j\in X}\mathrm{d}\lambda_j$.
\subsubsection{A simple bound}\label{case1} We first prove (\ref{estimate1}). By definition we expand
\begin{multline}\label{reduce1}
\mathcal{X}=\sum_{(k,k_{[U]}):\,\iota_1 k_1+\cdots+\iota_n k_n=k+d}\,\sum_{k_{[U]}^*}\int\mathrm{d}\lambda\mathrm{d}\lambda_{[U]}\\\times\eta\bigg(\lambda,\lambda-|k|^2-\sum_{j=1}^n\iota_j(\lambda_j-|k_j|^2)-\alpha\bigg)\overline{v_{k}(\lambda)}\prod_{j=1}^{n_1}\, \frac{g_{k_j^*}^{\iota_j}}{\langle k_j^*\rangle} \, h_{k_jk_j^*}^{(j)}(\lambda_j)^\pm\prod_{j=n_1+1}^n[v_{k_j}^{(j)}(\lambda_j)]^{\iota_j};
\end{multline} recall that $k_{[U]}$ means $(k_1,\cdots,k_n)$, etc. The sum in $k_{[U]}^*$ is restricted to $\frac{N_j}{2}<\langle k_j^*\rangle\leq N_j$, and $h_{k_jk_j^*}^{(j)}(\lambda_j,\omega)$ is defined as in (\ref{input2}) for $j\in\mathcal{C}$, and is defined to be $\mathbf{1}_{k_j=k_j^*}\widehat{\chi}(\lambda_j)$ for $j\in\mathcal{G}$.

Consider now the sum in $k_{[U]}^*$. By identifying all pairings among them (recall the definition of pairings in Definition \ref{dfpairing}), we may assume there are $p$ sets $Y_i(1\leq i\leq p, 2p\leq n_1)$ and a set $Z$ that partitions $\{1,\cdots,n_1\}$, such that:  (i)  each $Y_i$ contains a pairing, (ii) the $k_j^*$ takes a single value for $j$ in each $Y_i$, (iii) this value is different for different $Y_i$ and is different from $k_j^*$ for $j\in Z$, and (iv) there is no pairing in $\{k_j^*:j\in Z\}$.  Then we manipulate this sum and rewrite it as a combination of two types of sums, namely (1) where we only require\footnote{That is, we relax the requirement (iii) above, keeping only requirements (ii) and (iv).} that $k_j^*$ takes a single value for $j$ in each $Y_i$ and that there is no pairing in $\{k_j^*:j\in Z\}$, and (2) where there are more pairings in addition to case (1), namely when the value for $Y_i$ equals the value for some other $Y_{i'}$ or some $k_j^*$ for $j\in Z$. Since there are strictly more pairings in case (2) than in the sum we started with, we may repeat this process and eventually reduce to sums of type (1) only. The purpose of this manipulation is to ensure that the sum in $k_j^*\,(j\in Y_i)$ gives \emph{exactly}
\begin{equation}\label{contraction}P_{k_{[Y_i]}}^{(i)}(\lambda_{[Y_i]})=\sum_{k^*}\prod_{j\in Y_i}h_{k_jk^*}^{(j)}(\lambda_j)^{\pm}\langle k^*\rangle^{-q_3}(g_{k^*})^{q_1}(\overline{g_{k^*}})^{q_2},
\end{equation} where $q_1+q_2=q_3=|Y_i|$.

Note that $N_j$ for $j\in Y_i$ are all comparable. Without loss of generality we may assume $\{2i-1,2i\}\subset Y_i$ and $\iota_{2i-1}=-\iota_{2i}$. As $(k_{2i-1},k_{2i})$ is \emph{not} a pairing, $2i-1$ and $2i$ cannot both belong to $\mathcal{G}$. Now we may assume $|k_{2i-1}-k^*|\lesssim L_{2i-1}(N_*)^{C\kappa^{-1}}$ and similarly for $k_{2i}$, since otherwise we gain a power $(N_*)^{-200n^2}$ due to the last bound in (\ref{input2+}), which cancels any summation in any $(k_j,k_j^*)$ and the estimate will then follow immediately. 

Let $R_i=\max(L_{2i-1},L_{2i})$, say $R_i=L_{2i-1}$, then we have $|k_{2i-1}-k_{2i}|\lesssim R_i(N_*)^{C\kappa^{-1}}$ for $1\leq i\leq p$. For (\ref{contraction}) using the first two bounds in (\ref{input2+}),  we have that
\begin{align}
&\quad\,\,\|P_{k_{[Y_i]}}^{(i)}(\lambda_{[Y_i]})\prod_{j\in Y_i}\langle \lambda_j\rangle^b\|_{\ell_{k_{[Y_i]}}^2L_{\lambda_{[Y_i]}}^2}^2\nonumber\\&\lesssim\prod_{j\in Y_i}N_j^{-2}\cdot\|\langle \lambda_{2i}\rangle^bh_{k_{2i}k^*}^{(2i)}(\lambda_{2i})\|_{\ell_{k^*}^2\to\ell_{k_{2i}}^2L_{\lambda_{2i}}^2}^2\sum_{k^*}\prod_{2i\neq j\in Y_i}\|\langle \lambda_j\rangle^bh_{k_jk^*}^{(j)}(\lambda_j)\|_{\ell_{k_j}^2L_{\lambda_j}^2}^2\nonumber\\&\lesssim \|\langle \lambda_{2i-1}\rangle^bh_{k_{2i-1}k^*}^{(2i-1)}(\lambda_{2i-1})\|_{\ell_{k_{2i-1}k^*}^2L_{\lambda_{2i-1}}^2}^2\prod_{j\in Y_i}N_j^{-2}\prod_{\substack{j\in Y_i\\j\neq 2i-1,2i}}L_{j}^{-2\delta_0}\nonumber\\\label{contbd}&\lesssim N_{2i-1}^{1+2\gamma_0}\, R_i^{-1}\, \prod_{j\in Y_i}N_j^{-2} \prod_{\substack{j\in Y_i\\j\neq 2i-1,2i}}L_{j}^{-2\delta_0}.
\end{align}
 Now we have reduced the expression for $\mathcal{X}$ to
\begin{multline}\label{reduce2}
\mathcal{X}=\sum_{\substack{(k,k_{[U]}):\\\iota_1 k_1+\cdots+\iota_n k_n=k+d}}\sum_{k_{[Z]}^*}\int\mathrm{d}\lambda\mathrm{d}\lambda_{[U]}\cdot\eta\bigg(\lambda,\lambda-|k|^2-\sum_{j=1}^n\iota_j(\lambda_j-|k_j|^2)-\alpha\bigg)\\\times\overline{v_{k}(\lambda)}\prod_{i=1}^pP_{k_{[Y_i]}}^{(i)}(\lambda_{[Y_i]})\prod_{j\in Z}\frac{g_{k_j^*}^{\iota_j}}{\langle k_j^*\rangle}h_{k_jk_j^*}^{(j)}(\lambda_j)^\pm\prod_{j=n_1+1}^n[v_{k_j}^{(j)}(\lambda_j)]^{\iota_j}.
\end{multline} Compared to (\ref{reduce1}) it is important that there is no pairing in $k_{[Z]}^*$. For simplicity of notations we will write
\begin{equation}\label{simpleexp}\mathcal{X}=\sum_{(1)}\int\mathfrak{F}\cdot\mathfrak{G},\end{equation} where the symbol $\sum_{(1)}\int$ represents the sum in $k$ and $k_{[U\backslash Z]}$ and integration in $\lambda$ and $\lambda_{[U\backslash Z]}$, and the factor $\mathfrak{F}$ is
\begin{equation}\mathfrak{F} :=\overline{v_{k}(\lambda)}\prod_{i=1}^pP_{k_{[Y_i]}}^{(i)}(\lambda_{[Y_i]})\prod_{j=n_1+1}^n[v_{k_j}^{(j)}(\lambda_j)]^{\iota_j},\end{equation} and the multilinear Gaussian $\mathfrak{G}$ given by
\begin{equation}\mathfrak{G} :=\sum_{k_{[Z]}}\sum_{k_{[Z]}^*}\int\mathrm{d}\lambda_{[Z]}\cdot\prod_{j\in Z}\frac{g_{k_j^*}^{\iota_j}}{\langle k_j^*\rangle}\langle\lambda_j\rangle^bh_{k_jk_j^*}^{(j)}(\lambda_j)^\pm\cdot\mathfrak{A},
\end{equation} with coefficient $\mathfrak{A}$ of form
\begin{equation}\label{functiona}\mathfrak{A}:=\mathbf{1}_{\sum_{j\in Z}\iota_jk_j=d_0}\cdot\eta\bigg(\lambda,\alpha_0-\sum_{j\in Z}\iota_j(\lambda_j-|k_j|^2)\bigg)\prod_{j\in Z}\langle \lambda_j\rangle^{-b},
\end{equation} where
\begin{equation}\label{defnewd}d_0 :=k+d-\sum_{j\not\in Z}\iota_jk_j\in\mathbb{Z}^2,\quad \alpha_0:=\lambda-|k|^2-\alpha-\sum_{j\not\in Z}\iota_j(\lambda_j-|k_j|^2)\in\mathbb{R}.\end{equation} The goal now is to estimate $\mathfrak{G}$. For fixed values of $(k,\lambda,k_{[U\backslash Z]},\lambda_{[U\backslash Z]})$, we may apply Lemma \ref{lemma:4.2}; in order to make this uniform, we will apply the meshing argument in Remark \ref{mesh}. This allows us to reduce to at most $(N_*)^{C\delta^{-7}}$ choices, so in the end, after removing a set of probability $\leq C_{\theta}e^{-(\tau^{-1}N_*)^\theta}$ we can apply Lemma \ref{lemma:4.2} for \emph{all} choices of $(k,\lambda,k_{[U\backslash Z]},\lambda_{[U\backslash Z]})$, and use the first bound in (\ref{input2+}) to get that
\begin{equation}\label{boundforg}|\mathfrak{G}|^2\lesssim\prod_{j\in Z}N_j^{-2}\prod_{j\in Z}L_j^{-2\delta_0}\cdot\|\mathfrak{A}\|_{\mathcal{L}}^2\\\lesssim\prod_{j\in Z}N_j^{-2}\prod_{j\in Z}L_j^{-2\delta_0}\sum_{k_{[Z]}:\sum_{j\in Z}\iota_jk_j=d_0}\bigg\langle\alpha_0+\sum_{j\in Z}\iota_j|k_j|^2\bigg\rangle^{-a_0}.
\end{equation} Finally applying Cauchy-Schwartz in the variables $(k,\lambda,k_{[U\backslash Z]},\lambda_{[U\backslash Z]})$, we deduce that
\[|\mathcal{X}|^2\lesssim\bigg(\sum_{(1)}\int\langle\lambda\rangle^{2b}\prod_{j\in U\backslash Z}\langle\lambda_j\rangle^{2b}\cdot|\mathfrak{F}|^2\bigg)\bigg(\sum_{(1)}\int\langle\lambda\rangle^{-2b}\prod_{j\in U\backslash Z}\langle\lambda_j\rangle^{-2b}\cdot|\mathfrak{G}|^2 \bigg),\] where the first parenthesis (together with some factors from the second parenthesis) gives the product of all factors in (\ref{estimate1}) except $\mathcal{E}_1$, by using (\ref{input3}), (\ref{input0}) and (\ref{contbd}); the second parenthesis, after applying (\ref{boundforg}), integrating in $(\lambda,\lambda_{[U\backslash Z]})$ and pugging in (\ref{defnewd}), reduces to
\[\sum_{\substack{(k,k_{[U]}):\\\iota_1k_1+\cdots +\iota_nk_n=k+d}}\langle \Sigma-\alpha\rangle^{-a_0}\lesssim\mathcal{E}_1,\] where $\Sigma=\iota_1|k_1|^2+\cdots+\iota_n|k_n|^2-|k|^2$ as in Corollary \ref{corcounting}. This proves (\ref{estimate1}).
\subsubsection{A general $\mathcal{T}^*\mathcal{T}$ argument}\label{case2} Now we prove (\ref{estimate3}), starting from (\ref{reduce2}). Note that due to (\ref{input3}) and (\ref{input0}), the bound for $\mathcal{X}$ would follow from the $\ell_{k_n}^2L_{\lambda_n}^2\to \ell_{k}^2L_{\lambda}^2$ bound of the linear operator $\mathcal{T}$ with kernel
\begin{multline}\label{kernel0}\mathcal{T}_{kk_n}(\lambda,\lambda_n)=\sum_{k_{[V]}:\iota_1 k_1+\cdots+\iota_n k_n=k+d}\,\sum_{k_{[Z]}^*}\int\mathrm{d}\lambda_{[V]}\cdot\eta\bigg(\lambda,\lambda-|k|^2-\sum_{j=1}^n\iota_j(\lambda_j-|k_j|^2)-\alpha\bigg)\\\times\bigg(\prod_{i=1}^pP_{k_{[Y_i]}}^{(i)}(\lambda_{[Y_i]})\prod_{j\in Z}\frac{g_{k_j^*}^{\iota_j}}{\langle k_j^*\rangle}h_{k_jk_j^*}^{(j)}(\lambda_j)^\pm\prod_{j=n_1+1}^{n-1}[v_{k_j}^{(j)}(\lambda_j)]^{\iota_j}\bigg)\langle\lambda\rangle^{-b}\langle\lambda_n\rangle^{-b}.
\end{multline} We then calculate the kernel of $\mathcal{O}=\mathcal{T}^*\mathcal{T}$, which (similar to (\ref{simpleexp})) can be written as
\begin{equation}\label{simpleexp2}\mathcal{O}_{k_nk_n'}(\lambda_n,\lambda_n')=\langle\lambda_n\rangle^{-b}\langle\lambda_n'\rangle^{-b}\sum_{(2)}\int\mathfrak{F}\cdot\mathfrak{G},
\end{equation}where the symbol $\sum_{(2)}\int$ represents the sum in $(k_{[V\backslash Z]},k_{[V\backslash Z]}')$ and integration in $(\lambda_{[V\backslash Z]},\lambda_{[V\backslash Z]}')$, the factor $\mathfrak{F}$ is independent of $(k_n,k_n',\lambda_n,\lambda_n')$, and is now defined as
\begin{equation}\mathfrak{F}:=\prod_{i=1}^p\overline{P_{k_{[Y_i]}}^{(i)}(\lambda_{[Y_i]})}P_{k_{[Y_i]}'}^{(i)}(\lambda_{[Y_i]}')\prod_{j=n_1+1}^{n-1}\overline{[v_{k_j}^{(j)}(\lambda_j)]^{\iota_j}}[v_{k_j'}^{(j)}(\lambda_j')]^{\iota_j},
\end{equation} and the multilinear Gaussian $\mathfrak{G}$ is now given by
\begin{equation}\label{newdefg}\mathfrak{G}:=\sum_{(k_{[Z]}^*,k_{[Z]}'^*)}\prod_{j\in Z}g_{k_{j}^*}^{-\iota_j}g_{k_j'^*}^{\iota_j}\sum_{\substack{(k,k_{[Z]},k_{[Z]}'):\\\sum_{j\in Z}\iota_jk_j=k+d_0\\\sum_{j\in Z}\iota_jk_j'=k+d_0'}}\mathfrak{C},
\end{equation} with coefficient $\mathfrak{C}$ of form
\begin{multline}\mathfrak{C} :=\int\mathrm{d}\lambda\mathrm{d}\lambda_{[Z]}\mathrm{d}\lambda_{[Z]}'\cdot\langle\lambda\rangle^{-2b}\eta\bigg(\lambda,\lambda-|k|^2-\sum_{j\in Z}\iota_j(\lambda_j-|k_j|^2)-\alpha_0\bigg)\\\times\eta\bigg(\lambda,\lambda-|k|^2-\sum_{j\in Z}\iota_j(\lambda_j'-|k_j'|^2)-\alpha_0'\bigg)
\prod_{j\in Z}\overline{\frac{1}{\langle k_j^*\rangle}h_{k_jk_j^*}^{(j)}(\lambda_j)^\pm}\frac{1}{\langle k_j'^*\rangle}h_{k_j'k_j'^*}^{(j)}(\lambda_j')^\pm,
\end{multline}
where we now have \begin{equation}\label{defnewd2}d_0:=d-\sum_{j\not\in Z}\iota_jk_j,\,\,d_0':=d-\sum_{j\not\in Z}\iota_jk_j';\quad \alpha_0:=\alpha+\sum_{j\not\in Z}\iota_j(\lambda_j-|k_j|^2),\,\,\alpha_0':=\alpha+\sum_{j\not\in Z}\iota_j(\lambda_j'-|k_j'|^2).\end{equation}
As in Section \ref{case1} we may assume $|k_{2i-1}-k_{2i}|\lesssim R_i(N_*)^{C\kappa^{-1}}$ for $1\leq i\leq p$ and similarly for $k_{2i-1}'$ and $k_{2i}'$. The goal now is to estimate $\mathfrak{G}$ in \eqref{newdefg}. Let $L_+=\max\{L_j:j\in Z\}$, in view of the power $(L_+)^{40n^2}$ on the right hand side of (\ref{estimate3}), we may assume $N_j\gg (L_+)^2$ for each $j\in Z$, otherwise we simply sum over $(k_j,k_j^*)$ and $(k_j',k_j'^*)$ and get rid of these variables. By the meshing argument in Remark \ref{mesh}, we may reduce to $\lesssim (N_*)^{C\delta^{-7}}$ choices for $(k_{[U\backslash Z]},k_{[U\backslash Z]}',\lambda_{[U\backslash Z]},\lambda_{[U\backslash Z]}')$; for each single choice, as $\mathfrak{C}$ is $\mathcal{B}_{\leq L_+}^+$ measurable and there is no pairing in $k_{[Z]}^*$ or $k_{[Z]}'^*$, we may apply Lemma \ref{largedev} and get
\begin{equation}\label{individual}|\mathfrak{G}|^2\lesssim\sum_{(k_{[Z\backslash W]}^*,k_{[Z\backslash W']}'^*)}\bigg(\sum_{k_{a_l}^*=k_{b_l}'^*(1\leq l\leq s)}\sum_{\substack{(k,k_{[Z]},k_{[Z]}'):\\\sum_{j\in Z}\iota_jk_j=k+d_0\\\sum_{j\in Z}\iota_jk_j'=k+d_0'}}|\mathfrak{C}|\bigg)^2,
\end{equation} where $W=\{a_1,\cdots,a_s\}$ and $W'=\{b_1,\cdots, b_s\}$ are subsets of $Z$, and we have $N_{a_l}\sim N_{b_l}$ and $\iota_{a_l}=\iota_{b_l}$ for $1\leq l\leq s$. As before we may assume $|k_j-k_j^*|\lesssim L_+(N_*)^{C\kappa^{-1}}$ and similarly for $k_j'-k_j'^*$, and due to the $(L_+)^{40n^2}$ factor we may then fix the values of $k_j-k_j^*=e_j$ and $k_j'-k_j'^*=e_j'$. Therefore $k_{b_l}'-k_{a_l}=e_{b_l}'-e_{a_l}:=f_l$ is also fixed. 

Now the outer sum in (\ref{individual}) can be viewed as a sum over $k_{[Z\backslash W]}$ and $k_{[Z\backslash W']}'$, and the inner sums can be viewed as a sum over $(k,k_{a_l},k_{b_l}':1\leq l\leq s)$ \emph{that satisfies} $k_{b_l}'-k_{a_l}=f_l$. When all these $k$-variables are fixed, we have
\[\sup_{k_j,k_j^*}\|\langle\lambda_j\rangle^bh_{k_jk_j^*}^{(j)}(\lambda_j)\|_{L_{\lambda_j}^2}\lesssim 1,\quad \sup_{k_j',k_j'^*}\|\langle\lambda_j'\rangle^bh_{k_j'k_j'^*}^{(j)}(\lambda_j')\|_{L_{\lambda_j'}^2}\lesssim 1,\] due to the first bound in (\ref{input2+}). Using the algebra property of the norm $\|\langle\lambda\rangle^bh(\lambda)\|_{L^2}$ under convolution, we have
\begin{multline}\label{frakb} |\mathfrak{C}|\lesssim\prod_{j\in Z}N_j^{-2}\int\mathrm{d}\lambda\cdot\langle\lambda\rangle^{-2b}\bigg\langle\lambda-|k|^2+\sum_{j\in Z}\iota_j|k_j|^2-\alpha_0\bigg\rangle^{-b}\bigg\langle\lambda-|k|^2+\sum_{j\in Z}\iota_j|k_j'|^2-\alpha_0'\bigg\rangle^{-b}\\\lesssim \prod_{j\in Z}N_j^{-2}\cdot\bigg\langle|k|^2-\sum_{j\in Z}\iota_j|k_j|^2+\alpha_0\bigg\rangle^{-b}\bigg\langle|k|^2-\sum_{j\in Z}\iota_j|k_j'|^2+\alpha_0'\bigg\rangle^{-b}.\end{multline}
With \eqref{frakb} we can now bound $\mathfrak{G}$ by
\begin{multline}\label{boundsg}|\mathfrak{G}|^2\lesssim(L_+)^{40n^2}\prod_{j\in Z}N_j^{-4}\sum_{(k_{[Z\backslash W]},k_{[Z\backslash W']}')}\bigg(\sum_{(k,k_{a_l},k_{b_l}':1\leq l\leq s)}\bigg\langle|k|^2-\sum_{j\in Z\backslash W}\iota_j|k_j|^2-\sum_{l=1}^s\iota_{a_l}|k_{a_l}|^2+\alpha_0\bigg\rangle^{-b}\\\times\bigg\langle|k|^2-\sum_{j\in Z\backslash W'}\iota_j|k_j'|^2-\sum_{l=1}^s\iota_{b_l}|k_{b_l}'|^2+\alpha_0'\bigg\rangle^{-b}\bigg)^2,
\end{multline} and multiplying out the square we get
\begin{equation}\label{2tolastsum}
|\mathfrak{G}|^2\lesssim(L_+)^{40n^2}\prod_{j\in Z}N_j^{-4}\sum_{(k_{[Z\backslash W]},k_{[Z\backslash W']}')}\sum_{\substack{(k,k_{a_l},k_{b_l}':1\leq l\leq s)\\(\accentset{\circ}{k},\accentset{\circ}{k_{a_l}},\accentset{\circ}{k_{b_l}'}:1\leq l\leq s)}}\langle \Upsilon\rangle^{-b}\cdot\langle\accentset{\circ}{\Upsilon}\rangle^{-b}\cdot\langle\Upsilon'\rangle^{-b}\cdot\langle\accentset{\circ}{\Upsilon'}\rangle^{-b},
\end{equation} where
\[\Upsilon=|k|^2-\sum_{j\in Z\backslash W}\iota_j|k_j|^2-\sum_{l=1}^s\iota_{a_l}|k_{a_l}|^2+\alpha_0,\quad\accentset{\circ}{\Upsilon}=|\accentset{\circ}{k}|^2-\sum_{j\in Z\backslash W}\iota_j|k_j|^2-\sum_{l=1}^s\iota_{a_l}|\accentset{\circ}{k_{a_l}}|^2+\alpha_0,\]
\[\Upsilon'=|k|^2-\sum_{j\in Z\backslash W'}\iota_j|k_j'|^2-\sum_{l=1}^s\iota_{b_l}|k_{b_l}'|^2+\alpha_0',\quad \accentset{\circ}{\Upsilon'}=|\accentset{\circ}{k}|^2-\sum_{j\in Z\backslash W'}\iota_j|k_j'|^2-\sum_{l=1}^s\iota_{b_l}|\accentset{\circ}{k_{b_l}'}|^2+\alpha_0',\]with $\iota_{a_l}=\iota_{b_l}$, $\alpha_0$ and $\alpha'_0$ are as in \eqref{defnewd2}, and the variables in the summation (\ref{2tolastsum}) verify the following linear equations:
\begin{equation}\label{newequations}
\begin{aligned}\sum_{j\in Z\backslash W}\iota_jk_j+\sum_{l=1}^s\iota_{a_l}k_{a_l}-k&=\sum_{j\in Z\backslash W}\iota_jk_j+\sum_{l=1}^s\iota_{a_l}\accentset{\circ}{k_{a_l}}-\accentset{\circ}{k}=d_0,\\
\sum_{j\in Z\backslash W'}\iota_jk_j'+\sum_{l=1}^s\iota_{b_l}k_{b_l}'-k&=\sum_{j\in Z\backslash W'}\iota_jk_j'+\sum_{l=1}^s\iota_{b_l}\accentset{\circ}{k_{b_l}'}-\accentset{\circ}{k}=d_0',
\end{aligned}
\end{equation} with $d_0$ and  $d'_0$ as in \eqref{defnewd2}, as well as $k_{b_l}'-k_{a_l}=\accentset{\circ}{k_{b_l}'}-\accentset{\circ}{k_{a_l}}=f_l$.

By Cauchy-Schwartz, we may replace the summand on the right hand side of (\ref{2tolastsum}) by $\langle\Upsilon\rangle^{-2b}\cdot\langle\accentset{\circ}{\Upsilon'}\rangle^{-2b}$ (or by $\langle\accentset{\circ}{\Upsilon}\rangle^{-2b}\cdot\langle\Upsilon'\rangle^{-2b}$, which is treated similarly by symmetry). Now going back to (\ref{simpleexp2}) and applying Cauchy-Schwartz in the variables $(k_{[V\backslash Z]},k_{[V\backslash Z]}',\lambda_{[V\backslash Z]},\lambda_{[V\backslash Z]}')$, we get
\begin{multline*}|\mathcal{X}|^4\lesssim N_n^{-4(1-\gamma)}\sum_{k_n,k_n'}\int\mathrm{d}\lambda_n\mathrm{d}\lambda_n'|\mathcal{O}_{k_nk_n'}(\lambda_n,\lambda_n')|^2\lesssim N_n^{-4(1-\gamma)}\bigg(\sum_{(2)}\int\prod_{j\in V\backslash Z}\langle\lambda_j\rangle^{2b}\langle\lambda_j'\rangle^{2b}\cdot|\mathfrak{F}|^2\bigg)\\\times\bigg(\sum_{k_n,k_n'}\int\mathrm{d}\lambda_n\mathrm{d}\lambda_n'\cdot\langle\lambda_n\rangle^{-2b}\langle\lambda_n'\rangle^{-2b}\sum_{(2)}\int\prod_{j\in V\backslash Z}\langle\lambda_j\rangle^{-2b}\langle\lambda_j'\rangle^{-2b}\cdot|\mathfrak{G}|^2\bigg).
\end{multline*} The first parenthesis (together with some factors from the second parenthesis) give the product of all factors in (\ref{estimate3}) except $\mathcal{E}_1\mathcal{E}_1^{+}$, by using (\ref{input3}) and (\ref{contbd}). The second parenthesis, after applying (\ref{2tolastsum}) with the summand $\langle\Upsilon\rangle^{-b}\cdot\langle\accentset{\circ}{\Upsilon}\rangle^{-b}\cdot\langle\Upsilon'\rangle^{-b}\cdot\langle\accentset{\circ}{\Upsilon'}\rangle^{-b}$ replaced by $\langle\Upsilon\rangle^{-2b}\langle\accentset{\circ}{\Upsilon'}\rangle^{-2b}$, integrating in $\lambda_{[U\backslash Z]}$ and $\lambda_{[U\backslash Z]}'$, and plugging in (\ref{defnewd2}), reduces to
\begin{equation}\label{reducetosum}\sum_{(k_{[U\backslash W]},k_{[U\backslash W']}')}\sum_{\substack{(k,k_{a_l},k_{b_l}':1\leq l\leq s)\\(\accentset{\circ}{k},\accentset{\circ}{k_{a_l}},\accentset{\circ}{k_{b_l}'}:1\leq l\leq s)}}\langle\Sigma-\alpha\rangle^{-2b}\langle\accentset{\circ}{\Sigma'}-\alpha\rangle^{-2b},\end{equation}
where $\Sigma$ and $\accentset{\circ}{\Sigma'}$ are respectively
\begin{equation}\Sigma=\sum_{j\not\in W}\iota_j|k_j|^2+\sum_{l=1}^s\iota_{a_l}|k_{a_l}|^2-|k|^2,\quad\accentset{\circ}{\Sigma'}=\sum_{j\not\in W'}\iota_j|k_j'|^2+\sum_{l=1}^s\iota_{b_l}|\accentset{\circ}{k_{b_l}'}|^2-|\accentset{\circ}{k}|^2,
\end{equation} and the variables in the summation satisfy
\begin{equation}\label{newequations2}
\begin{aligned}\sum_{j\not\in W}\iota_jk_j+\sum_{l=1}^s\iota_{a_l}k_{a_l}-k&=\sum_{j\not\in W}\iota_jk_j+\sum_{l=1}^s\iota_{a_l}\accentset{\circ}{k_{a_l}}-\accentset{\circ}{k}=d,\\
\sum_{j\not\in W'}\iota_jk_j'+\sum_{l=1}^s\iota_{b_l}k_{b_l}'-k&=\sum_{j\not\in W'}\iota_jk_j'+\sum_{l=1}^s\iota_{b_l}\accentset{\circ}{k_{b_l}'}-\accentset{\circ}{k}=d,
\end{aligned}
\end{equation} as well as $k_{b_l}'-k_{a_l}=\accentset{\circ}{k_{b_l}'}-\accentset{\circ}{k_{a_l}}=f_l$. Now, when $k_{[U\backslash W]}$ and $(k,k_{a_l}:1\leq l\leq s)$ are fixed, the sum of $\langle\accentset{\circ}{\Sigma'}-\alpha\rangle^{-2b}$ over $k_{[U\backslash W']}'$ and $(\accentset{\circ}{k},\accentset{\circ}{k_{b_l}'}:1\leq l\leq s)$ can be bounded by $\mathcal{E}_1^{+}$ with $A=U\backslash W'$ in (\ref{extraeqn}) due to the equations (\ref{newequations2}); on the other hand the sum of $\langle\Sigma-\alpha\rangle^{-2b}$ over $k_{[U\backslash W]}$ and $(k,k_{a_l}:1\leq l\leq s)$ can be bounded by $\mathcal{E}_1$. This bounds the sum (\ref{reducetosum}) by $\mathcal{E}_1\mathcal{E}_1^+$ and proves (\ref{estimate3}).
\subsubsection{A special $\mathcal{T}^*\mathcal{T}$ argument}\label{case3} Assume now $N_n=N^{(1)}$ and $N_*\gtrsim N^{(2)}$. Again we only need to study the operator $\mathcal{T}$ given by the kernel (\ref{kernel0}); note that $\mathcal{T}_{kk_n}(\lambda,\lambda_n)$ is supported in the set $\{(k,k_n):|k-\iota_nk_n+d|\lesssim N^{(2)}\}$, by the standard orthogonality argument it suffices to prove the same operator bound for $\widetilde{\mathcal{T}}$ which is $\mathcal{T}$ restricted to the set $\{k:|k-f|\leq N^{(2)}\}$, \emph{uniformly} in $f\in\mathbb{Z}^2$. Below we will fix an $f$ and denote $\widetilde{\mathcal{T}}$ still by $\mathcal{T}$, so that in any summations below we may assume $|k-f|\lesssim N^{(2)}$ and $|k_n-\iota_n(f+d)|\lesssim N^{(2)}$ (same for $k_n'$). At this point the parameter $N^{(1)}$ or $N_n$ no longer explicitly appears in the estimate, so the set of $\omega$ we remove will be independent of it. Also we can prove (\ref{estimate3}), with $N_n$ replaced by $N^{(2)}$ in both $\mathcal{E}_1$ and $\mathcal{E}_1^{\mathrm{ex}}$, by essentially repeating the proof in Section \ref{case2} above (and making the bound uniform in $f$ in exactly the same way as below); it remains to prove (\ref{estimate5}).

We start with (\ref{simpleexp2}) and now look for further pairings in $(k_{[Z]}^*,k_{[Z]}'^*)$ in the expression $\mathfrak{G}$ given by (\ref{newdefg}). By repeating the same reduction step in Section \ref{case1}, we can find two partitions $(X_1,\cdots,X_q,W)$ and $(X_1',\cdots,X_q',W')$\footnote{This $W$ and $W'$ are different from the $W$ and $W'$ of Section \ref{case2}.} of the set $Z$, where $2p+q\leq n_1$, such that $N_j$ are all comparable for $j$ in each $X_l\cup X_l'$, and further reduce (\ref{simpleexp2}) to a sum
\begin{equation}\label{simpleexp3}\mathcal{O}_{k_nk_n'}(\lambda_n,\lambda_n')=\langle\lambda_n\rangle^{-b}\langle\lambda_n'\rangle^{-b}\sum_{(3)}\int\mathfrak{F}^+\cdot\mathfrak{G}^+,
\end{equation} where the symbol $\sum_{(3)}\int$ represents the sum in $k_{[V\backslash W]}$ and $k_{[V\backslash W']}'$ and integration in $\lambda_{[V\backslash W]}$ and $\lambda_{[V\backslash W']}'$, the factor $\mathfrak{F}^+$ is independent of $(k_n,k_n',\lambda_n,\lambda_n')$,
\begin{equation}\mathfrak{F}^+=\prod_{i=1}^p\overline{P_{k_{[Y_i]}}^{(i)}(\lambda_{[Y_i]})}P_{(k_{[Y_i]}')}^{(i)}(\lambda_{[Y_i]}')\prod_{j=n_1+1}^{n-1}\overline{[v_{k_j}^{(j)}(\lambda_j)]^{\iota_j}}[v_{k_j'}^{(j)}(\lambda_j')]^{\iota_j}\prod_{l=1}^qQ_{k_{[X_l]},k_{[X_l']}'}^{(l)}(\lambda_{[X_l]},\lambda_{[X_l']}'),
\end{equation}\begin{equation}\label{defnewq}Q_{k_{[X_l]},k_{[X_l']}'}^{(l)}(\lambda_{[X_l]},\lambda_{[X_l']}'):=\sum_{k^*}\prod_{j\in X_l}h_{k_jk^*}^{(j)}(\lambda_j)^{\pm}\prod_{j\in X_l'}h_{k_j'k^*}^{(j)}(\lambda_j')^{\pm}\langle k^*\rangle^{-q_3}(g_{k^*})^{q_1}(\overline{g_{k^*}})^{q_2},
\end{equation}where $q_1+q_2=q_3=|X_l|+|X_l'|$, and the $P$ factors are defined in (\ref{contraction}). We may also fix $a_l\in X_l$ and $b_l\in X_l'$ such that $\iota_{a_l}=\iota_{b_l}$; without loss of generality assume $b_l=2p+l$ for $1\leq l\leq q$. We can bound (\ref{defnewq}) just like we bound (\ref{contraction}) in (\ref{contbd}), except that now it is possible to have $ X_l\cup X_l'\subset \mathcal{G}$. 
Let $R_{2p+l}=\max\{L_j:j\in X_l\cup X_l'\}\gtrsim L_{2p+l}$, then the same argument as in (\ref{contbd}) gives
\begin{equation}\label{boundnewq}
\|Q_{k_{[X_l]},k_{[X_l']}'}^{(l)}(\lambda_{[X_l]},\lambda_{[X_l']}')\prod_{j\in X_l}\langle\lambda_j\rangle^b\prod_{j\in X_l'}\langle\lambda_j'\rangle^b\|_{\ell_{k_{[X_l]},k_{[X_l']}'}^2L_{\lambda_{[X_l]},\lambda_{[X_l']}'}^2}^2\lesssim\prod_{j\in X_l}N_j^{-2}\prod_{j\in X_l'}N_j^{-2}\cdot N_{2p+l}^{2+2\gamma_0}R_{2p+l}^{-2}.
\end{equation} Finally, the multilinear Gaussian $\mathfrak{G}^+$ is given by
\begin{equation}\mathfrak{G}^+=\sum_{(k_{[W]},k_{[W']}')}\sum_{(k_{[W]}^*,k_{[W']}'^*)}\int\mathrm{d}\lambda_{[W]}\mathrm{d}\lambda_{[W']}'\cdot\prod_{j\in W}\overline{\frac{g_{k_j^*}^{\iota_j}}{\langle k_j^*\rangle}\langle\lambda_j\rangle^{b}h_{k_jk_j^*}^{(j)}(\lambda_j)^\pm}\prod_{j\in W'}\frac{g_{k_j'^*}^{\iota_j}}{\langle k_j'^*\rangle}\langle\lambda_j'\rangle^{b}h_{k_j'k_j'^*}^{(j)}(\lambda_j')^\pm\cdot\mathfrak{A^+},
\end{equation} where there is no pairing among $(k_{[W]}^*,k_{[W']}'^*)$, and coefficient $\mathfrak{A}^+$ of form \begin{multline}\mathfrak{A}^+=\sum_{\substack{k:\iota_1k_1+\cdots+\iota_nk_n=k+d\\\iota_1k_1'+\cdots+\iota_nk_n'=k+d}}\int\mathrm{d}\lambda\cdot\eta\bigg(\lambda,\lambda-|k|^2-\sum_{j=1}^n\iota_j(\lambda_j-|k_j|^2)-\alpha\bigg)\\\times\langle\lambda\rangle^{-2b}\eta\bigg(\lambda,\lambda-|k|^2-\sum_{j=1}^n\iota_j(\lambda_j'-|k_j'|^2)-\alpha\bigg)\prod_{j\in W}\langle\lambda_j\rangle^{-b}\prod_{j\in W'}\langle\lambda_j'\rangle^{-b},
\end{multline} where the sum is over a single variable $k$. As before we may also assume $|k_j-k^*|\lesssim R_{2p+l}(N_*)^{C\kappa^{-1}}$ for $j\in X_l$ in (\ref{defnewq}) and similarly for $k_j'$ and $j\in X_l'$, so in particular $|k_{a_l}-k_{2p+l}'|\lesssim R_{2p+l}(N_*)^{C\kappa^{-1}}$.

The goal now is to estimate $\mathfrak{G}^+$. As before, we need to reduce to $(N_*)^{C\delta^{-7}}$ choices for $(k_{[U\backslash W]},\lambda_{[U\backslash W]})$ and $(k_{[U\backslash W]}',\lambda_{[U\backslash W]}')$, and $f$. By the meshing argument (Remark \ref{mesh}), we may assume $|\lambda|\lesssim (N_*)^{\delta^{-7}}$ and $|\lambda_j|\lesssim (N_*)^{\delta^{-7}}$ for $j\not\in W$ (similarly for $\lambda_j'$) and get rid of these parameters; in the same way we may also fix $k_{[V\backslash W]}$ and $k_{[V\backslash W']}'$, as well as $f+d-\iota_nk_n$ and $f+d-\iota_nk_n'$. Letting $k=f+g$, we can rewrite
\begin{multline}\mathfrak{A}^+=\sum_{|g|\leq N^{(2)}}\mathbf{1}_{\substack{\sum_{j\in W}\iota_jk_j=g+d_0\\\sum_{j\in W'}\iota_jk_j'=g+d_0'}}\int\mathrm{d}\lambda\cdot\eta\bigg(\lambda,\lambda-2f\cdot g-|g|^2-\sum_{j\in V}\iota_j(\lambda_j-|k_j|^2)+\beta(f)+\gamma\bigg)\\\times\langle\lambda\rangle^{-2b}\eta\bigg(\lambda,\lambda-2f\cdot g-|g|^2-\sum_{j\in V}\iota_j(\lambda_j'-|k_j'|^2)+\beta'(f)+\gamma'\bigg)\prod_{j\in W}\langle\lambda_j\rangle^{-b}\prod_{j\in W'}\langle\lambda_j'\rangle^{-b},
\end{multline} where $\gamma,\gamma'\in[0,1)$ are fixed, $d_0,d_0'\in \mathbb{Z}^2$, and $\beta(f),\beta'(f)$ are fixed integer-valued functions of $f$. We may assume $|d_0|,|d_0'|\lesssim N^{(2)}$ since otherwise $\mathfrak{A}^+\equiv 0$, then we may fix them and see that $f$ enters the whole expression only through the function $-2f\cdot g+\beta(f)$; moreover we may restrict $g$ to the set where $|-2f\cdot g+\beta(f)|\leq(N_*)^{\delta^{-7}}$ since otherwise either $\lambda$ or one $\lambda_j$ must be large and we close as before. The reduction to finitely many cases can then be done by invoking the following claim, which will be proved at the end of this section:
\begin{claim}\label{claim0} Let the function $F_{f,\beta}(g):=-2f\cdot g+\beta$, with the particular domain $\mathrm{Dom}(F_{f,\beta})=\{|g|\leq N^{(2)}:|-2f\cdot g+\beta|\leq (N_*)^{\delta^{-7}}\}$. Then when $f\in\mathbb{Z}^2$ and $\beta\in\mathbb{Z}$ varies, the function $F_{f,\beta}$ (together with its domain) has finitely many, and in fact $\lesssim (N_*)^{C\delta^{-7}}$ possibilities.
\end{claim}
From now on we may fix the value of $f$. By removing a set of probability $\leq C_\theta e^{-(\tau^{-1}N_*)^\theta}$, we can apply Lemma \ref{lemma:4.2} and conclude that (recall that $a_0=2b-10\delta^6$)
\begin{multline}|\mathfrak{G}^+|^2\lesssim\prod_{j\in W}N_j^{-2}\prod_{j\in W'}N_j^{-2}\sum_{(k_{[W]},k_{[W']}')}\int\mathrm{d}\lambda_{[W]}\mathrm{d}\lambda_{[W']}'\prod_{j\in W}\langle\lambda_j\rangle^{-a_0}\prod_{j\in W'}\langle\lambda_j'\rangle^{-a_0}\bigg[\sum_{\substack{k:\iota_1k_1+\cdots+\iota_nk_n=k+d\\\iota_1k_1'+\cdots+\iota_nk_n'=k+d}}\int \mathrm{d}\lambda\\\times\langle\lambda\rangle^{-2b}\eta\bigg(\lambda,\lambda-|k|^2-\sum_{j=1}^n\iota_j(\lambda_j-|k_j|^2)-\alpha\bigg)\eta\bigg(\lambda,\lambda-|k|^2-\sum_{j=1}^n\iota_j(\lambda_j'-|k_j'|^2)-\alpha\bigg)\bigg]^2.
\end{multline} The integral over $\lambda$ gives a factor
\[\bigg\langle\Sigma-\sum_{j=1}^n\iota_j\lambda_j-\alpha\bigg\rangle^{-b}\bigg\langle\Sigma'-\sum_{j=1}^n\iota_j\lambda_j'-\alpha\bigg\rangle^{-b},\] where $\Sigma$ and $\Sigma'$ are defined as
\[\Sigma=\sum_{j=1}^n\iota_j|k_j|^2-|k|^2,\quad \Sigma'=\sum_{j=1}^n\iota_j|k_j'|^2-|k|^2.\]Since there is only one value of $k$ in the summation, we can reduce
\begin{align}
|\mathfrak{G}^+|^2&\lesssim\sum_k\sum_{(k_{[W]},k_{[W']}')}\int\mathrm{d}\lambda_{[W]}\mathrm{d}\lambda_{[W']}'\prod_{j\in W}\langle\lambda_j\rangle^{-a_0}\prod_{j\in W'}\langle\lambda_j'\rangle^{-a_0}\nonumber\\&\times\prod_{j\in W}N_j^{-2}\prod_{j\in W'}N_j^{-2}\bigg\langle\Sigma-\sum_{j=1}^n\iota_j\lambda_j-\alpha\bigg\rangle^{-2b}\bigg\langle\Sigma'-\sum_{j=1}^n\iota_j\lambda_j'-\alpha\bigg\rangle^{-2b}\nonumber\\&\lesssim\prod_{j\in W}N_j^{-2}\prod_{j\in W'}N_j^{-2}\sum_k\sum_{(k_{[W]},k_{[W']}')}\bigg\langle\Sigma-\sum_{j\not\in W}\iota_j\lambda_j-\alpha\bigg\rangle^{-a_0}\bigg\langle\Sigma'-\sum_{j\not\in W'}\iota_j\lambda_j'-\alpha\bigg\rangle^{-a_0},\label{boundfinalg}
\end{align}
where in the summation over $k$ and $(k_{[W]},k_{[W']}')$ we assume that $\iota_1k_1+\cdots+\iota_nk_n=\iota_1k_1'+\cdots+\iota_nk_n'=k+d$. Returning to (\ref{simpleexp3}), by applying Cauchy-Schwartz in the variables $(k_{[V\backslash W]},\lambda_{[V\backslash W]})$ and $(k_{[V\backslash W']}',\lambda_{[V\backslash W']}')$ we conclude as before that
\begin{multline*}|\mathcal{X}|^4\lesssim (N^{(1)})^{-4(1-\gamma)}\bigg(\sum_{(3)}\int\prod_{j\in V\backslash W}\langle\lambda_j\rangle^{2b}\prod_{j\in V\backslash W'}\langle\lambda_j'\rangle^{2b}\cdot|\mathfrak{F}|^2\bigg)\\\times\bigg(\sum_{k_n,k_n'}\int\mathrm{d}\lambda_n\mathrm{d}\lambda_n'\cdot\langle\lambda_n\rangle^{-2b}\langle\lambda_n'\rangle^{-2b}\sum_{(3)}\int\prod_{j\in V\backslash W}\langle\lambda_j\rangle^{-2b}\prod_{j\in V\backslash W'}\langle\lambda_j'\rangle^{-2b}\cdot|\mathfrak{G}|^2\bigg).
\end{multline*} The first parenthesis (together with some factors from the second parenthesis) gives the product of all factors in (\ref{estimate5}) except $\widetilde{\mathcal{E}_1}\widetilde{\mathcal{E}_1^{+}}$, by using (\ref{input3}), (\ref{contbd}) and (\ref{boundnewq}). The second parenthesis, after applying (\ref{boundfinalg}) and integrating in $\lambda_{[U\backslash W]}$ and $\lambda_{[U\backslash W']}'$, reduces to
\begin{equation}\label{thirdsum}\sum_k \, \sum_{k_{[U]}:\iota_1k_1+\cdots+\iota_nk_n=k+d} \, \, \sum_{k_{[U]}':\iota_1k_1'+\cdots+\iota_nk_n'=k+d}\langle\Sigma-\alpha\rangle^{-a_0}\langle\Sigma'-\alpha\rangle^{-a_0}.\end{equation} Now when $(k,k_{[U]})$ are fixed, the sum of $\langle\Sigma'-\alpha\rangle^{-a_0}$ over $k_{[U]}'$ can be bounded by $\widetilde{\mathcal{E}_1^{+}}$ with $A=\{1,\cdots,n\}$ in (\ref{extraeqn}), due to the linear equation $\iota_1k_1'+\cdots+\iota_nk_n'=k+d$, the fact that $k_{2p+l}'$ belongs to a disc of radius $O(R_{2p+l}(N_*)^{C\kappa^{-1}})$ once $k_{a_l}$ is fixed, and the fact that $|\iota_nk_n'-f-d|\lesssim N^{(2)}$. Moreover the sum of $\langle\Sigma-\alpha\rangle^{-a_0}$  over $(k,k_{[U]})$ can be bounded by $\widetilde{\mathcal{E}_1}$, due to the fact that $|\iota_nk_n-f-d|\lesssim N^{(2)}$. This then bounds (\ref{thirdsum}) by $\widetilde{\mathcal{E}_1}\widetilde{\mathcal{E}_1^{+}}$ and proves (\ref{estimate5}).
\begin{proof}[Proof of Claim \ref{claim0}] Let $D=\mathrm{Dom}(F_{f,\beta})$. If $D$ contains three points $g_1,g_2,g_3$ that are not collinear, then we have $|f\cdot(g_1-g_2)|\lesssim (N_*)^{\delta^{-7}}$, $|f\cdot(g_1-g_3)|\lesssim (N_*)^{\delta^{-7}}$, and that $g_1-g_2$ and $g_1-g_3$ are linearly independent. This implies that $|f|\lesssim (N_*)^{2\delta^{-7}}$ and hence $|\beta|\lesssim (N_*)^{3\delta^{-7}}$ so the result is trivial. Now let us assume $D$ is contained in a line $\ell$; we may assume $\ell$ contains at least two points in the set $\{g:|g|\leq N^{(2)}\}$, otherwise $D$ is at most a singleton and the result is also trivial. Then the integer points in $\ell$ can be written as $p+q\sigma$, where $(p,q)\in(\mathbb{Z}^2)^2$ has at most $(N^{(2)})^{10}$ choices (so we may fix them), and hence $D=\{p+q\sigma:|p+q\sigma|\leq N^{(2)},|a\sigma+b|\leq (N_*)^{\delta^{-7}}\}$ where $a$ and $b$ are integers. Again as $|D|\geq 2$ we know that $|a|\lesssim (N_*)^{2\delta^{-7}}$ and $|b|\lesssim (N_*)^{3\delta^{-7}}$, so $F_{f,\beta}$ indeed has $\lesssim (N_*)^{C\delta^{-7}}$ possibilities, as claimed.
\end{proof}
\begin{rem}\label{extrarem}For later use we will also consider the following variant of $\mathcal{X}$ (same for $\mathcal{Y}$):
\begin{multline}\label{mainexp1+}
\mathcal{X}^+:=\sum_{\substack{(k,k_1,\cdots,k_n)\\\iota_1 k_1+\cdots+\iota_n k_n=k+d}}\int\mathrm{d}\lambda\mathrm{d}\lambda_1\cdots\mathrm{d}\lambda_n\mathrm{d}\mu_1\cdots\mathrm{d}\mu_s\\\times\eta\bigg(\lambda,\lambda-|k|^2-\sum_{j=1}^n\iota_j(\lambda_j-|k_j|^2)-\sum_{j=1}^s\mu_j-\alpha\bigg)\overline{v_{k}(\lambda)}\prod_{j=1}^n[v_{k_j}^{(j)}(\lambda_j)]^{\iota_j}\prod_{j=1}^sw_j(\mu_j),
\end{multline} where each $w_j$ satisfies
\[\|\langle \mu_j\rangle^bw_j(\mu_j)\|_{L_{\mu_j}^2}\lesssim 1.\] Then $\mathcal{X}^{+}$ will satisfy exactly the same estimates as $\mathcal{X}$ (same for $\mathcal{Y}$). In fact we can introduce a ``virtual'' variable $l_j$ which takes a single value and view $w_j(\mu_j)$ as a function of $l_j$ and $\mu_j$ which has type (D), and repeat all the above proof with these new variables.
\end{rem}
\subsection{Proof of Proposition \ref{multi0}}\label{conclude} Armed with Corollary \ref{corcounting} and Proposition \ref{general},  we can now prove Proposition \ref{multi0}. Recall that we will abuse notation and write $(v_{k_j}^{(j)})(\lambda_j)$ instead of $(\widetilde{v^{(j)}})_{k_j}(\lambda_j)$. We will proceed in three steps; note that as before, in the proof below we will ignore any factor $\lesssim \tau^{-\theta}(N_*)^{C\kappa^{-1}}$.

\emph{Step 1: reduction to estimating $\mathcal{X}$ and $\mathcal{Y}$}. First notice that, when the set of pairings among the variables involved in $\mathcal{N}_n$ is fixed, the coefficient in $\mathcal{N}_n$ will be a constant (see Remark \ref{proper}). By Lemma \ref{duhamelest}, we may replace $\mathcal{I}$ by $\mathcal{J}$ in all estimates. Now by definition of the relevant norms, the kernel bound (\ref{trunckernel}) and duality, we can reduce the desired estimates to the estimates of quantities of form $\mathcal{X}$ (for parts (1) and (2)) and $\mathcal{Y}$ (for part (3)) defined in (\ref{mainexp1}) and (\ref{mainexp2}), in fact with $d=\alpha=0$, except that the functions $v$ and $y$ introduced by duality only satisfy weaker bounds
\begin{equation}\label{weaker}\|\langle\lambda\rangle^{1-b_1}v_{k}(\lambda)\|_{\ell_{k}^2L_{\lambda}^2}\lesssim 1,\quad  \|\langle\lambda\rangle^{1-b_1}\langle \lambda'\rangle^by_{kk'}(\lambda,\lambda')\|_{\ell_{k,k'}^2L_{\lambda,\lambda'}^2}\lesssim 1,\end{equation} instead of (\ref{input0}), and that there may be pairings in $\mathcal{X}$ and $\mathcal{Y}$ (but they will always be over-paired).
Now if $|\lambda|\leq (N_*)^{C_0}$ where $C_0$ is a large constant depending only on $n$, then since $b_1-b\sim 2b-1\sim\kappa^{-1}\sim\delta^4$, we can replace the power $\langle\lambda\rangle^{1-b_1}$ by $\langle\lambda\rangle^b$ in (\ref{weaker}) to match (\ref{input0}), at a price of losing a factor $(N_*)^{C\kappa^{-1}}$ which is acceptable. Now we will assume $|\lambda|\geq (N_*)^{C_0}$; below we will only consider part (1) of Proposition \ref{multi0}, since we have $N_*\gtrsim N^{(1)}$ in part (2) and $N_*\gtrsim N_0$ in part (3), and the proof will be similar and much easier.

Here the point is to use the weight $\langle\lambda\rangle^{1-b_1}$ in (\ref{weaker}) to gain a power $\geq (N_*)^{-\frac{C_0}{3}}$, after which we still can assume $\|v_k(\lambda)\|_{\ell_k^2L_\lambda^2}\lesssim 1$. In view of this gain and the assumption $N_*\gtrsim N^{(2)}$, we may fix the values of $k_j$ and/or $k_j^*$ for each $1\leq j\leq n-1$. Moreover when $k_j$ and $k_j^*$ fixed the resulting function in $\lambda_j$ (we will call them $w_j(\lambda_j)$) satisfies $\|\langle \lambda_j\rangle^bw_j(\lambda_j)\|_{L_{\lambda_j}^2}\lesssim 1$, which implies the corresponding $L_{\lambda_j}^1$ bound, so we may fix $\lambda_j(1\leq j\leq n-1)$ also. Finally as 
\[\int\langle\lambda_n\rangle^{2b}\|v_{k_n}^{(n)}(\lambda_n)\|_{\ell_{k_n}^2}^2\,\mathrm{d}\lambda_n=\|\langle\lambda_n\rangle^{b}v_{k_n}^{(n)}(\lambda_n)\|_{\ell_{k_n}^2L_{\lambda_n}^2}^2\lesssim (N^{(1)})^{-2(1-\gamma)},\] we may also fix the value of $\lambda_n$, and reduce to
\[\mathcal{X}=\sum_{k}\int v_{k}(\lambda)\mathrm{d}\lambda\cdot \eta(\lambda,\lambda-F(k))G_{k-d'},\] where $\|G\|_{\ell^2}\sim (N^{(1)})^{-1+\gamma}$ and $F(k)$ is a function of $k$ which, as well as $d'$, depends on the choice of the other fixed variables. By first integrating in $\lambda$ using Cauchy-Schwartz and (\ref{kernelest2}), then summing in $k$ using Cauchy-Schwartz again, we deduce that
\[|\mathcal{X}|\lesssim \|v_k(\lambda)\|_{\ell_k^2L_\lambda^2}\cdot\|G\|_{\ell^2}\lesssim (N^{(1)})^{-1+\gamma},\] which suffices in view of the gain $(N_*)^{-C_0/3}$.

\emph{Step 2: the no-pairing case}. We have now reduced Proposition \ref{multi0} to the estimates for the quantities $\mathcal{X}$ and $\mathcal{Y}$. If we assume there is no pairing, then we can apply Proposition \ref{general}, and then Corollary \ref{corcounting}. Recall the new parameters such as $p$, $q$ and $R_j$ defined in Proposition \ref{general}; denote $L_+=\max(L_{2p+1},\cdots L_{n_1})$ and $N_+=\max(N_{n_1+1},\cdots N_n)$. Also when we talk about an estimate in Proposition \ref{counting1} we are actually talking about its counterpart in Corollary \ref{corcounting}.

In part (1), by combining (\ref{estimate3}), (\ref{bdset1}) and (\ref{bdset1}) with the improvement factor (\ref{exfactor}), with $N_n$ replaced by $N^{(2)}$ in both places, we obtain that
\[|\mathcal{X}|\lesssim (N^{(1)})^{-1+\gamma}(N^{(2)})^{C\gamma}(L_+)^{40n^3}(N^{(2)})^{-\frac{1}{4}};\] on the other hand by combining (\ref{estimate5}), (\ref{bdset1}) and (\ref{bdset1}) with the improvement factor (\ref{exfactor}), with the changes adapted to $\widetilde{\mathcal{E}_1}$ and $\widetilde{\mathcal{E}_1^{\mathrm{ex}}}$ indicated in Proposition \ref{general}, we obtain that
\[|\mathcal{X}|\lesssim (N^{(1)})^{-1+\gamma}(N^{(2)})^{C\gamma}(L_+)^{-\frac{1}{4}},\] noticing that $R_{2p+l}\gtrsim L_{2p+l}$ for $1\leq l\leq q$ and $N_j\gtrsim L_j$ for $2p+q+1\leq j\leq n_1$. Interpolating the above two bounds then gives (\ref{mainmult1}).

In parts (2) and (3), we have $N^{(1)}\sim N_a$ and $a\in\mathcal{G}\cup\mathcal{C}$, in particular the extra factor (\ref{exfactor}) is bounded by $(N_+)^{-1}$; note that in case (3) we may have $a\in\mathcal{D}$ but in this case the extra factor will be replaced by $(N^{(1)})^{-1}$. By combining (\ref{estimate1}) and (\ref{bdset1}) we obtain (noticing that $N_{PR}\lesssim N^{(2)}$)
\[|\mathcal{X}|\lesssim (N^{(1)}N^{(2)})^{-\frac{1}{2}}(N^{(2)})^{\gamma_0}(N_+)^{C\gamma}(L_+)^{-\delta_0},\] and by combining (\ref{estimate3}) and (\ref{bdset1}) together with the improvement factor (\ref{exfactor}) we obtain that
\[|\mathcal{X}|\lesssim (N^{(1)}N^{(2)})^{-\frac{1}{2}}(N^{(2)})^{\gamma_0}(N_+)^{C\gamma}(L_+)^{40n^3}(N_+)^{-\frac{1}{4}},\] and interpolating the above two bounds gives (\ref{mainmult2}); in the same way (\ref{mainmult3}) follows from (\ref{estimate1}), (\ref{estimate3}), (\ref{bdset2}) and (\ref{bdset2}) with the improvement factor (\ref{exfactor}), and (\ref{mainmult5}) follows from (\ref{estimate2}), (\ref{estimate4}), (\ref{bdset3}) and (\ref{bdset3}) with the suitable improvement factor.

Finally consider (\ref{mainmult4}); here we will define $N'=\max^{(2)}(N_{n_1+1},\cdots,N_n)$. Note that $\alpha=0$, so by combining (\ref{estimate1}) and (\ref{bdset5}) we get \begin{equation}\label{finalbd1}|\mathcal{X}|\lesssim (N^{(1)})^{-1+\gamma_0}(N_+)^{\gamma}(N')^{C\gamma}(L_+)^{-\delta_0};\end{equation} on the other hand, by combining (\ref{estimate3}) and either (\ref{bdset5}) with the improvement factor (\ref{exfactor}) or (\ref{bdset6}), we get that either
\begin{equation}\label{finalbd2}|\mathcal{X}|\lesssim (N^{(1)})^{-1+\gamma_0}(N_+)^{\gamma}(N')^{C\gamma}(L_+)^{40n^3}(N_+)^{-\frac{1}{4}},\end{equation} or
\begin{equation}\label{finalbd3}|\mathcal{X}|\lesssim (N^{(1)})^{-1+\gamma_0}(N_+)^{\gamma}(N')^{C\gamma}(L_+)^{40n^3}\min\big((N')^{-\frac{1}{4}},(N^{(1)})^{\frac{1}{4}}(N_+)^{-\frac{1}{2}}\big).\end{equation} Clearly interpolating (\ref{finalbd1}) and (\ref{finalbd2}) gives (\ref{mainmult4}); suppose instead we have (\ref{finalbd1}) and (\ref{finalbd3}). Now if $N_+\geq (N^{(1)})^{\frac{2}{3}}$ and $L_+\geq (N_+)^{\frac{1}{(40n)^4}}$ then (\ref{finalbd1}) implies (\ref{mainmult4}); if $N_+\geq (N^{(1)})^{\frac{2}{3}}$ and $L_+\leq (N_+)^{\frac{1}{(40n)^4}}$ then (\ref{finalbd3}) implies $|\mathcal{X}|\lesssim (N^{(1)})^{-1.01}$ which implies (\ref{mainmult4}); if $N_+\leq (N^{(1)})^{\frac{2}{3}}$ then interpolating (\ref{finalbd1}) and (\ref{finalbd3}) implies
\[|\mathcal{X}|\lesssim (N^{(1)})^{-1+\gamma_0}(N_+)^{\gamma}\lesssim (N^{(1)})^{-1+\gamma_0+\frac{2\gamma}{3}}\] which implies (\ref{mainmult4}). Note also that for general $\alpha$, due to the factor $\frac{\max((N^{(2)})^2,|\alpha|)}{(N^{(2)})^2}$ on the right hand side of (\ref{bdset5}), the above argument gives the bound
\begin{equation}\label{finalbd4}|\mathcal{X}|\lesssim (N^{(1)})^{-1+\frac{4\gamma}{5}}\max\bigg(1,\frac{|\alpha|^{\frac{1}{2}}}{N^{(2)}}\bigg).
\end{equation}

\emph{Step 3: the over-pairings}. We will only consider $\mathcal{X}$, $\mathcal{Y}$ is similar and easier since there cannot be any pairing between $\{k,k'\}$ and any $k_j$ due to $N^{(1)}\lesssim N_0^{1-\delta}$ and the restrictions $|k|,|k'|\geq \frac{N_0}{4}$ in (\ref{mainmult5}). Now due to simplicity, any pairing in $\mathcal{X}$ must be an over-pairing; by collecting all these pairings we can find a partition $(A_1,\cdots,A_p,B)$ of $\{1,\cdots,n\}$ such that $|A_i|\geq 3$ and $k_j$ takes a single value for $j$ in each $A_i$, that this value is different for different $1\leq i\leq p$, and there is no over-pairing among $\{k_j:j\in B\}$. Then we can check that either there is no pairing among $\{k,k_j:j\in B\}$, or there is a unique over-pairing $k=k_{j_1}=k_{j_2}$ with $j_1,j_2\in B$ and $(\iota_{j_1},\iota_{j_2})\neq (-,-)$. In the latter case denote $\{j_1,j_2\}=A_0$ and replace $B$ by $B\backslash A_0$, so that there is no pairing among $\{k,k_j:j\in B\}$. Below we will focus on the first case, and leave to the end the necessary changes caused by $A_0$.

Now $\mathcal{X}$ is reduced to
\begin{equation}\label{newx}\mathcal{X}=\sum_{l_1,\cdots,l_p}\int\prod_{i=1}^p\prod_{j\in A_i}(v_{l_i}^{(j)}(\lambda_j))^{\pm}\mathrm{d}\lambda_j\cdot\mathcal{X}',\end{equation}where $l_i$ is the common value of $k_j$ for $j\in A_i$ (so that $|l_i|\lesssim N^{(2)}$), and $\mathcal{X}'$ is an expression of the same form as $\mathcal{X}$, but only involves the variables $(k,k_j)$ and $(\lambda,\lambda_j)$ for $j\in B$, with $d$ being a fixed linear combination of $l_i$, and $\alpha$ being a fixed linear combination of $|l_i|^2$. This gives 
\[|\mathcal{X}|\lesssim\sum_{l_1,\cdots,l_p}\prod_{i=1}^pM_{l_i}^{(i)}\cdot\sup_{l_1,\cdots,l_p}\bigg|\int\prod_{i=1}^p\frac{1}{M_{l_i}^{(i)}}\prod_{j\in A_i}(v_{l_i}^{(j)}(\lambda_j))^{\pm}\mathrm{d}\lambda_j\cdot\mathcal{X}'\bigg|,\quad M_{l_i}^{(i)}:=\prod_{j\in A_i}\|\langle\lambda_j\rangle^{b_2}v_{l_i}^{(j)}(\lambda_j)\|_{L_{\lambda_j}^2},\] where recall that $b_2=b-\delta^6$. When each $l_i$ is fixed, by Remark \ref{extrarem}, the expression
\[\int\prod_{i=1}^p\frac{1}{M_{l_i}^{(i)}}\prod_{j\in A_i}(v_{l_i}^{(j)}(\lambda_j))^{\pm}\mathrm{d}\lambda_j\cdot\mathcal{X}'\] can be estimated in the same way as $\mathcal{X}'$ (replacing $b$ by $b_2$ will not change the proof), which is done in \emph{Step 2} above. We then only need to bound
\[\sum_{l_1,\cdots,l_p}\prod_{i=1}^pM_{l_i}^{(i)}=\prod_{i=1}^p\sum_{l_i}M_{l_i}^{(i)},\] which we establish in the following claim.
\begin{claim}\label{pairclaim} Let $K_i=\max(N_j:j\in A_i)$ and $K_i'=\max^{(2)}(N_j:j\in A_i)$, then $\tau^{-1}N_*$-certainly we have that
\begin{equation}\label{pairbd}\sum_{l_i}M_{l_i}^{(i)}\lesssim \left\{
\begin{split}& K_i^{-1+\gamma}(K_i')^{-\frac{1}{3}}, & K_i&\sim N_j,j\in\mathcal{D};\\
& K_i^{-1+\theta},&K_i&\sim N_j,j\in\mathcal{G}\cup\mathcal{C}.
\end{split}
\right.
\end{equation}
\end{claim}
\begin{proof}[Proof of Claim \ref{pairclaim}] Let $R_{l_i}^{(j)}=\|\langle\lambda_j\rangle^{b_2}v_{l_i}^{(j)}(\lambda_j)\|_{L_{\lambda_j}^2}$, then we have $\|R^{(j)}\|_{\ell_{l_i}^\infty}\lesssim \|R^{(j)}\|_{\ell_{l_i}^2}\lesssim N_j^{-1+\gamma}$ if $j\in\mathcal{D}$, $\|R^{(j)}\|_{\ell_{l_i}^\infty}\lesssim N_j^{-1+\theta}$ and $\|R^{(j)}\|_{\ell_{l_i}^2}\lesssim N_j^{\theta}$ if $j\in\mathcal{G}$. If $j\in\mathcal{C}$ we will apply Lemma \ref{largedev}, and again reduce to finitely many $\lambda_j$ by restricting the size of $\lambda_j$ and dividing into small intervals, and using the differentiability in $\lambda_j$ of $h_{k_jk_j^*}^{(j)}(\lambda_j)$, which is assumed in the statement of Proposition \ref{multi0}. In the same way as before, by removing a set of probability $\leq C_\theta e^{-(\tau^{-1}N_*)^\theta}$ and omitting any $\tau^{-\theta}(N_*)^{C\kappa^{-1}}$ factors, we conclude that
\[\|R^{(j)}\|_{\ell_{l_i}^\infty}\lesssim \|R^{(j)}\|_{\ell_{l_i}^2}\lesssim N_j^{-1}\|\langle\lambda_j\rangle^bh_{k_jk_j^*}^{(j)}(\lambda_j)\|_{\ell_{k_j,k_j^*}^2L_{\lambda_j}^2}\lesssim N_j^{-\frac{1}{2}+\gamma_0}L_{j}^{-\frac{1}{2}},\] as well as
\[\|R^{(j)}\|_{\ell_{l_i}^\infty}\lesssim N_j^{-1}L_j^2\sup_{k_j,k_j^*}\|\langle\lambda_j\rangle^bh_{k_jk_j^*}^{(j)}(\lambda_j)\|_{L_{\lambda_j}^2}\lesssim N_j^{-1}L_j^2,\] which also implies $\|R^{(j)}\|_{\ell_{l_i}^\infty}\lesssim N_j^{-0.55}$. Now let $K_i\sim N_j$ and $K_i'\sim N_s$, then if $j\in\mathcal{D}$, (\ref{pairbd}) follows from applying H\"{o}lder and measuring $R^{(j)}$ and another factor other than $R^{(j)}$ or $R^{(s)}$ in $\ell_{l_i}^2$, and all other factors in $\ell_{l_i}^\infty$. If $j\in\mathcal{G}\cup\mathcal{C}$, then we may assume $|l_i|\sim K_i$ (otherwise (\ref{pairbd}) follows trivially from the third inequality in (\ref{input2+})), so the $N_j$ for $j\in A_i$ must all be comparable. We may then assume $j\in\mathcal{G}\cup\mathcal{C}$ for each $j\in A_i$, and (\ref{pairbd}) follows from applying H\"{o}lder and measuring two factors in $\ell_{l_i}^2$ and the rest in $\ell_{l_i}^\infty$, such that at least one $R^{(j)}$ with $j\in\mathcal{G}$ is measured in $\ell_{l_i}^\infty$ if there is any. This completes the proof.
\end{proof}

The general case of Proposition \ref{multi0} then follows from the $\mathcal{X}'$ estimate, namely the no-pairing case in \emph{Step 2}, combined with Claim \ref{pairclaim}. More precisely, suppose $N^{(1)}=N_a$ with $a\in A_i$ for some $i$, then if $a\in\mathcal{D}$ the bound (\ref{pairbd}) gives the power $(N^{(1)})^{-1+\gamma}$, while the power $(K_i')^{-\frac{1}{3}}$ in (\ref{pairbd}), as well as the no-pairing case of the bounds (\ref{mainmult1}) and (\ref{mainmult2}), give the power $(N^{(2)})^{-\frac{1}{4}}$. If $a\in\mathcal{G}\cup\mathcal{C}$, then the power $(N^{(1)})^{-1+\theta}$ from (\ref{pairbd}) is already enough. If $N^{(1)}=N_a$ with $a\in B$ then we simply apply the no-pairing case and use (\ref{pairbd}) to gain decay in $N^{(2)}$ when $N^{(2)}=N_j$ and $j\not\in B$. The only nontrivial case is (\ref{mainmult4}), where there is no need to gain decay in $N^{(2)}$, but we have an extra factor $\lesssim\max(1,|\alpha|^{\frac{1}{2}})$ from (\ref{finalbd4}), where $\alpha$ is a linear combination of $|l_i|^2$. By Claim \ref{pairclaim} we have
\[\langle\alpha\rangle^{\frac{1}{2}}\lesssim\max_{1\leq i\leq p}\min_{j\in A_i}N_j,\quad \prod_{i=1}^p\sum_{l_i}M_{l_i}^{(i)}\lesssim (N_*)^\theta\langle\alpha\rangle^{-\frac{1}{2}},\] which cancels this extra factor and proves (\ref{mainmult4}).

Finally we consider the case with $A_0$, say $k=k_{j_1}=k_{j_2}$ and $N_{j_1}\geq N_{j_2}$. Here we can check that $\mathcal{X}$ still has the form (\ref{newx}), except that in $\mathcal{X}'$ the input $v_k(\lambda)$ is replaced by (essentially){\footnote{To deal with the $\eta$ factor in the expression of $\mathcal{X}$, see (\ref{mainexp1}), we only need to assume $\eta(\lambda,\mu)=\eta_1(\lambda)\eta_2(\mu)$ has factorized form. In general it is easy to write $\eta$ as a linear combination of functions of such form with summable coefficients, by invoking the explicit formula for $\eta$, which can be deduced from the calculations in the proof of Lemma \ref{duhamelest}.}}
\[\widetilde{v}_k(\lambda)=\int_{\pm\lambda_0\pm\lambda_1\pm\lambda_2=\lambda} v_{k}(\lambda_0)^{\pm}\cdot v_{k}^{(j_1)}(\lambda_1)^{\pm}\cdot v_{k}^{(j_2)}(\lambda_2)^{\pm}\,\mathrm{d}\lambda_1\mathrm{d}\lambda_2,\] which satisfies, due to the same proof as in Claim \ref{pairclaim},
\begin{equation}\label{pairbd2}
\begin{aligned}\|\langle\lambda\rangle^{b_2}\widetilde{v}_k(\lambda)\|_{\ell_k^2L_{\lambda}^2}&\lesssim \|\langle\lambda_0\rangle^{b_2}v_k(\lambda_0)\|_{\ell_k^2L_{\lambda_0}^2}\|\langle\lambda_1\rangle^{b_2}v_k^{(j_1)}(\lambda_1)\|_{\ell_k^\infty L_{\lambda_1}^2}\|\langle\lambda_2\rangle^{b_2}v_{k}^{(j_2)}(\lambda_2)\|_{\ell_k^\infty L_{\lambda_2}^2}\\
&\lesssim\left\{
\begin{split}& N_{j_1}^{-1+\gamma}N_{j_2}^{-\frac{1}{3}},&j_1&\in\mathcal{D},\\
& N_{j_1}^{-1+\theta},&j_1&\in\mathcal{G}\cup\mathcal{C}.
\end{split}
\right.
\end{aligned}\end{equation} The rest of proof now goes exactly as above using the additional bound (\ref{pairbd2}), which has exactly the same gain as in Claim \ref{pairclaim}, and the set $A_0$ is treated together with the other sets $A_i$. This completes the proof of Proposition \ref{multi0}.
\subsection{Stability and convergence}\label{stability} 
Recall that $v_N$ is the solution to (\ref{gauged}). Proposition \ref{localmain} already implies the convergence of $v_N$ on the short time interval $[-\tau,\tau]$. For the purpose of proving global well-posedness, we need some additional results, namely a commutator estimate, a stability estimate and two convergence results laid out in Proposition \ref{prop:stability} and \ref{conv}  below. For the notations involved in the proof, see Sections \ref{decompsol} and \ref{apriori}. 
\begin{prop}\label{prop:stability} Recall the relevant constants defined in (\ref{defparam}), and that $\tau\ll 1$, $J=[-\tau, \tau]$. The following two statements hold $\tau^{-1}$-certainly.

(1) (Commutator estimate) For any $N\leq N'$, we have
\begin{equation}\label{commutator}\|v_N-\Pi_Nv_{N'}\|_{X^{b}(J)}\leq N^{-1+\gamma};
\end{equation}

(2) (Stability) Let $w=\Pi_Nw$ be a solution to (\ref{gauged}) on $J$, but with data $w(t_0)$ assigned at some $t_0\in J$ such that $\|w(t_0)-v_N(t_0)\|_{L^2}\leq AN^{-1+\gamma}(\log N)^\alpha$, where $\alpha\geq 0$ is an integer, then we have
\begin{equation}\label{perturbation2}
\|w-v_N\|_{X^{b}(J)}\leq BN^{-1+\gamma}(\log N)^{\alpha+1},
\end{equation} where $B$ depends only on $A$ and $\alpha$.
\end{prop}
\begin{proof} (1) It suffices to prove $\|\Pi_Nv_{N'}^\dagger-v_N^\dagger\|_{X^{b}}\leq N^{-1+\gamma}$. We write
\[\Pi_Nv_{N'}^\dagger-v_N^\dagger=\sum_{N<M\leq N'}\Pi_Ny_M^\dagger=\sum_{N<M\leq N'}(\Pi_N\psi_{M,L_0(M)}^\dagger+\Pi_Nz_M^\dagger),\] where $L_0(M)$ is the largest $L$ satisfying $(M,L)\in\mathcal{K}$. The bound for $\Pi_Nz_M^\dagger$ follows from Proposition \ref{localmain2}, so it suffices to bound $\Pi_N\psi_{M,L_0}^\dagger$, where $L_0=L_0(M)$. Let $\psi=\psi_{M,L_0}^\dagger$, then we have
\begin{equation}\label{expand0}\psi(t)=\chi(t) e^{it\Delta}(\Delta_Mf(\omega))-i\chi_{\tau}(t)\sum_{l=0}^r(l+1)c_{rl}(m_M^*)^{r-l}\cdot\mathcal{I}\Pi_M\mathcal{N}_{2l+1}\big(\psi,v_{L_0}^\dagger,\cdots,v_{L_0}^\dagger\big).
\end{equation} Since $N\leq \frac{M}{2}$, $\Pi_N\psi$ solves the equation
\begin{multline}\label{expand1}\Pi_N\psi(t)=-i\chi_{\tau}(t)\sum_{l=0}^r(l+1)c_{rl}(m_M^*)^{r-l}\cdot\mathcal{I}\Pi_N\mathcal{N}_{2l+1}\big(\Pi_N\psi,v_{L_0}^\dagger,\cdots,v_{L_0}^\dagger\big)\\-i\chi_{\tau}(t)\sum_{l=0}^r(l+1)c_{rl}(m_M^*)^{r-l}\cdot\mathcal{I}\Pi_N\mathcal{N}_{2l+1}\big(\Pi_N^\perp\psi,v_{L_0}^\dagger,\cdots,v_{L_0}^\dagger\big).
\end{multline} Now $\tau^{-1}$-certainly we may assume (\ref{induct4}) and the variant of (\ref{mainmult4}) described in Remark \ref{gamma2}. Note that (\ref{induct4}) allows us to control the first line of (\ref{expand1}); the second line of (\ref{expand1}) is controlled by using (\ref{sttime1}) and the variant of (\ref{mainmult4}). In the end we get that
\[\|\Pi_N\psi\|_{X^{b}}\lesssim {\tau}^\theta\|\Pi_N\psi\|_{X^{b}}+{\tau}^\theta M^{-1+\frac{4\gamma}{5}},\] which proves (\ref{commutator}).

(2) If $N\leq O_{A,\alpha}(1)$ there is nothing to prove, so we may assume $N$ is large depending on $(A,\alpha)$. Let $\sigma=\mathcal{A}|w|^2-\mathcal{A}|v_N|^2$ (this is conserved), then we have
\[|\sigma|\lesssim(\|v_N(t_0)\|_{L^2}+\|w(t_0)\|_{L^2})\|w(t_0)-v_N(t_0)\|_{L^2}\lesssim\tau^{-\theta} AN^{-1+\gamma}(\log N)^{\alpha+1}.\] Note the log loss due to the fact that  $\|v_N\|_{L^2}^2\lesssim\tau^{-\theta}\log N$. Recall that $v_N$ and $w$ satisfy the equations
\begin{equation}\label{equationsnew}
\left\{
\begin{aligned}(i\partial_t+\Delta)v_N&=\sum_{l=0}^rc_{rl}(m_N^*)^{r-l}\Pi_N\mathcal{N}_{2l+1}(v_N,\cdots,v_N),\\
(i\partial_t+\Delta)w&=\sum_{l=0}^rc_{rl}(m_N^*+\sigma)^{r-l}\Pi_N\mathcal{N}_{2l+1}(w,\cdots,w)
\end{aligned}
\right.
\end{equation} on $J$, so $z=w-v_N$ satisfies the equation
\begin{multline}\label{equationsnew2}
(i\partial_t+\Delta)z=\sum_{l=0}^rc_{rl}[(m_N^*+\sigma)^{r-l}-(m_N^*)^{r-l}]\Pi_N\mathcal{N}_{2l+1}(v_N,\cdots,v_N)\\+\sum_{l=0}^rc_{rl}(m_N^*+\sigma)^{r-l}\Pi_N[\mathcal{N}_{2l+1}(z+v_N,\cdots,z+v_N)-\mathcal{N}_{2l+1}(v_N,\cdots,v_N)]
\end{multline} on $J$, and $z_0=z(t_0)$ satisfies $\|z_0\|_{L^2}\leq AN^{-1+\gamma}(\log N)^\alpha$. In order to bound $\|z\|_{X^{b}(J)}$, it will suffice to prove that given $z_0$ and $\sigma$, the mapping
\begin{multline}\label{equationsnew3}
z^\dagger\mapsto \chi(t-t_0)e^{i(t-t_0)\Delta}z_0-i\chi_{2\tau}(t-t_0)\sum_{l=0}^rc_{rl}[(m_N^*+\sigma)^{r-l}-(m_N^*)^{r-l}]\mathcal{I}_{t_0}\Pi_N\mathcal{N}_{2l+1}(v_N^\dagger,\cdots,v_N^\dagger)\\-i\chi_{2\tau}(t-t_0)\sum_{l=0}^rc_{rl}(m_N^*+\sigma)^{r-l}\mathcal{I}_{t_0}\Pi_N[\mathcal{N}_{2l+1}(z^\dagger+v_N^\dagger,\cdots,z^\dagger+v_N^\dagger)-\mathcal{N}_{2l+1}(v_N^\dagger,\cdots,v_N^\dagger)]
\end{multline} is a contraction mapping from the set $\{z^\dagger:\|z^\dagger\|_{X^{b}}\leq AN^{-1+\gamma}(\log N)^{\alpha+1}\}$ to itself, where 
\[\mathcal{I}_{t_0}F(t)=\mathcal{I}F(t)-\chi(t)e^{i(t-t_0)\Delta}\mathcal{I}F(t_0);\quad \mathcal{I}_{t_0}F(t)=\chi(t)\int_{t_0}^t\chi(t')e^{i(t-t')\Delta}F(t')\,\mathrm{d}t'.\]To this end we will decompose $v_N^\dagger=\sum_{N'\leq N}y_{N'}^\dagger$ and ($\tau^{-1}$-certainly) apply the estimates (\ref{induct4}) and (\ref{mainmult1}), in the same way as in the proof of (\ref{induct6}). More precisely, we may use (\ref{induct4}) to control the terms in (\ref{equationsnew3}) that contain only one factor $z^\dagger$ (where we use the ${\tau}^\theta$ gain to ensure smallness), and use (\ref{mainmult1}) to control the terms in (\ref{equationsnew3}) that contain at least two factors $z^\dagger$ (where we use the gain of powers of $N$ to ensure smallness, noticing that $N$ is large enough compared to $A$). Note that in applying these estimates we need to replace $\mathcal{I}$ by $\mathcal{I}_{t_0}$ and $\chi_{\tau}(t)$ by $\chi_{2\tau}(t-t_0)$. This can be done because in Section \ref{reductmulti} all estimates for $\chi_{\tau}(t)\cdot\mathcal{I}[\cdots]$ are deduced from (\ref{sttime1}) and the corresponding estimates for $\mathcal{I}[\cdots]$; here by definition we have $\|\mathcal{I}_{t_0}F\|_{X^{\widetilde{b}}}\lesssim\|\mathcal{I}F\|_{X^{\widetilde{b}}}$ for $\widetilde{b}\in\{b,b_1\}$ which allows us to replace $\mathcal{I}$ by $\mathcal{I}_{t_0}$, and that $\mathcal{I}_{t_0}F(t_0)=0$ so (\ref{sttime1}) is still applicable with $\mathcal{I}$ replaced by $\mathcal{I}_{t_0}$ and $\chi_{\tau}(t)$ replaced by $\chi_{2\tau}(t-t_0)$. The rest of the proof will be the same.
\end{proof}
\begin{prop}[Convergence]\label{conv} Recall the relevant constants defined in (\ref{defparam}), the $\varepsilon$ fixed as in Remark \ref{fixep}, and that $\tau\ll 1$ and $J=[-\tau,\tau]$. Then the followings hold ${\tau}^{-1}$-certainly.

(1) For any $N\leq N'$ we have
\begin{equation}\label{converge1}\|v_N-v_{N'}\|_{X^{-\theta,b_2}(J)}\leq {\tau}^{-\theta}N^{-\frac{\theta}{2}},
\end{equation}
\begin{equation}\label{converge2}\|(v_N-e^{it\Delta}v_N(0))-(v_{N'}-e^{it\Delta}v_{N'}(0))\|_{X^{\frac{1}{2}-\gamma_0-\theta,b_2}(J)}\leq N^{-\frac{\theta}{2}}.
\end{equation}
Note that the $X^{s,b}(J)$ bounds also imply the corresponding $C_t^0H_x^s(J)$ bounds.

(2) Let $\mathcal{N}_n(v)$ be a polynomial, also viewed as a multilinear form $\mathcal{N}_n(v^{(1)},\cdots,v^{(n)})$, as in (\ref{multiform}), but is only assumed to be input-simple (instead of simple). Then for any $N\leq N'$, the distance in $C_t^0H_x^{-\varepsilon}(J)$ between any two of the following expressions
\begin{equation}\label{converge}\mathcal{N}_n(v),\,\Pi_N\mathcal{N}_n(v),\,\Pi_{N'}\mathcal{N}_n(v):v\in\{v_N,v_{N'},\Pi_Nv_{N'}\}
\end{equation} is bounded by ${\tau}^{-\theta}N^{-\gamma}$. The same conclusion holds if $\mathcal{N}_n$ is replaced by $W_N^n$ or $W_{N'}^n$, or if $v$ is perturbed by any $w_N$ satisfying $\|w_N\|_{X^{b}(J)}\leq AN^{-1+\gamma}(\log N)^\alpha$. In the latter case the bound will be $O_{A,\tau,\alpha}(1)N^{-\gamma}$.
\end{prop}
\begin{proof} (1) We only need to prove (\ref{converge2}). By taking a summation we may assume $N'=2N$, and it suffices to prove that $\|y_{N'}-e^{it\Delta}(\Delta_{N'}f(\omega))\|_{X^{b_2}(J)}\leq (N')^{-\frac{1}{2}+\gamma_0+\frac{\theta}{2}}$. Now an extension of this function is given by
\[y^\dagger=\sum_{L}\zeta_{N',L}^\dagger+z_{N'}^\dagger,\] see Sections \ref{decompsol} and \ref{apriori}, where $\zeta_{N',L}^\dagger$ is defined from $h^{N',L,\dagger}$ by (\ref{linearity}) and (\ref{matrices}), and $h^{N',L,\dagger}$ and $z_{N'}^\dagger$ satisfy (\ref{induct2}) and (\ref{induct6}). This controls the second term; to bound the first term, we use the $\mathcal{B}_{\leq L}^+$ measurability of $h^{N',L,\dagger}$ and Lemma \ref{largedev}, and perform the same reduction step as in the proof of Claim \ref{pairclaim} using differentiability in $\lambda$ of $\widetilde{h}_{kk^*}(\lambda)$ with $h=h^{N',L,\dagger}$, to bound ${\tau}^{-1}N'$-certainly that
\begin{equation}\label{estfirst}\sum_L\|\zeta_{N,L}^\dagger\|_{X^{b_2}}\lesssim\sum_{L}(N')^{-1}\|h^{N',L,\dagger}\|_{Z^b}\lesssim (N')^{-\frac{1}{2}+\gamma_0+\frac{\theta}{2}}.\end{equation} The right hand side of (\ref{estfirst}) may involve a ${\tau}^{-\theta}$ factor, but this loss can always by recovered by attaching $\chi_{\tau}$ since $y^\dagger(0)=0$.

(2) The bounds for $W^n$ follows from the bounds for $\mathcal{N}_n$ and the formulas (\ref{prop3.2:eq1}) and (\ref{prop3.2:eq2}), noticing that $:\mathrel{|v|^{2r}v}:$ and $:\mathrel{|v|^{2r}}:$ are input-simple. As for $\mathcal{N}_n$, by decomposing \[v_N^\dagger=\sum_{N'\leq N}y_{N'}^\dagger,\quad y_{N'}^\dagger=\chi(t)e^{it\Delta}(\Delta_{N'}f(\omega))+\sum_{L}\zeta_{N',L}^\dagger+z_{N'}^\dagger,\] it suffices to ${\tau}^{-1}N^{(1)}$-certainly bound $\mathcal{N}_n(v^{(1)},\cdots,v^{(n)})$ in $C_t^0H_x^{-\varepsilon}(J)$ by ${\tau}^{-\theta}(N^{(1)})^{-\gamma}$ for $v^{(j)}$ as in the assumptions of Proposition \ref{multi0}. The proof is a much easier variant of the arguments in Section \ref{case1}, so we will only sketch the most important points.

First, since the $\partial_t$ derivative of all the $v^{(j)}$'s are bounded by $(N^{(1)})^{C}$ (by restricting the size of $\lambda_j$ variables as we did in Section \ref{case1}), by dividing $J$ into $(N^{(1)})^{\delta^{-1}}$ intervals we may reduce to  $(N^{(1)})^{C\delta^{-1}}$ exceptional sets and thus fix a time $t\in J$. This gets rid of all the $\lambda_j$ variables (so we are considering $v_{k_j}^{(j)}$ and $h_{k_jk_j^*}^{(j)}$), and by a simple $H_t^{\frac{1}{2}+}\hookrightarrow C_t^0$ argument, the estimates (\ref{input2+}) and (\ref{input3}) remain true with the obvious changes. Now by repeating the arguments in Section \ref{case1} (in a simplified situation without $\lambda_j$ integrations) and Section \ref{conclude} (which deals with over-pairings) we get that
\[\|\Delta_{N_0}\mathcal{N}_n(v^{(1)},\cdots,v^{(n)})\|_{L^2}^2\lesssim {\tau}^{-\theta}(N^{(1)})^{C\kappa^{-1}}(N^{(1)})^{C\gamma}\cdot(\#S)\prod_{j=1}^nN_j^{-2}\prod_{i=1}^p\frac{N_{2i-1}}{R_i},\] where $N_{2i-1}\sim N_{2i}\gtrsim R_i$ and
\begin{multline*}S=\bigg\{(k,k_1,\cdots,k_n)\in(\mathbb{Z}^2)^{n+1}:\sum_{j=1}^n\iota_jk_j=k,\quad |k|\lesssim N_0,\\|k_j|\lesssim N_j\,(1\leq j\leq n),\quad|k_{2i-1}-k_{2i}|\lesssim R_i(N^{(1)})^{C\kappa^{-1}}\,(1\leq i\leq p)\bigg\},\end{multline*} and a simple counting estimate yields
\[\|\Delta_{N_0}\mathcal{N}_n(v^{(1)},\cdots,v^{(n)})\|_{L^2}^2\lesssim {\tau}^{-\theta}(N^{(1)})^{C\gamma}\min(1,(N^{(1)})^{-1}N_0)\lesssim {\tau}^{-\theta}N_0^{\varepsilon}(N^{(1)})^{-\gamma}\] as by our choice $\gamma\ll\varepsilon$, which concludes the proof. With the $w_N$ perturbations the proof works the same way, except that the constants may depend also on $A$ and $\alpha$.
\end{proof}
Before ending this section, we would like to shift the point of view from the probability space $(\Omega,\mathcal{B},\mathbb{P})$ to the spaces $\mathcal{V}$ and $\mathcal{V}_N$. Given $0<\tau\ll 1$, all the above proof has allowed us to identify a Borel set $E_\tau$ of $\mathcal{V}$ with $\rho(E_\tau)\geq 1-C_{\theta}e^{-\tau^{-\theta}}$, such that when $u_{\mathrm{in}}=v_{\mathrm{in}}\in E_\tau$, all the results in Sections \ref{structuresol} and \ref{multiest}, including Propositions \ref{localmain}, \ref{localmain2} and \ref{prop:stability}, are true.

In reality we will be using finite dimensional truncations of $E_\tau$, namely $E_{\tau}^{\overline{N}}=\Pi_{\overline{N}}E_\tau$. Clearly when $\Pi_{\overline{N}}u_{\mathrm{in}}\in E_{\tau}^{\overline{N}}$, all the quantitative estimates proved before will remain true if all the frequencies $N,N',L$, etc., are $\leq\overline{N}$. We moreover know that $\rho_{\overline{N}}(E_{\tau}^{\overline{N}})\geq \rho(E_\tau)\geq 1-C_{\theta}e^{-\tau^{-\theta}}$; since the Radon-Nikodym derivative $\frac{\mathrm{d}\mu_{\overline{N}}^\circ}{\mathrm{d}\rho_{\overline{N}}}$ is uniformly bounded in $L^2(\mathrm{d}\rho_{\overline{N}})$, we have that
\begin{equation}\label{measbound}\mu_{\overline{N}}^\circ(E_{\tau}^{\overline{N}})\geq 1-C\sqrt{\rho_{\overline{N}}(\mathcal{V}_{\overline{N}}\backslash E_{\tau}^{\overline{N}})}\geq 1-C_{\theta}e^{-\tau^{-\theta}}.\end{equation} Finally, due to the gauge symmetry of (\ref{truncnls}) and (\ref{gauged}), we may assume that $E_\tau$ (and $E_\tau^{\overline{N}}$) is rotation invariant, i.e. $e^{i\alpha} E_\tau=E_\tau$ for $\alpha\in\mathbb{R}$.
\section{Global well-posedness and measure invariance}\label{global}
In this section we will prove Theorem \ref{main}. Recall the sets $E_\tau^{\overline{N}}$ defined at the end of Section \ref{stability}. Denote the solution flow of (\ref{truncnls}) by $\Phi_t^{N}$ and the solution flow of (\ref{gauged}) by $\Psi_t^N$, which are mappings from $\mathcal{V}_N$ to itself. Define successively the sets
\begin{equation}\label{defsigma1}F_{T,K}^{\overline{N}}=\bigcap_{|j|\leq K}(\Psi_{\frac{jT}{K}}^{\overline{N}})^{-1}E_{\frac{T}{K}}^{\overline{N}},
\end{equation}
\begin{equation}\label{defsigma2}G_{T,K,A,D}^{\overline{N},\alpha}=\bigg\{v\in \mathcal{V}_{\overline{N}}:\exists t\in[-D,D]\textrm{ s.t. }\Psi_t^{\overline{N}}v=v'+v'',\,\,v'\in F_{T,K}^{\overline{N}},\,\,\|v''\|_{L^2}\leq A{\overline{N}}^{-1+\gamma}(\log\overline{N})^\alpha\bigg\},
\end{equation}
\begin{equation}\label{defsigma3}\Sigma=\bigcup_{D\geq 1}\bigcap_{T\geq 2^{10}D}\bigcup_{K\gg T;A,\alpha\geq 1}\limsup_{\overline{N}\to\infty}\Pi_{\overline{N}}^{-1}G_{T,K,A,D}^{\overline{N},\alpha}.
\end{equation} Here $\Pi_{\overline{N}}^{-1}G=G\times\mathcal{V}_{\overline{N}}^{\perp}$ is the cylindrical set. We understand that $T,K,A,D$ all belong to some given countable set (say powers of two), and $\alpha$ is an integer. All these sets are Borel, since in (\ref{defsigma2}) we may replace the $\leq$ sign by the $<$ sign, and then restrict to rational $t$ by continuity. We will start by proving global well-posedness and then measure invariance.
\begin{prop}\label{globalexist} The set $\Sigma$ satisfies $\mu(\mathcal{V}\backslash\Sigma)=0$, and $W^{2r+1}(u)\in H^{-\varepsilon}$ is well-defined for $u\in\Sigma$. For any $u_{\mathrm{in}}\in\Sigma$, the solutions $u_N(t)=\Phi_t^N\Pi_Nu_{\mathrm{in}}$ to (\ref{truncnls}) converge to some $u(t)=\Phi_tu_{\mathrm{in}}$ in $C_t^0H_x^{-\varepsilon}([-T,T])$ for any $T>0$. This $u$ is a distributional solution to (\ref{nls}), and $u(t)\in\Sigma$ for each $t$. The mappings $\Phi_t:\Sigma\to\Sigma$ satisfy $\Phi_0=\mathrm{Id}$ and $\Phi_{t+t'}=\Phi_t\Phi_{t'}$.
\end{prop}
\begin{proof} We first prove $\mu(\mathcal{V}\backslash\Sigma)=0$. By definition we have
\begin{equation}\label{setinclusion}\Sigma\supset\bigcap_{T\geq 2^{10}}\bigcup_{K\gg T}\limsup_{\overline{N}\to\infty}\Pi_{\overline{N}}^{-1}F_{T,K}^{\overline{N}},
\end{equation} so it suffices to prove for any fixed $T\geq 2^{10}$ that
\[\sup_{K\gg T}\mu\bigg(\limsup_{\overline{N}\to\infty}\Pi_{\overline{N}}^{-1}F_{T,K}^{\overline{N}}\bigg)=1.\] Now by Fatou's lemma and the fact that the total variation of $\mu-\mu_{\overline{N}}$ converges to $0$, we have
\[\mu\bigg(\limsup_{\overline{N}\to\infty}\Pi_{\overline{N}}^{-1}F_{T,K}^{\overline{N}}\bigg)\geq\limsup_{\overline{N}\to\infty}\mu_{\overline{N}}\big(\Pi_{\overline{N}}^{-1}F_{T,K}^{\overline{N}}\big)=\limsup_{\overline{N}\to\infty}\mu_{\overline{N}}^\circ\big(F_{T,K}^{\overline{N}}\big).\] By invariance of $\mathrm{d}\mu_{\overline{N}}^\circ$ under the flow $\Psi_{t}^{\overline{N}}$ (Proposition \ref{measurefact}) we know that
\[\mu_{\overline{N}}^\circ\big(F_{T,K}^{\overline{N}}\big)\geq 1-(2K+1)\mu_{\overline{N}}^\circ\big(\mathcal{V}_{\overline{N}}\backslash E_{\frac{T}{K}}^{\overline{N}}\big)\geq 1-C_\theta Ke^{-(KT^{-1})^\theta}\] uniformly in $\overline{N}$, and the right hand side converges to $1$ as $K\to\infty$, so $\mu(\mathcal{V}\backslash\Sigma)=0$.

Now suppose $u_{\mathrm{in}}\in\Sigma$. By definition we may choose some $D$, then for any $T\geq 2^{10}D$ we can find $(K,A,\alpha)$ such that $\Pi_{\overline{N}}u_{\mathrm{in}}\in G_{T,K,A,D}^{\overline{N},\alpha}$ for infinitely many $\overline{N}$. We may fix this $T$ (hence also $(K,A,\alpha)$) and this $\overline{N}$, so that $\|\Psi_{t_0}^{\overline{N}}\Pi_{\overline{N}}u_{\mathrm{in}}-v'\|_{L^2}\leq A\overline{N}^{-1+\gamma}(\log\overline{N})^\alpha$ for some $t_0\in[-D,D]$ and $v'\in F_{T,K}^{\overline{N}}$. We proceed in three steps.

\emph{Step 1: analyzing $v'$}. We first prove that, for any $N\leq \overline{N}$ and $|j|\leq K$ there holds that
\begin{equation}\label{commute1}\|\Pi_N\Psi_{\frac{jT}{K}}^{\overline{N}}v'-\Psi_{\frac{jT}{K}}^N\Pi_Nv'\|_{L^2}\leq BN^{-1+\gamma}(\log N)^{|j|}
\end{equation} for some $B$ depending only on $(T,K)$. This is obviously true for $j=0$; suppose this is true for $j$, since $\Psi_{\frac{jT}{K}}^{\overline{N}}v'\in E_{\frac{T}{K}}^{\overline{N}}$, by Proposition \ref{prop:stability} (1) we have
\[\|\Pi_N\Psi_{\frac{(j\pm 1)T}{K}}^{\overline{N}}v'-\Psi_{\frac{\pm T}{K}}^N\Pi_N\Psi_{\frac{jT}{K}}^{\overline{N}}v'\|_{L^2}\leq N^{-1+\gamma}
\] (note that as $K\gg T$ the local theory is applicable on intervals of length $\frac{T}{K}$), and
\[\|\Psi_{\frac{\pm T}{K}}^N\Pi_N\Psi_{\frac{jT}{K}}^{\overline{N}}v'-\Psi_{\frac{(j\pm 1)T}{K}}^N\Pi_Nv'\|_{L^2}\leq B'N^{-1+\gamma}(\log N)^{|j|+1}\] by Proposition \ref{prop:stability} (2) and (\ref{commute1}), where $B'$ depends only on $B$ and $(T,K)$, so (\ref{commute1}) holds also for $j\pm 1$ which concludes the inductive proof. By the same argument, we can show that (\ref{commute1}) remains true with $\frac{jT}{K}$ replaced by any $t\in[-T,T]$ and $|j|$ replaced by $K$.

Similarly, since $\Psi_{\frac{jT}{K}}^{\overline{N}}v'\in E_{\frac{T}{K}}^{\overline{N}}$ for each $|j|\leq K$, by combining Propositions \ref{prop:stability} and \ref{conv}, we conclude that for any $N\leq N'\leq\overline{N}$,
\begin{equation}\label{propv'}\sup_{t\in[-T,T]}\|\Psi_{t}^N\Pi_Nv'-\Psi_{t}^{N'}\Pi_{N'}v'\|_{H^{-\theta}}\leq  O_{T,K}(1) N^{-\frac{\theta}{2}},
\end{equation}
\begin{equation}\label{propv'2}\sup_{t\in[-T,T]}\|W_N^n(\Psi_{t}^N\Pi_Nv')-W_{N'}^n(\Psi_{t}^{N'}\Pi_{N'}v')\|_{H^{-\varepsilon}}\leq  O_{T,K}(1) N^{-\gamma},
\end{equation} and the same is true if $W_N^n$ and $W_{N'}^n$ in (\ref{propv'2}) is replaced by $\Pi_NW_N^n$ and $\Pi_{N'}W_{N'}^n$.

\emph{Step 2: linking $u_{\mathrm{in}}$ to $v'$}. Recall that $\|\Psi_{t_0}^{\overline{N}}\Pi_{\overline{N}}u_{\mathrm{in}}-v'\|_{L^2}\leq A\overline{N}^{-1+\gamma}(\log\overline{N})^\alpha$. Since $|t_0|\leq D\ll T$, by iterating Proposition \ref{prop:stability} (2) we deduce that $\|\Pi_{\overline{N}}u_{\mathrm{in}}-\Psi_{-t_0}^{\overline{N}}v'\|_{L^2}\leq A'\overline{N}^{-1+\gamma}(\log\overline{N})^{\alpha+K}$ where $A'$ depends only on $(T,K,A,\alpha)$. Writing $-t_0=\frac{jT}{K}+t'$ with $|j|\leq 2^{-8} K$ and $|t'|\leq\frac{T}{K}$, we may apply Proposition \ref{prop:stability} (2) again and combine this with (\ref{commute1}) and similar estimates to deduce for any $N\leq\overline{N}$ that
\begin{equation}\label{link}\sup_{t\in[-\frac{T}{2},\frac{T}{2}]}\|\Psi_t^N\Pi_Nu_{\mathrm{in}}-\Psi_{t-t_0}^N\Pi_Nv'\|_{L^2}\leq BN^{-1+\gamma}(\log N)^{\alpha+2K}
\end{equation} with $B$ depending only on $(T,K,A,\alpha)$. By (\ref{propv'}), (\ref{propv'2}), (\ref{link}) and Proposition \ref{conv}, we conclude for all $N\leq N'\leq \overline{N}$ that \begin{equation}\label{propuin}\sup_{t\in[-\frac{T}{2},\frac{T}{2}]}\|\Psi_{t}^N\Pi_Nu_{\mathrm{in}}-\Psi_{t}^{N'}\Pi_{N'}u_{\mathrm{in}}\|_{H^{-\theta}}\leq O_{T,K,A,\alpha}(1) N^{-\frac{\theta}{2}},
\end{equation}
\begin{equation}\label{propuin2}\sup_{t\in[-\frac{T}{2},\frac{T}{2}]}\|W_N^n(\Psi_{t}^N\Pi_Nu_{\mathrm{in}})-W_{N'}^n(\Psi_{t}^{N'}\Pi_{N'}u_{\mathrm{in}})\|_{H^{-\varepsilon}}\leq  O_{T,K,A,\alpha}(1) N^{-\gamma},
\end{equation} and the same is true for projections of $W_N^n$.

\emph{Step 3: completing the proof}. Now, for fixed $(D,T,K,A,\alpha)$ we know that there exists infinitely many $\overline{N}$ such that (\ref{propuin}) and (\ref{propuin2}) are true for all $N\leq N'\leq \overline{N}$, so (\ref{propuin}) and (\ref{propuin2}) are simply true for all $N\leq N'$. This implies the convergence of $\Psi_{t}^N\Pi_Nu_{\mathrm{in}}$ in $C_t^0H_x^{-\varepsilon}([-\frac{T}{2},\frac{T}{2}])$, and we will define $\Psi_t=\lim_{N\to\infty}\Psi_t^N\Pi_N$. Since by the definition of gauge transform we have 
\begin{equation}\label{linkphipsi}\Phi_t^N\Pi_Nu_{\mathrm{in}}=\Psi_t^N\Pi_Nu_{\mathrm{in}}\cdot e^{-iB_N(t)},\quad B_N(t)=(r+1)\int_0^t\mathcal{A}[W_N^{2r}(\Psi_t^N\Pi_Nu_{\mathrm{in}})]\,\mathrm{d}t',
\end{equation} (\ref{propuin}) and (\ref{propuin2}) also implies the convergence of $\Phi_{t}^N\Pi_Nu_{\mathrm{in}}$ in $C_t^0H_x^{-\varepsilon}([-\frac{T}{2},\frac{T}{2}])$, as well as the convergence of $\Pi_NW_N^{2r+1}(\Phi_t^N\Pi_Nu_{\mathrm{in}})$ in the same space. As $u_N=\Phi_{t}^N\Pi_Nu_{\mathrm{in}}$ solves the equation (\ref{truncnls}) with right hand side being $\Pi_NW_N^{2r+1}(u_N)$, we know that the limit $u=\lim_{N\to\infty}u_N$ solves (\ref{nls}) in the distributional sense. Let $u(t)=\Phi_tu$, the group properties of $\Phi_t$ follow from the group properties of $\Psi_t^N$ and limiting arguments similar to the above.

Finally we prove that $\Phi_{t_1}u_{\mathrm{in}}\in \Sigma$ for $u_{\mathrm{in}}\in\Sigma$ and any $t_1$. Let $D$ be associated with the assumption $u_{\mathrm{in}}\in\Sigma$, and fix $D_1\gg D+|t_1|$. For any $T\geq 2^{10}D_1$ there exists $(K,A,\alpha)$ such that $\Pi_{\overline{N}}u_{\mathrm{in}}\in G_{T,K,A,D}^{\overline{N},\alpha}$ for infinitely many $\overline{N}$. It suffices to show that for such $\overline{N}$ we must have $\Pi_{\overline{N}}\Phi_{t_1}u_{\mathrm{in}}\in G_{T,K,B,D_1}^{\overline{N},\alpha+3K}$ with $B$ depending only on $(T,K,A,\alpha)$. Since $\Pi_{\overline{N}}\Phi_{t_1}u_{\mathrm{in}}$ and $\Pi_{\overline{N}}\Psi_{t_1}u_{\mathrm{in}}$ only differs by a rotation and the sets we constructed are rotation invariant, we only need to prove the same thing for $\Pi_{\overline{N}}\Psi_{t_1}u_{\mathrm{in}}$.

Now, on the one hand we know for some $|t_0|\leq D$ that $\|\Pi_{\overline{N}}u_{\mathrm{in}}-\Psi_{-t_0}^{\overline{N}}v'\|_{L^2}\leq A'\overline{N}^{-1+\gamma}(\log\overline{N})^{\alpha+K}$ for some $A'$ depending only on $(T,K,A,\alpha)$ (see \emph{Step 2}) and similarly $\|\Psi_{t_1}^{\overline{N}}\Pi_{\overline{N}}u_{\mathrm{in}}-\Psi_{t_1-t_0}^{\overline{N}}v'\|_{L^2}\leq A'\overline{N}^{-1+\gamma}(\log\overline{N})^{\alpha+2K}$. On the other hand by taking limits in (\ref{commute1}) and (\ref{link}) we also get $\|\Pi_{\overline{N}}\Psi_{t_1}u_{\mathrm{in}}-\Psi_{t_1}^{\overline{N}}\Pi_{\overline{N}}u_{\mathrm{in}}\|_{L^2}\leq A'\overline{N}^{-1+\gamma}(\log\overline{N})^{\alpha+2K}$, and hence $\|\Pi_{\overline{N}}\Psi_{t_1}u_{\mathrm{in}}-\Psi_{t_1-t_0}^{\overline{N}}v'\|_{L^2}\leq A'\overline{N}^{-1+\gamma}(\log\overline{N})^{\alpha+2K}$. Applying Proposition \ref{prop:stability} (2) again we get that $\|\Psi_{t_0-t_1}^{\overline{N}}\Pi_{\overline{N}}\Psi_{t_1}u_{\mathrm{in}}-v'\|_{L^2}\leq B\overline{N}^{-1+\gamma}(\log\overline{N})^{\alpha+3K}$ with $|t_0-t_1|\leq D_1$ and $B\leq B(T,K,A,\alpha)$, thus by definition $\Pi_{\overline{N}}\Psi_{t_1}u_{\mathrm{in}}\in G_{T,K,B,D_1}^{\overline{N},\alpha+3K}$. This completes the proof. \end{proof}
\begin{prop}\label{last} For any Borel subset $E\subset\Sigma$ and any $t_0\in\mathbb{R}$, we have $\mu(E)=\mu(\Phi_{t_0}E)$.
\end{prop}
\begin{proof} The map $\Phi_t$ is a limit of continuous mappings, so it is Borel measurable. By taking limits, we may assume the set $E$ is compact in $H^{-\varepsilon}$ topology. We may also assume that $|t_0|\leq 1$, and that for some fixed $(T,K,A,\alpha,D)$ with $K\gg T\geq 2^{10}D$ we have $E\subset\limsup_{\overline{N}\to\infty}G_{T,K,A,D}^{\overline{N},\alpha}$. By the proof of Proposition \ref{globalexist} we can deduce that  for $u\in E$ and $|t|\leq 2$, \begin{equation}\label{commfin1}\|\Psi_t^{N}\Pi_Nu-\Pi_N\Psi_tu\|_{L^2}\lesssim N^{-1+\gamma}(\log N)^{\alpha+3K}\end{equation} with constants depending on $(T,K,A,\alpha)$ (same below). Moreover, concerning the phase $B_N(t)$ involved in the gauge transform, namely
\begin{equation*}B_N(t)=(r+1)\int_0^t\mathcal{A}[W_N^{2r}(\Psi_t^N\Pi_Nu)]\,\mathrm{d}t',
\end{equation*} we can show that as $N\to\infty$, $B_N(t)$ converges to its limit $B(t)$ at a rate $\|B_N(t)-B(t)\|_{C_t^0([-2,2])}\lesssim N^{-1+\gamma}(\log N)^{\alpha+4K}$. In fact we may first reduce to short time intervals where local theory is applicable, then notice that
\[\int_0^t\mathcal{A}[W_N^{2r}(\Psi_t^N\Pi_Nu)]\,\mathrm{d}t'=\mathcal{A}\mathcal{I}[W_N^{2r}(\Psi_t^N\Pi_Nu)]\] for $|t|\ll 1$, and apply Proposition \ref{multi0}, more precisely (\ref{mainmult2}), with the observation that the mean $\mathcal{A}$ restricts the two highest input frequencies in any multilinear expression $\mathcal{N}_n$ occurring in $W_N^{2r}$ to be comparable, i.e. $N^{(1)}\sim N^{(2)}$. We omit the details.

With the explicit convergence rate of $B_N(t)$, we see that (\ref{commfin1}) holds with $\Psi_t^N$ and $\Psi_t$ replaced by $\Phi_t^N$ and $\Phi_t$, and with $3K$ replaced by $4K$. This gives, for $|t|\leq 1$, that
\[\Pi_N\Phi_tE\subset\Phi_t^N\Pi_NE+\mathfrak{B}_{L^2}(A_1N^{-1+\gamma}(\log N)^{\alpha+4K})\subset\Phi_t^N(\Pi_NE+\mathfrak{B}_{L^2}(A_2N^{-1+\gamma}(\log N)^{\alpha+5K})),\] where $A_{1,2}$ are constants depending only on $(T,K,A,\alpha)$, and $\mathfrak{B}_{L^2}(R)$ is the ball of radius $R$ in $L^2$ centered at the origin; note that the second subset relation follows from long-time stability, which is also a consequence of the proof of Proposition \ref{globalexist}. By invariance of $\mathrm{d}\mu_N^\circ$ under $\Phi_t^N$ we have that
\begin{multline}\mu_N(\Phi_{t_0}E)\leq\mu_N^\circ(\Pi_N\Phi_{t_0}E)\leq\mu_N^\circ\Phi_{t_0}^N(\Pi_NE+\mathfrak{B}_{L^2}(A_2N^{-1+\gamma}(\log N)^{\alpha+5K}))\\=\mu_N^\circ(\Pi_NE+\mathfrak{B}_{L^2}(A_2N^{-1+\gamma}(\log N)^{\alpha+5K})).\end{multline} It then suffices to prove that
\begin{equation}\label{compactness}\limsup_{N\to\infty}\Pi_N^{-1}(\Pi_NE+\mathfrak{B}_{L^2}(A_2N^{-1+\gamma}(\log N)^{\alpha+5K}))\subset E,\end{equation} which would imply $\mu(\Phi_{t_0}E)\leq \mu(E)$, and conclude the proof by time reversibility\footnote{By repeating the proof of Proposition \ref{globalexist} we can show that $\Phi_t^N\Pi_Nu\to\Phi_tu$ in $H^{-\varepsilon}$ uniformly for $u\in E$ and $|t|\leq 2$, so $\Phi_{t_0}E$ is also compact in $H^{-\varepsilon}$ and satisfies similar properties as $E$.}. To prove (\ref{compactness}), suppose $u$ is such that $\|\Pi_N(u-u_N)\|_{L^2}\leq A_2N^{-1+\gamma}(\log N)^{\alpha+5K}$ with $u_N\in E$ for infinitely many $N$, then by compactness we may assume $u_N\to v\in E$ in $H^{-\varepsilon}$, so $u_N\to u$ coordinate-wise and $u_N\to v$ coordinate-wise, hence $u=v\in E$ and the proof is complete.
\end{proof}
\begin{rem} With a little more effort, we can show that the structural information for the short time solution, i.e. (\ref{ansatz1}), is propagated for arbitrarily long time. This follows by iterating short time intervals using measure invariance, and applying (\ref{commutator}) and (\ref{perturbation2}) in Proposition \ref{prop:stability}, in the same way as in the proof of global well-posedness (Proposition \ref{globalexist}) above. We omit the details.
\end{rem}


\begin{thebibliography}{99}
\bibitem{Aiz} M. Aizenman. Geometric analysis of $\varphi^4$ fields and Ising models. Part I and II, \emph{Comm. Math. Phys.} {\bf 86} (1982), 1--48.
\bibitem{ADC} M. Aizenman and H. Duminil-Copin. Marginal triviality of the scaling limits of critical 4D Ising and $\varphi^4$ models. \emph{Ann. of Math. (2)}  {\bf 194} no.1 (2021) 163--235.
\bibitem{AC}  S. Albeverio and A. Cruzeiro. Global flows with invariant (Gibbs) measures for Euler and Navier-Stokes two dimensional fluids, \emph{Comm. Math. Phys.} {\bf 129} (1990), 431--444.
\bibitem{BaBe} I. Bailleul and F. Bernicot. High order paracontrolled calculus. \emph{Forum of Math. Sigma} {\bf 7} (2016), e44.
\bibitem{BaBe2} I. Bailleul, I. and F. Bernicot.  Heat semigroup and singular PDEs, \emph{J. Funct. Anal.} {\bf 270} no. 9 (2016), 3344--3452.
\bibitem{BaBeFr} I. Bailleul, F. Bernicot and D. Frey. Space-time paraproducts for paracontrolled calculus, 3D-PAM and multiplicative Burgers equations, \emph{Ann. Sci. \'Ec Norm. Sup\'er. (4)} {\bf 51} no. 6 (2018), 1399--1456.
\bibitem{BOP}  \'A. B\'enyi, T. Oh and O. Pocovnicu. On the probabilistic Cauchy theory of the cubic nonlinear Schr\"{o}dinger equation on $\mathbb{R}^d$. \emph{Trans. Amer. Math. Soc. Ser. B} {\bf 2} (2015), 1--50.
\bibitem{BOhP1}  \'A. B\'enyi, T. Oh and O. Pocovnicu. Higher order expansions for the probabilistic local Cauchy theory of the cubic nonlinear Schr\"{o}dinger equation on $\mathbb{R}^3$, {\em Trans. Amer. Math. Soc. Ser. B.} {\bf 6} no. 4 (2019), 114--160.
\bibitem{Bogachev} V. I. Bogachev. {\em Gaussian measures}, Mathematical Surveys and Monographs {\bf 62}. American Mathematical Society, Providence, RI, (1998), xii+433.
\bibitem{Bourgain94} J. Bourgain. Periodic nonlinear Schr\"{o}dinger equation and invariant measures, \emph{Comm. Math. Phys.} {\bf 166} (1994), 1--26.
\bibitem{Bourgain} J. Bourgain. Invariant measures for the 2D-defocusing nonlinear Schr\"{o}dinger equation, \emph{Comm. Math. Phys.} {\bf 176} (1996), 421--445.
\bibitem{Bourgain2} J. Bourgain. {\em Global solutions of nonlinear Schr\"odinger equations},  Amer. Math. Soc. Colloq. Pub. {\bf 46}, Amer. Math. Soc. (1999).
\bibitem{Bou99} J. Bourgain. Nonlinear Schr\"{o}dinger equations. In \emph{Hyperbolic equations and frequency interactions (Park City, UT, 1995)}, volume {\bf 5} of \emph{IAS/Park City Math. Ser.}, pages 3--157. Amer. Math. Soc., Providence, RI, 1999.
\bibitem{BB4} J. Bourgain and A. Bulut. Almost sure global well-posedness for the radial nonlinear Schr\"odinger equation on the unit ball II: the 3d case, {\em J. Eur. Math. Soc. (JEMS)} {\bf 16} no. 6 (2014), 1289--1325.
\bibitem{Bringmann} B. Bringmann. Almost sure local well-posedness for a derivative nonlinear wave equation. \emph{Int. Math. Res. Not. (IMRN)} {no. 11} (2021) 8657--8697.
\bibitem{BDNY22} B. Bringmann, Y. Deng, A. Nahmod and H. Yue. Invariant Gibbs measures for the three dimensional cubic nonlinear wave equation. To appear in \emph{Invent. Math.}, arXiv:2205.03893, 2022.
\bibitem{BrySl} D. Brydges and G. Slade. Statistical Mechanics of the 2-Dimensional Focusing Nonlinear Schr\"odinger Equation, \emph{Commun. Math. Phys.} {\bf 182} (1996), 485-504. 
\bibitem{BTT} N. Burq, L. Thomann and N. Tzvetkov. Remarks on the Gibbs measures for nonlinear dispersive equations, \emph{Ann. Fac. Sci. Toulouse Math.}, Ser. 6 Vol. {\bf 27} no. 3 (2018), 527--597.
\bibitem{BTlocal} N. Burq and N. Tzvetkov. Random data Cauchy theory for supercritical wave equations I: local theory, \emph{Invent. Math.}  {\bf 173}  no. 3 (2008), 449--475.
\bibitem{CCh} R. Catellier and K. Chouk.  Paracontrolled distributions and the 3-dimensional stochastic quantization equation,  {\em Ann Prob.} {\bf 46} no. 5 (2018), 2621--2679.
\bibitem{ChW} A. Chandra and H. Weber.  Stochastic PDEs, regularity structures, and interacting particle systems, \emph{Ann. Fac. Sci. Toulouse Math.}  (6) {\bf 26} no. 4 (2017), 847--909.
\bibitem{CoOh} J. Colliander and T. Oh. Almost sure well-posedness of the cubic nonlinear Schr\"{o}dinger equation below $L^2(\mathbb{T})$, {\em Duke Math. J.}  {\bf 161} no. 3 (2012), 367--414.
\bibitem{DD} G. Da Prato and A. Debussche. Two-dimensional Navier-Stokes equations driven by a space-time white noise, \emph{J. Funct. Anal.} {\bf 196} no. 1 (2002), 180--210.
 \bibitem{DD2} G. Da Prato and A. Debussche. Strong solutions to the stochastic quantization equations, \emph{Ann. Probab.} {\bf 31} no .4 (2003), 1900--1916.
\bibitem{DaT} G. Da Prato and L. Tubaru. Wick powers in stochastic PDEs: an introduction, Technical Report UTM 711, 2006, 39 pp. \url{http://eprints.biblio.unitn.it/1189/1/UTM711.pdf}.
\bibitem{Deng} Y. Deng. Two dimensional nonlinear Schr\"{o}dinger equation with random radial data, \emph{Anal. PDE} {\bf 5} no. 5 (2012), 913--960.
\bibitem{Deng2} Y. Deng. Invariance of the Gibbs measure for the Benjamin-Ono equation. \emph{J. Eur. Math. Soc. (JEMS)} {\bf 17}  no. 5 (2015), 1107--1198.
\bibitem{DNY} Y. Deng, A. Nahmod and H. Yue. Optimal local well-posedness for the periodic derivative nonlinear Schr\"{o}dinger equation. \emph{Comm. Math. Phys.} {\bf 384} no. 2 (2021), 1061--1107.
\bibitem{DNY22} Y. Deng, A. Nahmod and H. Yue. Random tensors, propagation of randomness, and nonlinear dispersive equations. \emph{Invent. Math.} {\bf 228} no. 2 (2022), 539--686.
\bibitem{DNY24} Y. Deng, A. Nahmod and H. Yue. The Probabilistic scaling paradigm. \emph{Vietnam J. Math.} (2024) \url{https://doi.org/10.1007/s10013-023-00672-w}.
\bibitem{DTV} Y. Deng, N. Tzvetkov, and N. Visciglia. Invariant Measures and Long Time Behavior for the Benjamin-Ono Equation III. \emph{Comm. Math. Phys.} {\bf 339} no. 3 (2015), 815--857.
\bibitem{DLM0}  B. Dodson, J. L\"uhrmann and D. Mendelson, Almost sure local well-posedness and scattering for the 4D cubic nonlinear Schr\"odinger equation, {\em Adv. Math.}  {\bf 347} (2019), 619--676. 
\bibitem{Fre85} L. Friedlander,  An invariant measure for the equation $u_{tt}-u_{xx}+u^3=0$. \emph{Comm. Math. Phys.} {\bf 98} (1985), 1--16.
\bibitem{FrH}  P. K. Friz and M. Hairer. {\em A course on rough paths, with an introduction to regularity structures}, Universitext. Springer, Cham, 2014. xiv+251 pp.
\bibitem{Fro} J. Fr\"ohlich.  On the triviality of $\lambda \phi^4_d$ theories and the approach to the critical point in $d_{(-)}> 4$ dimensions, \emph{Nuclear Physics B} {\bf 200}, issue 2 (1982), 281--296.
\bibitem{GJ1} J. Glimm and A. Jaffe. {\em Quantum physics,  A functional integral point of view}, Second edition, Springer-Verlag, New York, 1987. xxii+535 pp.
\bibitem{GIP} M. Gubinelli, P. Imkeller and N. Perkowski. Paracontrolled distributions and singular PDEs, \emph{Forum Math. Pi} 3, e6, {\bf 75} pp, 2015.
\bibitem{GIP2} M. Gubinelli, P. Imkeller and N. Perkowski.  A Fourier analytic approach to pathwise stochastic integration, \emph{Electron. J. Probab.} {\bf 21} no. 2, (2016) 37 pp.  
\bibitem{GKO} M. Gubinelli, H. Koch and T. Oh. Renormalization of the two-dimensional stochastic nonlinear wave equations, \emph{Trans. Amer. Math. Soc.} {\bf 370} no. 10 (2018), 7335--7359.
\bibitem{GKO2} M. Gubinelli, H. Koch and T. Oh. Paracontrolled approach to the three-dimensional stochastic nonlinear wave equation with quadratic nonlinearity. \emph{J. Eur. Math. Soc. (JEMS)} {\bf 26}  no. 3 (2024), 817--874
\bibitem{GP}  M. Gubinelli and N. Perkowski. Lectures on singular stochastic PDEs, \emph{Ensaios Matem\'aticos, Mathematical Surveys}, 29. Sociedade Brasileira de Matem\'atica, Rio de Janeiro. (2015), 89 pp.
\bibitem{GP2} M. Gubinelli and N. Perkowski. KPZ reloaded, \emph{Comm. Math. Phys.} {\bf 349} no. 1 (2017), 165--269.
\bibitem{GP3} M. Gubinelli and N. Perkowski.  Energy solutions of KPZ are unique, \emph{J. Amer. Math. Soc.} {\bf 31} (2018), no. 2, 427--471.
\bibitem{GP4} M. Gubinelli and N. Perkowski. An introduction to singular SPDEs, in \emph{Stochastic partial differential equations and related fields}, 69--99, \emph{Springer Proc. Math. Stat.},  229, Springer, Cham, 2018. 
\bibitem{GO15} Z. Guo and T. Oh. Non-existence of solutions for the periodic cubic NLS below $L^2$. \emph{Int. Math. Res. Not. (IMRN)} {no. 6} (2018),  1656--1729.
\bibitem{Hairer0} M. Hairer. Solving the KPZ equation, \emph{Ann. of Math. (2)} {\bf 178}  (2013), 559--664.
\bibitem{Hairer} M. Hairer. A theory of regularity structures, \emph{Invent. Math.} {\bf 198} no. 2 (2014), 269--504.
\bibitem{Hairer2}  M. Hairer. Singular Stochastic PDE, \emph {Proceedings of the ICM-Seoul}, Vol. I, (2014), 685-709.
\bibitem{Hairer3} M. Hairer. Introduction to regularity structures, \emph{ Braz. J. Probab. Stat.} {\bf 29} no. 2 (2015), 175--210. 
\bibitem{Hairer4} M. Hairer.  Regularity structures and the dynamical $\Phi^4_3$ model, \emph{Current Developments in Mathematics 2014}, Int. Press, Somerville, MA,  (2016), 1--49.
\bibitem{Hairer5} M. Hairer.  Renormalisation of parabolic stochastic PDEs,  \emph{Jpn. J. Math.} {\bf 13} no. 2 (2018), 187--233.     
\bibitem{HLab} M. Hairer and C. Labb\'e. The reconstruction theorem in Besov spaces, \emph{J. Funct. Anal.} {\bf 273} (2017), 2578--2618. 
\bibitem{IW} K. Iwata.   An infinite-dimensional stochastic differential equation with state space $C(\mathbf{R})$. \emph{Probab. Theory Related Fields} {\bf 74} no. 1 (1987), 141--159.
\bibitem{KM} C. Kenig and D. Mendelson, The focusing energy-critical nonlinear wave equation with random initial data. {\em Int. Math. Res. Not.} {\bf 19} (2021), 14508--14615.
\bibitem{KMV} R. Killip, J. Murphy and M. Visan, Almost sure scattering for the energy-critical NLS with radial data below $H^1(\mathbb{R}^4)$, {\em Comm. Partial Differential Equations} {\bf 44} (2019), no. 1, 51--71.
\bibitem{KRS} S. Klainerman, I. Rodnianski and J. Szeftel. The bounded $L^2$ curvature conjecture, \emph{Invent. Math.} {\bf 202} no. 1 (2015), 91--216.
\bibitem{Kupia}  A. Kupiainen. Renormalization group and stochastic PDEs, \emph{Annales Henri Poincar\'e} {\bf 17} no. 3 (2016), 497--535.
\bibitem{LRS} J. Lebowitz, R. Rose and E. Speer. Statistical mechanics of the nonlinear Schr\"{o}dinger equation, \emph{J. Statist. Phys.} {\bf 50} (1988), 657--687.
\bibitem{LM} J. L\"uhrmann and D. Mendelson.  Random data Cauchy theory for nonlinear wave equations of power type on $\mathbb{R}^3$. \emph{Comm. Partial Differential Equations} {\bf 39} no.12 (2014), 2262--2283.
\bibitem{MouWe3}  J.C. Mourrat and H. Weber. The dynamic  $\Phi^4_3$ model comes down from infinity, \emph{Comm. Math. Phys.} {\bf 356} no. 3 (2017), 673--753.
\bibitem{MouWeXu} J.C. Mourrat, H. Weber and W. Xu.  Construction of $\Phi^4_3$ diagrams for pedestrians, in \emph{From particle systems to partial differential equations}, 1--46, \emph{Springer Proc. Math. Stat.}, {\bf 209}, Springer, Cham, 2017.
\bibitem{NORS} A. Nahmod, T. Oh, L. Rey-Bellet and G. Staffilani. Invariant weighted Wiener measures and almost sure global well-posedness for the periodic derivative NLS, \emph{J. Eur. Math. Soc. (JEMS)} {\bf 14} no. 4 (2012), 1275--1330.
\bibitem{NPST}  A. Nahmod, N.  Pavlovic, G. Staffilani and N. Totz.  Global Flows with Invariant Measures for the Inviscid Modified SQG Equations,  \emph{Stoch. Partial Differ. Equ. Anal. Comput.} {\bf 6} no. 2 (2018), 184--210.
\bibitem{NRSS} A. Nahmod, L. Rey-Bellet, S. Sheffield and G. Staffilani. 
Absolute continuity of Brownian bridges under certain gauge transformations,  \emph{Math. Res. Lett.} {\bf 18} no. 5 (2011), 875--887.
\bibitem{NS} N. Nahmod and G. Staffilani. Almost sure well-posedness for the periodic 3D quintic nonlinear Schr\"odinger equation below the energy space, \emph{ J. Eur. Math. Soc. (JEMS)} {\bf 17} no. 7 (2015), 1687--1759.
\bibitem{Nel1} E. Nelson. A quartic interaction in two dimensions, 1966 Mathematical Theory of Elementary Particles,
\emph{Proc. Conf., Dedham, Mass. 1965}  69--73,  MIT Press, Cambridge, Mass.
\bibitem{Nel2} E. Nelson. Construction of quantum fields from Markoff fields, \emph{J. Functional Analysis} {\bf 12} (1973), 97--112.
\bibitem{Oh} T. Oh. Invariant Gibbs measures and a.s. global well-posedness for coupled KdV systems, \emph{Diff. Int. Eq.} {\bf 22} no. 7--8 (2009), 637--668.
\bibitem{OT} T. Oh and L. Thomann. A pedestrian approach to the invariant Gibbs measures for the 2-d defocusing nonlinear Schr\"{o}dinger equations, \emph{Stoch. Partial Differ. Equ. Anal. Comput.} {\bf 6} no. 3 (2018), 397--445.
\bibitem{OT2} T. Oh and L. Thomann. Invariant Gibbs measure for the 2-d defocusing nonlinear wave equations. \emph{Ann. Fac. Sci. Toulouse Math.} {\bf 29} no. 1 (2020),  1--26.
\bibitem{Perk} N. Perkowski. Para-controlled distributions and singular SPDEs, \url{https://personal-homepages.mis.mpg.de/perkow/paracontrolled-leipzig.pdf}
\bibitem{Giordi} G. Richards. Invariance of the Gibbs measure for the periodic quartic gKdV, \emph{Ann. Inst. H. Poincar\'e Anal. Non Lin\'eaire}, {\bf 33} no. 3 (2016), 699--766.
\bibitem{Simon} B. Simon.  \emph{The $P(\varphi)_2$ Euclidean (quantum) field theory}, Princeton Series in Physics. Princeton University Press, Princeton, N.J., 1974. xx+392 pp.
\bibitem{ST} H. Smith and D. Tataru. Sharp local well-posedness results for the nonlinear wave equation, \emph{Ann. of Math.} (2) {\bf 162} no. 1 (2005), 291--366.
\bibitem{dS} A-S. de Suzzoni. Invariant measures for the cubic eave equation on the unit ball in $\mathbb{R}^3$, \emph{Dyn. Partial Differ. Equ.} {\bf 8} no. 2 (2011),  127--148.
\bibitem{Sy} M. Sy. Invariant measure and long time behavior of regular solutions of the Benjamin-Ono equation, \emph{Anal. PDE} {\bf 11} no. 8 (2018), 1841--1879.
\bibitem{Thomann} L. Thomann. Invariant Gibbs measures for dispersive PDEs, Lecture Notes from \emph{Hamiltonian dynamics, PDEs and waves on the Amalfi coast}, (2016)  \url{http://www.iecl.univ-lorraine.fr/~Laurent.Thomann/Gibbs_Thomann_Maiori.pdf}.
\bibitem{TTz} L. Thomann and N. Tzvetkov.  Gibbs measure for the periodic derivative nonlinear Schr\"odinger equation, \emph{Nonlinearity}, {\bf 23} (2010), 2771--2791.
 \bibitem{Tz0} N. Tzvetkov. Invariant measures for the Nonlinear Schr\"odinger equation on the disc. \emph{Dyn. Partial Differ. Equ.} {\bf 3} (2006), 111--160.
\bibitem{Tz} N. Tzvetkov.  Construction of a Gibbs measure associated to the periodic Benjamin-Ono equation,
\emph{Probab. Theory Related Fields} {\bf 146} (2010), 481--514.
\bibitem{WangYue} W. Wang and H. Yue. Almost sure existence of global weak solutions to the Boussinesq equations. \emph{Dynamics of Partial Differential Equations} {\bf 17} no. 2 (2020), 165--183.
\bibitem{HYue} H. Yue. Almost sure well-posedness for the cubic nonlinear Schr\"odinger equation in the super-critical regime on $\mathbb{T}^d$. \emph{Stochastics and Partial Differential Equations: Analysis and Computations} {\bf 9} (2021), 243--294.
\bibitem{ZhF} T. Zhang and D. Fang. Random dada Cauchy theory for the generalized incompressible Navier-Stokes equations. \emph{J. Math. Fluid. Mech.} {\bf 14} no. 2 (2012), 311--324.
\bibitem{Zhi94} P. E. Zhidkov. An invariant measure for a nonlinear wave equation. \emph{Nonlinear Anal.} {\bf 22} (1994), 319--325.

\end{thebibliography}
\end{document}